\documentclass[11pt]{article}
\usepackage{amssymb,amsmath,accents}
\usepackage{amscd}
\usepackage{amsfonts,amsthm,mathrsfs}
\usepackage{setspace} 
\usepackage{graphics}
\usepackage{latexsym}
\usepackage{color}
\usepackage{authblk}

\newtheorem{theorem}{Theorem}
\newtheorem{lemma}[theorem]{Lemma}
\newtheorem{proposition}[theorem]{Proposition}
\newtheorem{corollary}[theorem]{Corollary}

\theoremstyle{definition}

\newtheorem{remark}[theorem]{Remark}

\title{\textbf{Anomalous smoothing effect on the incompressible Navier-Stokes-Fourier limit from Boltzmann with periodic velocity}}

\author[1]{Zhongyang Gu\thanks{zgu@ms.u-tokyo.ac.jp}}  
\author[2]{Xin Hu\thanks{k99poaaron2@gmail.com}}
\author[3]{Tsuyoshi Yoneda\thanks{t.yoneda@r.hit-u.ac.jp}}
\affil[1]{Graduate School of Mathematical Sciences, The University of Tokyo, 3-8-1 Komaba, Meguro-ku, 153-8914, Tokyo, Japan}
\affil[2]{School of Mathematics and Statistics, Wuhan University, Wuchang District, Wuhan 430070, P.\ R.\ China}
\affil[3]{Graduate School of Economics, Hitotsubashi University, 2-1 Naka, Kunitachi, 186-8601, Tokyo, Japan}

\date{}

\begin{document}
\maketitle
\begin{abstract}
Adding some nontrivial terms composed from a microstructure, we prove the existence of a global-in-time weak solution, whose enstrophy is bounded for all the time, to an incompressible 3D Navier-Stokes-Fourier type system for arbitrary initial data.
It cannot be expected to directly derive the energy inequality for this new system of equations.
The main idea is to employ the hydrodynamic limit from the Boltzmann equation with periodic velocity and a specially designed collision operator.
\end{abstract}

\begin{center}
\end{center}

\section{Introduction} 
\label{sec:intro}

In this paper, we always consider space dimension three.
The hydrodynamic limit from the Boltzmann equation attracts tremendous interests in modern research of fluid mechanics since the work of Bardos, Golse and Levermore \cite{BGL91}, \cite{BGL93}, where they derived Leray solutions to the incompressible Navier-Stokes equations from DiPerna-Lions' renormalized solutions of the Boltzmann equation with Grad's cutoff kernel \cite{DL}. 
It is now well-known that by considering a solution in the form of a fluctuation near the Maxwellian, i.e., $f_\varepsilon = \mu(1 + \varepsilon g_\varepsilon)$ with
\[
\mu(v) := (2 \pi)^{-\frac{3}{2}} \mathrm{exp}\left( - \frac{|v|^2}{2} \right), \quad v \in \mathbf{R}^3,
\]
the incompressible Navier-Stokes equations can be derived as the hydrodynamic limit from the Boltzmann equation with the Navier-Stokes type scaling
\begin{equation} \label{Boltzmann}
\varepsilon^2 \partial_t f_\varepsilon + \varepsilon v \cdot \nabla_x f_\varepsilon = C(f_\varepsilon), \quad f_\varepsilon \bigm|_{t=0} = f_{\varepsilon,0}
\end{equation}
where $\varepsilon > 0$ denotes the Knudsen number that represents the ratio of the mean free path to the macroscopic length scale, $f_\varepsilon(t,x,v)$ and $g_\varepsilon(t,x,v)$ are density distribution functions of particles having position $x \in \mathbf{R}^3$ with velocity $v \in \mathbf{R}^3$ at time $t \geq 0$ and $C(f_\varepsilon)$ is the collision operator which characterizes particle collisions. 
For the classical Boltzmann equation 
\begin{equation} \label{Boltzmann:C}
\partial_t f + v \cdot \nabla_x f = C(f), \quad f \bigm|_{t=0} = f_0
\end{equation}
which physically models the interaction of particles through collisions, the collision operator $C(f)$ is given by the formula
\begin{align} \label{Tru:Col}
C(f) = \int_{\mathbf{R}^3} \int_{\mathbf{S}^2} B(v - v_\ast, \sigma) \big( f(v'_\ast) f(v') - f(v_\ast) f(v) \big) \, d \sigma \, dv_\ast
\end{align}
where $v, v_\ast$ denote velocities of two particles before the collision and $v', v'_\ast$ denote their velocities after the collision. The non-negative cross section $B(z, \sigma)$, which is a function of $|z|$ and the inner product $\big\langle \frac{z}{|z|}, \sigma \big\rangle$ only, depends upon the intermolecular force or potential.
The derivation of Bardos, Golse and Levermore \cite{BGL93} was established in the time-discretized case under two assumptions bearing on the sequence of renormalized solutions.
In particular, these two assumptions do not necessarily hold for general Boltzmann equation.
The method of Bardos, Golse and Levermore \cite{BGL93} was extended to more general time-continuous case by Lions and Masmoudi \cite{LiMaI} under the same two assumptions.
Golse and Saint-Raymond \cite{GSR04} got rid of these two assumptions and established the convergence of DiPerna-Lions' renormalized solutions \cite{DL} to Leray solutions in the case for cutoff Maxwellian collision operator and later on, to the case for hard cutoff potentials \cite{GSR09}. 
The convergence for the case of soft potentials was established by Levermore and Masmoudi \cite{LM}. 
Furthermore, Arsenio \cite{Arse} considered this problem in the case of non-cutoff potentials.

As a simplified model of the Boltzmann equation (\ref{Boltzmann}), one can consider a different collision operator
\[
C(f_\varepsilon) = \frac{1}{\nu} \big( f_{\mathrm{eq},\varepsilon} - f_\varepsilon \big) \quad \text{with} \quad f_{\mathrm{eq},\varepsilon}(t,x,v) := \frac{R_\varepsilon(t,x)}{\big( 2 \pi T_\varepsilon(t,x) \big)^{3/2}} \, \mathrm{exp} \left( - \frac{\big| v - U_\varepsilon(t,x) \big|^2}{2 T_\varepsilon(t,x)} \right)
\]
where
\[
R_\varepsilon = \int_{\mathbf{R}^3} f_\varepsilon \, dv, \quad R_\varepsilon U_\varepsilon = \int_{\mathbf{R}^3} v f_\varepsilon \, dv, \quad R_\varepsilon U_\varepsilon^2 + 3 R_\varepsilon T_\varepsilon = \int_{\mathbf{R}^3} |v|^2 f_\varepsilon \, dv.
\]
This is so-called the BGK Boltzmann model, which is usually considered when doing numerical simulations to study the kinetic theory of particles, see e.g. \cite{Moha}.
Apart from that, it can also be used to do Large-eddy simulations to model turbulence \cite{KPS}. 
More significantly, the BGK Boltzmann model contains some basic properties of hydrodynamics which are not guaranteed by the theory of the Boltzmann equation with classical collision operator (\ref{Tru:Col}), i.e., the local conservation of momentum.
Saint-Raymond \cite{SaRay} studied the hydrodynamic limit from the BGK Boltzmann model and established that if we analogously consider solutions to the BGK Boltzmann model of the form $f_\varepsilon = \mu(1 + \varepsilon g_\varepsilon)$ where $\mu$ denotes the Maxwellian, then Leray solutions to the Navier-Stokes equations can be obtained as the scaling limit.

In order to obtain hydrodynamic limits with better regularity than Leray solutions, Jiang, Xu and Zhao \cite{JXZ} constructed a global energy estimate, which holds when the initial data is sufficiently small, that controls the $H^N( \mathbf{R}_x^3; L^2(\mathbf{R}_v^3) )$ norm of  solutions to the classical Boltzmann equation in cases for both non-cutoff and Grad's angular cutoff collision operators. 
Using this global energy estimate, the existence of a global-in-time solution to the Boltzmann equation in the space $L^\infty( [0,\infty); H^N( \mathbf{R}_x^3; L^2(\mathbf{R}_v^3) ) )$ could be established for sufficiently small initial data for $N \geq 3$. 
Then, by taking the Knudsen number $\varepsilon \to 0$, it can be deduced that the incompressible Naiver-Stokes-Fourier equations admit a global classical solution in the space $C( [0,\infty); H^{N-1}(\mathbf{R}_x^3) ) \cap L^\infty( [0,\infty); H^N(\mathbf{R}_x^3) )$ for small initial data.

In this paper, we introduce a new collision operator
\begin{align} \label{Col:ano}
C(f_\varepsilon) = - \frac{1}{\nu_\ast} \mathcal{L} (f_\varepsilon) + \frac{\varepsilon \kappa}{\nu_\ast} \mathcal{L}(f_\varepsilon^2) - \frac{\varepsilon^2 \kappa^2}{\nu_\ast} f_\varepsilon^3
\end{align}
and consider the Boltzmann equation (\ref{Boltzmann}) with collision operator (\ref{Col:ano}) in torus $\mathbf{T}^3 = \mathbf{R}^3 / \mathbf{Z}^3$ for position variable $x$ and in torus $\Omega := [-1/2,1/2]^3$ for velocity variable $v$, i.e., for $\widetilde{v} = v + \beta$ with $v \in \Omega$ and $\beta = (\beta_1, \beta_2, \beta_3) \in \mathbf{Z}^3$, we treat $\widetilde{v} \cdot \big( \nabla_x f_\varepsilon \big) (x, \widetilde{v}, t)$ in Boltzmann equation (\ref{Boltzmann}) as $v \cdot \big( \nabla_x f_\varepsilon \big) (x,v,t)$.
Coefficients $\nu_\ast, \kappa > 0$ are fixed constants which will be used later to balance the coefficients of the Navier-Stokes equations in the hydrodynamic limit.
We call the Boltzmann equation (\ref{Boltzmann}) with collision operator (\ref{Col:ano}) by the Boltzmann equation with anomalous smoothing effect.
To explain the reason why we consider operator (\ref{Col:ano}) in such a form, let us recall that if one considers a solution of the form $f = \mu + \sqrt{\mu} g$ to the classical Boltzmann equation (\ref{Boltzmann:C}) with collision operator (\ref{Tru:Col}), then the collision operator (\ref{Tru:Col}) is often decomposed as
\begin{align} \label{Decomp:Col}
C(f) = - L(g) + \Gamma(g,g)
\end{align}
with the perturbation $g$ satisfying
\begin{align*}
\partial_t g + v \cdot \nabla_x g + L(g) = \Gamma(g,g), \quad g \bigm|_{t=0} = g_0
\end{align*}
and $g_0 = \mu^{-\frac{1}{2}} f_0 - \mu^{\frac{1}{2}}$.
$L$ in decomposition (\ref{Decomp:Col}) is the linearized collision operator which is positive definite for any perturbation solution $g$; see e.g. \cite[Chap. 3]{Glassey}, \cite{GuoW}. 
The operator $\mathcal{L}$ in operator (\ref{Col:ano}) is the microscopic projection which is defined analogously as in (\cite[Chap. 3]{Glassey}, \cite{GuoW}); see Section \ref{sub:GEE}.
The microscopic projection $\mathcal{L}$ in collision operator (\ref{Col:ano}) inherits the positive definiteness of the linearized collision operator $L$ in collision operator (\ref{Decomp:Col}).
Hence, $\mathcal{L}(f_\varepsilon)$ in collision operator (\ref{Col:ano}) is compatible with $L(g)$ in collision operator (\ref{Decomp:Col}). 
On the other hand, $\mathcal{L}(f_\varepsilon^2)$ in collision operator (\ref{Col:ano}) is an approximation to the nonlinear term $\Gamma(g,g)$ in collision operator (\ref{Decomp:Col}). 
Such approximation is inspired by the fact that in the cutoff case, $L(g^2)$ equals a constant multiple of $\Gamma(g,g)$ for all $g \in \mathrm{ker}(L)$, see e.g. \cite[Proposition 1.5]{GSR04}.
Although the first two terms of collision operator (\ref{Col:ano}) loses some properties of collision operator (\ref{Decomp:Col}), such as they do not behave locally as a fractional Laplacian (\cite{ADVW}, \cite{GS}), they contain all the properties that we need to formally derive the incompressible Navier-Stokes-Fourier equations as the hydrodynamic limit.
The cubic term of collision operator (\ref{Col:ano}) is a specially designed term which has the magic power to cancel the nonlinear effect coming from the second term of (\ref{Col:ano}) in deriving the a priori energy estimate, i.e., with the help of this cubic term in (\ref{Col:ano}), the Boltzmann equation that we work with in this paper admits a perfect energy estimate which holds globally in time.
Moreover, this global energy estimate holds up to regularity $H^1( \mathbf{T}_x^3; L^2(\Omega_v) )$ without requiring any size conditions on the initial data.
Following the idea of Jiang, Xu and Zhao \cite{JXZ}, the theme of this paper is to justify the hydrodynamic limit of the Boltzmann equation (\ref{Boltzmann}) with collision operator (\ref{Col:ano}), which further implies the existence of a global $H^1$ weak solution, without requiring any smallness condition on the initial data, to some incompressible Navier-Stokes-Fourier type system.

In order to justify the hydrodynamic limit, we firstly need the solvability of the Boltzmann equation. Instead of working directly with the Boltzmann equation (\ref{Boltzmann}) with collision operator (\ref{Col:ano}), we consider its approximate equation 
\begin{equation} \label{AEtoBA}
\begin{split}
&\varepsilon^2 \partial_t \Lambda_\varepsilon (f_\varepsilon) + \varepsilon \Lambda_\varepsilon \big( v \cdot \nabla_x \Lambda_\varepsilon (f_\varepsilon) \big) = C_{\mathrm{cut}}(f_\varepsilon), \\
&C_{\mathrm{cut}}(f_\varepsilon) := - \frac{1}{\nu_\ast} \Lambda_\varepsilon \Big( \mathcal{L}^\varepsilon \big( \Lambda_\varepsilon (f_\varepsilon) \big) \Big) + \frac{\varepsilon \kappa}{\nu_\ast} \Lambda_\varepsilon \Big( \mathcal{L}^\varepsilon \big( \Lambda_\varepsilon (f_\varepsilon)^2 \big) \Big) - \frac{\varepsilon^2 \kappa^2}{\nu_\ast} \Lambda_\varepsilon \big( \Lambda_\varepsilon (f_\varepsilon)^3 \big), \\
&f_\varepsilon \bigm|_{t=0} = \Lambda_\varepsilon (f_{\varepsilon,0}) \\
\end{split}
\end{equation}
where $\Lambda_\varepsilon$ denotes an operator that does cutoff in Fourier space and $\mathcal{L}^\varepsilon$ denotes the cutoff in Fourier space version of the microscopic projection $\mathcal{L}$. 
Here we would like to direct readers to Section \ref{sub:CF} for the precise definition of $\Lambda_\varepsilon$ and $\mathcal{L}^\varepsilon$.
For simplicity of notations, we define that
\[
\| h \|_X^2 := \| h \|_{H^1( \mathbf{T}_x^3; L^2(\Omega_v) )}^2 = \sum_{|\alpha| \leq 1} \int_{\mathbf{T}^3} \int_{\Omega} \big| \partial_x^\alpha h \big|^2 \, dv \, dx
\]
for $h \in H^1( \mathbf{T}_x^3; L^2(\Omega_v) )$ where $\partial_x^\alpha$ represents the differentiation $\partial_{x_1}^{\alpha_1} \partial_{x_2}^{\alpha_2} \partial_{x_3}^{\alpha_3}$ with $\alpha$ denoting the multi-index $\alpha = (\alpha_1, \alpha_2, \alpha_3) \in \mathbf{N}_0^3$ and $\mathbf{N}_0 := \mathbf{N} \cup \{0\}$. In addition, we set that
\[
\mathcal{E}(h) := \| h \|_X, \quad \mathcal{D}_\varepsilon(h) := \| \mathcal{L}^\varepsilon(h) \|_X.
\]
Regarding the approximate equation (\ref{AEtoBA}), we establish the following global existence result.

\begin{lemma} \label{GloSol:AE}
Let $\varepsilon > 0$. For any $f_{\varepsilon,0} \in H^1( \mathbf{T}_x^3; L^2(\Omega_v) )$, the approximate Boltzmann equation with anomalous smoothing effect (\ref{AEtoBA}) admits a unique global solution
\[
f_\varepsilon \in L^\infty( [0,\infty); H^1( \mathbf{T}_x^3; L^2(\Omega_v) ) )
\]
satisfying $f_\varepsilon = \Lambda_\varepsilon(f_\varepsilon)$ and the global energy estimate
\begin{align} \label{GloEE:ep}
\sup_{t \geq 0} \, \mathcal{E}\big( f_\varepsilon \big)^2 (t) + \frac{1}{\varepsilon^2 \nu_\ast} \int_0^\infty \mathcal{D}_\varepsilon \big( f_\varepsilon \big)^2 (s) \, ds \leq \mathcal{E}\big( f_{\varepsilon,0} \big)^2
\end{align}
with $f_\varepsilon \bigm|_{t=0} = \Lambda_\varepsilon(f_{\varepsilon, 0})$.
\end{lemma}

Since we are considering velocity $v$ in the torus $\Omega$ for the approximate equation (\ref{AEtoBA}), Lemma \ref{GloSol:AE} can be proved by the standard Picard's method. 
Since the form of the collision operator (\ref{Col:ano}) is specially designed, the global energy estimate (\ref{GloEE:ep}) can be easily derived by the argument of the traditional energy method. 
Hence, it is sufficient to establish the existence and uniqueness of a local solution to the approximate equation (\ref{AEtoBA}).
To show the existence of a local solution, we consider a sequence of functions $\{ g_{\varepsilon, j}(t) \}_{j \in \mathbf{N}_0}$ defined inductively by
\begin{align*}
g_{\varepsilon, j+1}(t) 
&:= g_{\varepsilon, 0} - \frac{1}{\varepsilon} \int_0^t \Lambda_\varepsilon \big( v \cdot \nabla_x g_{\varepsilon, j}(s) \big) \, ds - \frac{1}{\varepsilon^2 \nu_\ast} \int_0^t \Lambda_\varepsilon \Big( \mathcal{L}^\varepsilon\big( g_{\varepsilon, j}(s) \big) \Big) \, ds \\
&\ \ + \frac{\kappa}{\varepsilon \nu_\ast} \int_0^t \Lambda_\varepsilon \Big( \mathcal{L}^\varepsilon\big( g_{\varepsilon, j}(s)^2 \big) \Big) \, ds - \frac{\kappa^2}{\nu_\ast} \int_0^t \Lambda_\varepsilon \big( g_{\varepsilon, j}(s)^3 \big) \, ds
\end{align*}
for $j \geq 0$ where $g_{\varepsilon, 0} = \Lambda_\varepsilon(f_{\varepsilon, 0})$. 
One of the most crucial reason why we work with the approximate equation (\ref{AEtoBA}) instead of the original Boltzmann equation (\ref{Boltzmann}) with collision operator (\ref{Col:ano}) is because we want to make use of a special property of the cutoff operator $\Lambda_\varepsilon$. 
Since $\Lambda_\varepsilon$ does the cutoff in Fourier space, we have the Bernstein-type lemma which allows us to estimate the $L^q$ norm of $\partial_{x,v}^\alpha \Lambda_\varepsilon(h)$ by a constant multiple, where the constant depends only on $\varepsilon$ and $|\alpha|$, of the $L^p$ norm of $h$ with $p < q$; see Lemma \ref{Re:BerL} in Section \ref{sub:CF}.
As a result, for any $f_1, f_2, ..., f_n \in H^1( \mathbf{T}_x^3; L^2(\Omega_v) )$, we are able to establish a multiplication  rule regarding the $X$-norm of $\prod_{i=1}^n \Lambda_\varepsilon(f_i)$, i.e., we can estimate $\| \prod_{i=1}^n \Lambda_\varepsilon(f_i) \|_X$ by $\prod_{i=1}^n \| \Lambda_\varepsilon(f_i) \|_X$ with a constant depending on $\varepsilon$ and $n$ only; see Proposition \ref{MRX} in Section \ref{sub:SAE}.
Having this tool, we can then prove by induction that the sequence $\{ g_{\varepsilon, j}(t) \}_{j \in \mathbf{N} \cup \{0\}}$ is Cauchy in $L^\infty( [0,\infty); H^1( \mathbf{T}_x^3; L^2(\Omega_v) ) )$. 
Then, the existence of a local solution can be concluded by the contraction mapping theorem and the Banach fixed point theorem. 
The uniqueness of the local solution can be easily shown by a simple energy method argument as well due to the good form of the collision operator (\ref{Col:ano}).
This completes the proof of Lemma \ref{GloSol:AE}.
Since the $X$-norm of $\Lambda_\varepsilon\big( v \cdot \nabla_x \Lambda_\varepsilon(f_\varepsilon) \big)$ is controlled due to the fact that the $L^\infty$-norm of $v$ in $\Omega$ is bounded, we can observe that $\partial_t f_\varepsilon \in L^\infty( [0,\infty); H^1( \mathbf{T}_x^3; L^2(\Omega_v) ) )$, i.e., the unique solution to the approximate equation (\ref{AEtoBA}) is indeed a strong solution.

With respect to a global solution $f_\varepsilon$ constructed in Lemma \ref{GloSol:AE}, we set 
\[
\rho^\varepsilon := \int_\Omega f_\varepsilon \, dv, \quad u^\varepsilon := \int_\Omega \mathrm{e}_1^\varepsilon f_\varepsilon \, dv, \quad \theta^\varepsilon := \int_\Omega \mathrm{e}_2^\varepsilon f_\varepsilon \, dv
\]
where $\mathrm{e}_1^\varepsilon$ and $\mathrm{e}_2^\varepsilon$ are defined by (\ref{Cf:e012}) in Section \ref{sub:CF}.
Simply speaking, $\mathrm{e}_1^\varepsilon$ is the cutoff in Fourier space version of $2 \sqrt{3} v$ and $\mathrm{e}_2^\varepsilon$ is the cutoff in Fourier space version of $6 \sqrt{5} (|v|^2 - \frac{1}{4})$.
By taking the Knudsen number $\varepsilon \to 0$ for the approximate equation (\ref{AEtoBA}) with $\kappa = \sqrt{3}$ and $\nu = \frac{\nu_\ast}{12}$, the main convergence result of this paper reads as follows.

\begin{theorem} \label{MCT}
Let $0 < \varepsilon < 1$. For any $(\rho_0, u_0, \theta_0) \in H^1(\mathbf{T}_x^3)$, let 
\[
f_{\varepsilon,0} = \Lambda_\varepsilon(\rho_0) + \Lambda_\varepsilon(2 \sqrt{3} v \cdot u_0) + \Lambda_\varepsilon\left( 6 \sqrt{5} \bigg( |v|^2 - \frac{1}{4} \bigg) \theta_0 \right)
\]
where $\Lambda_\varepsilon$ is a cutoff in Fourier space operator that will be defined in Section \ref{sub:CF}.
Let $f_\varepsilon$ be the unique solution to the approximate Boltzmann equation with anomalous smoothing effect (\ref{AEtoBA}).
Then, there exist
\[
f \in L^\infty( [0,\infty); H^1(\mathbf{T}_x^3; L^2(\Omega_v) ) ) \quad \text{and} \quad (\rho, u, \theta) \in L^\infty( [0,\infty); H^1(\mathbf{T}_x^3) )
\]
such that $f_\varepsilon \overset{\ast}{\rightharpoonup} f$ in the weak-$\ast$ topology
\begin{align*}
\sigma\big( L^\infty( [0,\infty); H^1(\mathbf{T}_x^3; L^2(\Omega_v) ) ), L^1( [0,\infty); H^{-1}(\mathbf{T}_x^3; L^2(\Omega_v) ) ) \big)
\end{align*}
and $(\rho^\varepsilon, u^\varepsilon, \theta^\varepsilon) \overset{\ast}{\rightharpoonup} (\rho, u, \theta)$ in the weak-$\ast$ topology
\[
\sigma\big( L^\infty( [0,\infty); H^1(\mathbf{T}_x^3) ), L^1( [0,\infty); H^{-1}(\mathbf{T}_x^3) ) \big)
\]
as $\varepsilon \to 0$.
Moreover, there exist $\mathcal{F}(t,x), \mathcal{G}(t,x) \in L^\infty( [0,\infty); L^2(\mathbf{T}_x^3) )$ and $\mathcal{H}(t,x), \mathcal{K}(t,x) \in L^\infty( [0,\infty); L^{\frac{3}{2}}(\mathbf{T}_x^3) )$ satisfying the global energy inequality
\begin{equation*}
 \begin{aligned}
 \| \mathcal{F} \|_{L^\infty( [0,\infty); L^2(\mathbf{T}_x^3) )} + \| \mathcal{G} \|_{L^\infty( [0,\infty); L^2(\mathbf{T}_x^3) )} &\lesssim& &\| \rho_0 \|_{H^1(\mathbf{T}_x^3)}^3 + \| u_0 \|_{H^1(\mathbf{T}_x^3)}^3 + \| \theta_0 \|_{H^1(\mathbf{T}_x^3)}^3,& \\
  \| \mathcal{H} \|_{L^\infty( [0,\infty); L^{\frac{3}{2}}(\mathbf{T}_x^3) )} + \| \mathcal{K} \|_{L^\infty( [0,\infty); L^{\frac{3}{2}}(\mathbf{T}_x^3) )} &\lesssim& &\| \rho_0 \|_{H^1(\mathbf{T}_x^3)}^2 + \| u_0 \|_{H^1(\mathbf{T}_x^3)}^2 + \| \theta_0 \|_{H^1(\mathbf{T}_x^3)}^2&
 \end{aligned}
\end{equation*}
such that 
\begin{align*}
(u, \widetilde{\theta}) \in L^\infty( [0,\infty); H^1(\mathbf{T}_x^3) ) \cap C( [0,\infty); L^2(\mathbf{T}_x^3) ), \quad \widetilde{\theta} := \theta - \frac{2 \sqrt{5}}{5} \rho
\end{align*}
is a global weak solution to the incompressible Navier-Stokes-Fourier type system
\begin{equation} \label{NSF:limit}
\left\{
 \begin{aligned}
 \partial_t u - \nu \Delta_x u + u \cdot \nabla_x u + \nabla_x p &=& &\frac{1}{\nu} \mathcal{F} + \mathcal{H} + \nu \mathcal{J},& \\ 
 \nabla_x \cdot u &=& &0,& \\
 \partial_t \widetilde{\theta} - \frac{291}{133} \nu \Delta_x \widetilde{\theta} + \frac{97}{35} u \cdot \nabla_x \widetilde{\theta} &=& &\frac{1}{\nu} \mathcal{G} + \mathcal{K},& \\
 u \bigm|_{t=0} &=& &\mathbb{P}(u_0),& \\
 \widetilde{\theta} \bigm|_{t=0} &=& &\theta_0 - \frac{2 \sqrt{5}}{5} \rho_0&
 \end{aligned}
\right.
\end{equation} 
where $\mathcal{J} := - \frac{6}{5} (\partial_{x_1}^2 u_1^\varepsilon, \partial_{x_2}^2 u_2^\varepsilon, \partial_{x_3}^2 u_3^\varepsilon)$ and $\mathbb{P}$ denotes the Helmholtz projection for the space $L^2(\mathbf{T}_x^3)$.
Furthermore, $(\rho, \widetilde{\theta})$ satisfy the Boussinesq relation
\begin{align} \label{Boussi:MCT}
\nabla_x \big( 2 \sqrt{5} \widetilde{\theta} + 19 \rho \big) = 0
\end{align}
and $(\rho, u, \widetilde{\theta})$ satisfy the global energy inequality
\begin{align*}
&\| \rho \|_{L^\infty( [0,\infty); H^1(\mathbf{T}_x^3) )} + \| u \|_{L^\infty( [0,\infty); H^1(\mathbf{T}_x^3) )} + \| \widetilde{\theta} \|_{L^\infty( [0,\infty); H^1(\mathbf{T}_x^3) )} \\
&\ \ \lesssim \| \rho_0 \|_{H^1(\mathbf{T}_x^3)} + \| u_0 \|_{H^1(\mathbf{T}_x^3)} + \| \theta_0 \|_{H^1(\mathbf{T}_x^3)}.
\end{align*}
\end{theorem}

To summarize the main contribution of this paper in simple words, we establish that by adding to the $3$D incompressible Navier-Stokes-Fourier equations a nontrivial term of second order differentiations and some forcing terms that are well-controlled in the $L^2$ and $L^{\frac{3}{2}}$ sense, it becomes possible to eliminate the smallness condition required on the initial data in order to obtain the global existence of an $H^1$ weak solution.
In the title, the meaning of ``anomalous smoothing effect'' is different from the standard notion of ``smoothing effect'' which describes some gain in regularity.
We name ``anomalous smoothing effect'' to describe the capability of concluding of the existence of a global $H^1$ solution without requiring the initial data to be small.

The idea of proving Theorem \ref{MCT} is somehow standard.
Briefly speaking, with respect to the unique global solution $f_\varepsilon$, we firstly rewrite the approximate equation (\ref{AEtoBA}) in terms of $\rho^\varepsilon$, $u^\varepsilon$ and $\theta^\varepsilon$ and then use the global energy estimate (\ref{GloEE:ep}) to prove the convergence.
For the convenience of readers, we would like to elaborate more about the formal derivation of the incompressible Navier-Stokes-Fourier type system (\ref{NSF:limit}).
By taking the inner product of our approximate equation (\ref{AEtoBA}) with $1$, $\mathrm{e}_1^\varepsilon$ and $\mathrm{e}_2^\varepsilon$ in the sense of $L^2(\Omega_v)$, we can obtain that
\begin{eqnarray} \label{ConsLaw}
\left\{
\begin{array}{lcl}
\varepsilon \partial_t \rho^\varepsilon + \frac{1}{c_1^\varepsilon} \nabla_{x} \cdot u^\varepsilon = - \frac{\varepsilon \kappa^2}{\nu_\ast} \langle f_\varepsilon^3 \rangle, \\
\varepsilon \partial_t u^\varepsilon + c_1^\varepsilon \nabla_{x} \cdot \big\langle (\Lambda_\varepsilon(v) \otimes v) f_{\varepsilon} \big\rangle= - \frac{\varepsilon \kappa^2}{\nu_\ast} \langle \mathrm{e}_1^\varepsilon f_\varepsilon^3 \rangle, \\
\varepsilon \partial_t \theta^\varepsilon + \nabla_x \cdot \big\langle \mathrm{e}_2^\varepsilon v f_\varepsilon \big\rangle = - \frac{\varepsilon \kappa^2}{\nu_\ast} \langle \mathrm{e}_2^\varepsilon f_\varepsilon^3 \rangle.
\end{array}
\right.
\end{eqnarray}
Here, the notation $\langle h \rangle$ represents the integral of function $h$ in $\Omega$ with respect to $v$.
The divergence free condition of system (\ref{NSF:limit}) can be derived from the first equation of system (\ref{ConsLaw}) easily.
The first equation of system (\ref{NSF:limit}) results from the second equation of system (\ref{ConsLaw}).
By introducing a matrix $A_\varepsilon$, which is the cutoff in Fourier space version of $A := v \otimes v - \frac{|v|^2}{3} I$, we can rewrite 
\[
\nabla_x \cdot \big\langle (\Lambda_\varepsilon(v) \otimes v) f_\varepsilon \big\rangle = \nabla_x \cdot \langle A_\varepsilon f_\varepsilon \rangle + c_{1,\varepsilon} \nabla_x \theta^\varepsilon + c_{2,\varepsilon} \nabla_x \rho^\varepsilon
\]
with some constants $c_{1,\varepsilon}, c_{2,\varepsilon}$ that converge as $\varepsilon \to 0$.
Here we would like to remark that the matrix $A_\varepsilon$ is constructed so that it satisfies $A_\varepsilon \in \mathrm{ker}^\perp(\mathcal{L}^\varepsilon)$ where $\mathrm{ker}^\perp(\mathcal{L}^\varepsilon)$ denotes the annihilator of the kernel of $\mathcal{L}^\varepsilon$.
Due to this fact, $\mathcal{L}^\varepsilon(f_\varepsilon)$ can be rewritten in the form of 
\[
\mathcal{L}^\varepsilon(f_\varepsilon) = - \varepsilon^2 \nu_\ast \partial_t f_\varepsilon - \varepsilon \nu_\ast \Lambda_\varepsilon\big( v \cdot \nabla_x f_\varepsilon \big) + \varepsilon \kappa \mathcal{L}^\varepsilon\big( \Lambda_\varepsilon(f_\varepsilon^2) \big) - \varepsilon^2 \kappa^2 \Lambda_\varepsilon(f_\varepsilon^3).
\]
As a result, we can then deduce that
\begin{equation} \label{Re:Af}
\begin{split}
\frac{1}{\varepsilon} \nabla_x \cdot \langle A_\varepsilon f_\varepsilon \rangle &= - \varepsilon \nu_\ast \nabla_x \cdot \partial_t \langle A_\varepsilon f_\varepsilon \rangle - \nu_\ast \nabla_x \cdot \langle A_\varepsilon \Lambda_\varepsilon (v \cdot \nabla_x f_\varepsilon) \rangle + \kappa \nabla_x \cdot \big\langle A_\varepsilon \mathcal{L}^\varepsilon\big( \Lambda_\varepsilon(f_\varepsilon^2) \big) \big\rangle \\
&\ \ - \varepsilon \kappa^2 \nabla_x \cdot \langle A_\varepsilon \Lambda_\varepsilon(f_\varepsilon^3) \rangle.
\end{split}
\end{equation}
Due to the global energy estimate (\ref{GloEE:ep}), the first and the fourth term on the right hand side of equation (\ref{Re:Af}) can be shown to have order $\mathcal{O}(\varepsilon^{1-\tau_1})$ for some $\tau_1 \in (0,1)$. 
Hence, the convergence behavior of $\varepsilon^{-1} \nabla_x \cdot \langle A_\varepsilon f_\varepsilon \rangle$ is governed by the second and the third term on the right hand side of equation (\ref{Re:Af}) when taking the limit $\varepsilon \to 0$.
Specifically speaking, for the first equation of system (\ref{NSF:limit}), the diffusion term $\Delta_x u^\varepsilon$ is derived from the rewriting of $\nabla_x \cdot \langle A_\varepsilon \Lambda_\varepsilon( v \cdot \nabla_x f_\varepsilon) \rangle$ whereas the transport term $u^\varepsilon \cdot \nabla_x u^\varepsilon$ is derived from the rewritten of $\nabla_x \cdot \big\langle A_\varepsilon \mathcal{L}^\varepsilon\big( \Lambda_\varepsilon(f_\varepsilon^2) \big) \big\rangle$.
At this point, we can formally rewrite the second equation of system (\ref{ConsLaw}) as
\begin{align} \label{Form:NS}
\partial_t u^\varepsilon - \nu \Delta_x u^\varepsilon + u^\varepsilon \cdot \nabla_x u^\varepsilon + \frac{1}{\varepsilon} (c_1 \nabla_x \theta^\varepsilon + c_2 \nabla_x \rho^\varepsilon) + c_3 \nabla_x |u^\varepsilon|^2 = - \frac{c_4}{\nu} F^\varepsilon + H^\varepsilon + \mathcal{O}_1^\varepsilon,
\end{align}
where $F^\varepsilon$ is a forcing term resulting from $\langle \mathrm{e}_1^\varepsilon f_\varepsilon^3 \rangle$ and $\mathcal{O}_1^\varepsilon$ is a remainder term that converges to zero in the sense of distributions as $\varepsilon \to 0$.
The Boussinesq relation (\ref{Boussi:MCT}) can be obtained by multiplying $\varepsilon$ to both sides of equation (\ref{Form:NS}) and then take the limit $\varepsilon \to 0$.
In order to get rid of the terms $\nabla_x \rho^\varepsilon$ and $\nabla_x \theta^\varepsilon$ whose coefficient are constant multiples of $\varepsilon^{-1}$, we apply the Helmholtz projection $\mathbb{P}$ of $L^2(\mathbf{T}_x^3)$ to equation (\ref{Form:NS}).
Then by taking the limit $\varepsilon \to 0$, we obtain the first equation in system (\ref{NSF:limit}) as the hydrodynamic limit.

By analogous derivation of the first equation of system (\ref{NSF:limit}), we can derive the third equation of system (\ref{NSF:limit}) from the third equation of system (\ref{ConsLaw}).
Different from the usage of matrix $A_\varepsilon$, in this case by considering vector $B_\varepsilon \in \mathrm{ker}^\perp(\mathcal{L}^\varepsilon)$, which is the cutoff in Fourier space version of $B := v (|v|^2 - \frac{19}{60})$, we can deduce that
\[
\frac{1}{\varepsilon} \nabla_x \cdot \big\langle \mathrm{e}_2^\varepsilon v f_\varepsilon \big\rangle = \frac{c_{3,\varepsilon}}{\varepsilon} \nabla_x \cdot \langle B_\varepsilon f_\varepsilon \rangle + \frac{c_{4,\varepsilon}}{\varepsilon} \nabla_x \cdot u^\varepsilon
\]
and
\begin{equation} \label{Intro:Bf}
\begin{split}
\frac{1}{\varepsilon} \nabla_x \cdot \langle B_\varepsilon f_\varepsilon \rangle &= - \varepsilon \nu_\ast \nabla_x \cdot \partial_t \langle B_\varepsilon f_\varepsilon \rangle - \nu_\ast \nabla_x \cdot \big\langle B_\varepsilon \Lambda_\varepsilon\big( v \cdot \nabla_x f_\varepsilon \big) \big\rangle + \kappa \nabla_x \cdot \big\langle B_\varepsilon \mathcal{L}^\varepsilon\big( \Lambda_\varepsilon(f_\varepsilon^2) \big) \big\rangle \\
&\ \ - \varepsilon \kappa^2 \nabla_x \cdot \big\langle B_\varepsilon \Lambda_\varepsilon(f_\varepsilon^3) \big\rangle.
\end{split}
\end{equation}
Similarly, since the first and the fourth term on the right hand side of equation (\ref{Intro:Bf}) have order $\mathcal{O}(\varepsilon^{1 - \tau_2})$ for some $\tau_2 \in (0,1)$, the convergence behavior of $\varepsilon^{-1} \nabla_x \cdot \big\langle B_\varepsilon f_\varepsilon \big\rangle$ is governed by the middle two terms on the right hand side of equation (\ref{Intro:Bf}).
Specifically speaking, in the third equation of system (\ref{NSF:limit}), the term $\Delta_x \theta^\varepsilon$ is derived from $\nabla_x \cdot \langle B_\varepsilon \Lambda_\varepsilon( v \cdot \nabla_x f_\varepsilon ) \rangle$ whereas the term $\operatorname{div} (u^\varepsilon \theta^\varepsilon)$ is derived from $\nabla_x \cdot \big\langle B_\varepsilon \mathcal{L}^\varepsilon\big( \Lambda_\varepsilon(f_\varepsilon^2) \big) \big\rangle$.
Hence, the third equation of system (\ref{ConsLaw}) can be rewritten as 
\begin{align} \label{Form:theta}
\partial_t \theta^\varepsilon + \frac{c_5}{\varepsilon} \nabla_x \cdot u^\varepsilon - c_6 \nu \Delta_x \theta^\varepsilon + c_7 \operatorname{div} (u^\varepsilon \theta^\varepsilon) = - \frac{c_8}{\nu} G_\varepsilon + \mathcal{O}_{2,\varepsilon}, 
\end{align}
where $G_\varepsilon$ is a forcing term resulting from $\langle \mathrm{e}_2^\varepsilon f_\varepsilon^3 \rangle$ and $\mathcal{O}_{2,\varepsilon}$ is a remainder term that converges to zero in the sense of distributions as $\varepsilon \to 0$.
To get rid of the term $\varepsilon^{-1} \nabla_x \cdot u^\varepsilon$ in equation (\ref{Form:theta}), we subtract a constant multiple of the first equation of system (\ref{ConsLaw}) from equation (\ref{Form:theta}).
Then by taking the limit $\varepsilon \to 0$, we obtain the third equation of system (\ref{NSF:limit}).

Before we end this introduction, we would like to explain more about our motivation to do cutoff in Fourier space to the Boltzmann equation (\ref{Boltzmann}) with collision operator (\ref{Col:ano}).
In the case for working with Boltzmann equation (\ref{Boltzmann}) with the classical collision operator (\ref{Tru:Col}), we have trilinear estimates to control inner products of the nonlinear part of collision operator (\ref{Tru:Col}) with another function in cases for both cutoff (\cite{GuoW}, \cite{Guo03}) and non-cutoff collision operators (\cite{AMUXY}, \cite{GS}). 
However, in our case it is not only hard to estimate the $X$-norm of the last two terms of collision operator (\ref{Col:ano}), but also hard to estimate their inner products with another function.
Doing cutoff in Fourier space to the whole equation allows us to use the Bernstein-type lemma (Lemma \ref{Re:BerL}) to estimate the $X$-norm of products of several functions (Proposition \ref{MRX}).
Since we are considering the Boltzmann equation with periodic velocity, the existence of a local solution to the approximate equation (\ref{AEtoBA}) can therefore be established by the classical Picard's method.
The bad news of applying the Bernstein-type lemma to estimate the $X$-norm of the product of multiple functions is that we will have $\varepsilon^{-\tau}$ coming out as a coefficient for the estimate with some $\tau>0$.
Fortunately, in constructing a local solution to the approximate equation (\ref{AEtoBA}) we are considering fixed $\varepsilon>0$, and in proving the convergence as $\varepsilon \to 0$, whenever we apply the multiplication rule to estimate the $X$-norm of products of multiple functions, either the product has a coefficient $\varepsilon$ or one of the function in the product is $\mathcal{L}^\varepsilon(f_\varepsilon)$ whose $L^2_T X$-norm is a constant multiple of $\varepsilon$ according to the global energy estimate (\ref{GloEE:ep}) for any $T \in [0,\infty]$.
As a result, by considering the cutoff operator $\Lambda_\varepsilon$ in a suitable form, we can always make sure that when we apply the multiplication rule, the final coefficient can be adjusted to $\varepsilon^{1 - \tau_\ast}$ with some $0<\tau_\ast<1$. 
This guarantees all convergences to zero, either in the weak sense or the strong sense, that are needed.
On the other hand, working with the Boltzmann equation with periodic velocities causes the presence of rubbish terms $\mathcal{H}$ and $\mathcal{J}$ in the incompressible Navier-Stokes-Fourier type limit (\ref{NSF:limit}) due to the fact that 
\[
\int_\Omega v_i^4 \, dv \neq \int_\Omega v_i^2 v_j^2 \, dv \quad \text{for} \quad 1 \leq i,j \leq 3 \quad \text{with} \quad i \neq j.
\]

This paper is organized as follows. Chapter \ref{sec:GSaBE} is devoted to the global solvability of the approximate equation (\ref{AEtoBA}). 
In Section \ref{sub:GEE}, we show the global energy estimate which holds for both the Boltzmann equation with anomalous smoothing effect and its approximate equation. 
In Section \ref{sub:CF}, we define the cutoff in Fourier space operator $\Lambda_\varepsilon$ and the cutoff in Fourier space version $\mathcal{L}^\varepsilon$ of microscopic projection $\mathcal{L}$. Furthermore, we recall some properties of our cutoff operator $\Lambda_\varepsilon$ which are crucial for estimates in this paper, especially the Bernstein-type lemma.
In Section \ref{sub:SAE}, we establish the multiplication rule regarding the $X$-norm and give the proof to Lemma \ref{GloSol:AE}.
Chapter \ref{sec:ConNS} is devoted to the formal derivation from the approximate equation (\ref{AEtoBA}) to a incompressible Navier-Stokes-Fourier type system.
In Section \ref{sub:BS}, we derive the divergence free condition and prove its convergence to zero in the sense of distributions in the limit $\varepsilon \to 0$. 
In Section \ref{sub:ReMome}, we define the matrix $A_\varepsilon$, i.e., the cutoff in Fourier space version of matrix $A$, and investigate its convergence behavior as $\varepsilon \to 0$.
We then use matrix $A_\varepsilon$ to rewrite $\varepsilon^{-1} \nabla_x \cdot \big\langle (\Lambda_\varepsilon(v) \otimes v) f_\varepsilon \big\rangle$.
In Section \ref{Sub:Diff}, we derive the diffusion term $\Delta_x u^\varepsilon$ of the Navier-Stokes equations.
In Section \ref{Sub:Nonli}, we derive the transport term $u^\varepsilon \cdot \nabla_x u^\varepsilon$ of the Navier-Stokes equations.
In Section \ref{Sub:thetaeq}, we derive the equation of $\theta^\varepsilon$ in the incompressible Navier-Stokes-Fourier type system.
Chapter \ref{Sec:ConNSF} is devoted to the proof of our main convergence theorem.
In Section \ref{Sub:ConFor}, we investigate the convergence behavior for the zeroth, first and second moment of $f_\varepsilon$ in the limit $\varepsilon \to 0$.
In Section \ref{Sub:ConNSF}, we apply the Helmholtz projection to the main equation and establish some strong convergence result regarding the divergence free part of $u^\varepsilon$.
In Section \ref{sub:cvNSF}, we conclude the convergence to the incompressible Navier-Stokes-Fourier type system (\ref{NSF:limit}).

Throughout this paper, the notation $A \lesssim B$ will mean that there exists a constant $c$, which is independent of $\varepsilon$ and $\nu_\ast$, such that $A \leq c B$.

\section{Global solvability of the approximate equation} 
\label{sec:GSaBE}

The key idea here is to consider the standard energy method.

\subsection{Global energy estimate for the Boltzmann equation with anomalous smoothing effect} 
\label{sub:GEE}

Firstly, we would like to define the microscopic projection $\mathcal{L}$ in detail.
Let $\mathrm{e}_0 := 1$,
\[
\mathrm{e}_{1,i} := 2 \sqrt{3} v_i \quad \text{and} \quad \mathrm{e}_{2,i} :=  6 \sqrt{5} \left( v_i^2 - \frac{1}{12} \right) \quad \text{for} \quad i = 1,2,3
\] 
be $3$-dimensional Legendre polynomials in $\Omega$ up to power $2$.
We further set
\[
\mathrm{e}_1 := \big( \mathrm{e}_{1,1}, \mathrm{e}_{1,2}, \mathrm{e}_{1,3} \big) \quad \text{and} \quad \mathrm{e}_2 := \sum_{i=1}^3 \mathrm{e}_{2,i} = 6 \sqrt{5} \left( |v|^2 - \frac{1}{4} \right).
\]
It can be easily observed that the set $\{ \mathrm{e}_0, \mathrm{e}_{1,i}, \mathrm{e}_2 \}_{1 \leq i \leq 3}$ is orthonormal, i.e.,
\[
\int_\Omega \mathrm{e}_i \mathrm{e}_j \, dv = \delta_{ij}, \quad \int_\Omega \mathrm{e}_i \mathrm{e}_{1,k} \, dv = 0, \quad \int_\Omega \mathrm{e}_{1,k} \mathrm{e}_{1,\ell} \, dv = \delta_{k \ell} \quad \forall \; i,j \in \{0,2\}; \; k, \ell \in \{1,2,3\}.
\]
For $h \in L^2( \mathbf{T}_x^3; L^2(\Omega_v) )$, we define the macroscopic projection $\mathcal{P}$ acting on $h$ to be
\begin{align} \label{PPv}
\mathcal{P} (h) := \rho \mathrm{e}_0 + u \cdot \mathrm{e}_1 + \theta \mathrm{e}_2
\end{align}
where
\[
\rho := \int_\Omega h \mathrm{e}_0 \, dv, \quad \theta := \int_\Omega h \mathrm{e}_2 \, dv
\]
and
\[
u_i := \int_\Omega h \mathrm{e}_{1,i} \, dv \quad \text{for} \quad i = 1,2,3, \quad u := (u_1, u_2, u_3);
\]
i.e., the macroscopic projection $\mathcal{P}$ acting on $h$ is indeed the expansion of $h$ in $v$-variable in Legendre polynomials up to power $2$.
The microscopic projection $\mathcal{L}$ in the collision operator (\ref{Col:ano}) is defined to be $\mathcal{L} := I - \mathcal{P}$ where $I$ denotes the identity projection operator.
For $h \in H^1( \mathbf{T}_x^3; L^2(\Omega_v) )$, we further define that
\[
\mathcal{D}(h) := \| \mathcal{L}(h) \|_X.
\]

\begin{lemma} \label{GSmBE}
Let $\varepsilon>0$ and $T \in [0,\infty]$. Suppose that $f_\varepsilon \in L^\infty\big( [0,T); H^1( \mathbf{T}_x^3; L^2(\Omega_v) ) \big)$ is a solution to the Boltzmann equation (\ref{Boltzmann}) with collision operator (\ref{Col:ano}), then $f_\varepsilon$ satisfies the global energy estimate
\begin{align} \label{GloEE}
\sup_{t \in [0,T)} \, \mathcal{E}\big( f_\varepsilon \big)^2 (t) + \frac{1}{\varepsilon^2 \nu_\ast} \int_0^T \mathcal{D} \big( f_\varepsilon \big)^2 (s) \, ds \leq \mathcal{E} \big( f_{\varepsilon,0} \big)^2.
\end{align}
\end{lemma}
\begin{proof}
Let $\alpha \in \mathbf{N}_0^3$ with $|\alpha| := \sum_{i=1}^3 \alpha_i$ be either $0$ or $1$. 
In the case where $|\alpha|=0$, $\partial_x^\alpha f_\varepsilon$ simply means $f_\varepsilon$ itself.
To obtain the energy estimate for $f_\varepsilon$, we apply the differentiation $\partial_x^\alpha$ to the Boltzmann equation (\ref{Boltzmann}) with collision operator (\ref{Col:ano}) and then take its inner product with $\partial_x^\alpha f_\varepsilon$ in the $L^2( \mathbf{T}_x^3; L^2(\Omega_v) )$ sense. By integration by parts, we observe that
\[
\int_{\mathbf{T}^3} \partial_{x_j} (\partial_x^\alpha f_\varepsilon) \partial_x^\alpha f_\varepsilon \, dx = \frac{1}{2} \int_{\mathbf{T}^3} \partial_{x_j} \big( \partial_x^\alpha f_\varepsilon \big)^2 \, dx = 0
\]
for any $1 \leq j \leq 3$. 
It can be easily verified that the microscopic projection $\mathcal{L}$ has properties that 
\[
\int_{\mathbf{T}^3} \int_\Omega \mathcal{L}(h) g \, dv \, dx = \int_{\mathbf{T}^3} \int_\Omega h \mathcal{L}(g) \, dv \, dx \quad \text{and} \quad \mathcal{L}^2(h) = \mathcal{L}(h)
\]
for any $g,h \in L^2( \mathbf{T}_x^3; L^2(\Omega_v) )$. 
Moreover, we observe that the microscopic projection $\mathcal{L}$ commutes with the differentiation $\partial_x^\alpha$.
As a result, we can deduce that
\begin{equation}\label{EI}
\begin{split}
&\frac{1}{2} \frac{\mathrm{d}}{\mathrm{d} t} \mathcal{E}(f_\varepsilon)^{2}
+\frac{1}{\varepsilon^2 \nu_\ast} \sum_{|\alpha| \leq 1}\int_{\mathbf{T}^3}\int_\Omega \mathcal{L}(\partial_x^\alpha f_\varepsilon) (\partial_x^\alpha f_\varepsilon) \, dv \, dx
=\frac{1}{2} \frac{\mathrm{d}}{\mathrm{d} t} \mathcal{E}(f_\varepsilon)^{2}
+\frac{1}{\varepsilon^2 \nu_\ast} \mathcal{D}(f_\varepsilon)^2
\\
&\ \ = \frac{1}{\nu_\ast} \sum_{|\alpha| \leq 1}\int_{\mathbf{T}^3}\int_\Omega\left( \frac{\kappa}{\varepsilon} \partial_{x}^{\alpha} (f_\varepsilon)^2 \partial_x^\alpha \mathcal{L}(f_\varepsilon) - \kappa^2 \partial_x^\alpha (f_\varepsilon)^3 \partial_x^\alpha f_\varepsilon \right) \, dv \, dx
\\
&\ \ \leq \frac{1}{\nu_\ast} \sum_{|\alpha| \leq 1} \bigg( \frac{2^{2 |\alpha|} \kappa^2}{2} \|f_\varepsilon \partial_x^\alpha f_\varepsilon\|^2_{L^2( \mathbf{T}_x^3; L^2(\Omega_v) )} + \frac{1}{2\varepsilon^{2}} \|\partial_x^\alpha \mathcal{L}(f_\varepsilon)\|^2_{L^2( \mathbf{T}_x^3; L^2(\Omega_v) )} \\
&\quad \quad \quad \quad \quad \quad \ \ - 3^{|\alpha|} \kappa^2 \|f_\varepsilon \partial_x^\alpha f_\varepsilon\|^2_{L^2( \mathbf{T}_x^3; L^2(\Omega_v) )} \bigg)
\\
&\ \ \leq \frac{1}{2\varepsilon^2 \nu_\ast} \mathcal{D}(f_\varepsilon)^2
\end{split}
\end{equation}
for any $t \geq 0$. Therefore, by the absorption principle, it holds for any $t \geq 0$ that 
\begin{equation}\label{GEB}
\mathcal{E}\big( f_\varepsilon \big)^2 (t) + \frac{1}{\varepsilon^2 \nu_\ast} \int_0^t \mathcal{D}\big( f_\varepsilon \big)^2(s) \, ds \leq \mathcal{E}(f_{\varepsilon,0})^2.
\end{equation}
This completes the proof of Lemma \ref{GSmBE}.
\end{proof}

\subsection{Cutoff in Fourier space} 
\label{sub:CF}

In this section, we investigate in detail the cutoff in Fourier space operator $\Lambda_\varepsilon$ that is needed for the analysis of approximate equation (\ref{AEtoBA}). We shall begin with its definition.

Let $\varepsilon \in (0,1)$ and $p,q \in [1,\infty]$. 
Let $\gamma>0$ be a real number which we will determine later in this paper.
For $h \in L^p(\Omega_v)$ and $1 \leq j \leq 3$, we define that
\[
\Lambda_\varepsilon^{v_j}(h) := \sum_{m_j \in \mathbf{Z}, \, |m_j| < \frac{1}{\varepsilon^\gamma}} \mathcal{F}_{v_j}\big( h \big) (m_j) \mathrm{e}^{2 \pi i m_j v_j}, \quad \mathcal{F}_{v_j}\big( h \big) (m_j) := \int_{-\frac{1}{2}}^{\frac{1}{2}} h(v) \mathrm{e}^{- 2 \pi i m_j v_j} \, dv_j,
\]
i.e., $\Lambda_\varepsilon^{v_j}(h)$ is the cutoff of $h$ in Fourier space in the sense of Fourier series with respect to the $j$-th component of $v$. 
Let us note that 
\[
\Lambda_\varepsilon^{v_j} \big( \Lambda_\varepsilon^{v_j}(h) \big) = \Lambda_\varepsilon^{v_j}(h), \quad \forall \; 1 \leq j \leq 3.
\]
Moreover, it holds that
\[
\Lambda_\varepsilon^{v_i} \big( \Lambda_\varepsilon^{v_j}(h) \big) = \Lambda_\varepsilon^{v_j}\big( \Lambda_\varepsilon^{v_i}(h) \big) \quad \forall \; 1 \leq i,j \leq 3 \quad \text{such that} \quad i \neq j.
\]
We further define that
\[
\Lambda_\varepsilon^v(h) := \Lambda_\varepsilon^{v_3}\Big( \Lambda_\varepsilon^{v_2}\big( \Lambda_\varepsilon^{v_1}(h) \big) \Big) = h \ast D_{\varepsilon^{-1},v}^3
\]
where
\[
D_{\varepsilon^{-1},v}^3 := \sum_{m \in \mathbf{Z}^3, \, |m_j| < \frac{1}{\varepsilon^\gamma}} \mathrm{e}^{2 \pi i m \cdot v}
\]
represents the square Dirichlet kernel on torus $\Omega$.

On the other hand, for $f \in L^p(\mathbf{T}_x^3)$, we define that
\[
\Lambda_\varepsilon^x(f) := \sum_{m \in \mathbf{Z}^3, \, |m| < \frac{1}{\varepsilon^\gamma}} \mathcal{F}_x\big( f \big) (m) \mathrm{e}^{2 \pi i m \cdot x}, \quad \mathcal{F}_x\big( f \big) (m) := \int_{\mathbf{T}^3} f(x) \mathrm{e}^{- 2 \pi i m \cdot x} \, dx,
\]
i.e., $\Lambda_\varepsilon^x(f)$ is the cut-off in Fourier space for $f$ in the sense of Fourier series with respect to $x$. 
Similarly, we have that
\[
\Lambda_\varepsilon^x(f) = f \ast \mathring{D}^3_{\varepsilon^{-1},x}
\]
with
\[
\mathring{D}^3_{\varepsilon^{-1},x} := \sum_{m \in \mathbf{Z}^3, \, |m| < \frac{1}{\varepsilon^\gamma}} \mathrm{e}^{2 \pi i m \cdot x}
\]
representing the spherical Dirichlet kernel on torus $\mathbf{T}^3$. 
It is easy to observe that for any $g \in L^p( \mathbf{T}_x^3; L^q(\Omega_v) )$, it holds that
\begin{align} \label{Comxv}
\Lambda_\varepsilon^x\big( \Lambda_\varepsilon^v(g) \big) = \Lambda_\varepsilon^v \big( \Lambda_\varepsilon^x(g) \big).
\end{align}
Hence, we define the cutoff in Fourier space operator 
\[
\Lambda_\varepsilon(g) := \Lambda_\varepsilon^x\big( \Lambda_\varepsilon^v(g) \big) \quad \text{for any} \quad g \in L^p( \mathbf{T}_x^3; L^q(\Omega_v) )
\]
without causing any ambiguity.
Since 
\begin{align} \label{SqCtf}
\big( \Lambda_\varepsilon^v \big)^2 = \Lambda_\varepsilon^v \quad \text{and} \quad \big( \Lambda_\varepsilon^x \big)^2 = \Lambda_\varepsilon^x,
\end{align}
from equality (\ref{Comxv}) we see that
\[
\Lambda_\varepsilon \big( \Lambda_\varepsilon (g) \big) = \Lambda_\varepsilon (g) \quad \text{for any} \quad g \in L^p( \mathbf{T}_x^3; L^q(\Omega_v) ).
\]

It is worth to mention that the cutoff in Fourier space operator in the continuous setting satisfies the Bernstein-type lemma, see e.g. \cite[Lemma 2.1]{BCD}. In the Fourier series case, we have similar estimates for cutoff operators $\Lambda_\varepsilon^x$ and $\Lambda_\varepsilon^v$.

\begin{lemma} \label{BerL}
Let $k \in \mathbf{N}_0$, $1 \leq p \leq q \leq \infty$ and $\alpha \in \mathbf{N}_0^3$ with $|\alpha|=k$. There exists a constant $C(k)$, which depends only on $k$, such that the estimate
\[
\big\| \partial_v^\alpha \Lambda_\varepsilon^v(h) \big\|_{L^q(\Omega_v)} \leq \frac{C(k)}{\varepsilon^{\gamma (k+3)}} \| h \|_{L^p(\Omega_v)}
\]
holds for any $h \in L^p(\Omega_v)$ and the estimate
\[
\big\| \partial_x^\alpha \Lambda_\varepsilon^x(f) \big\|_{L^q(\mathbf{T}_x^3)} \leq \frac{C(k)}{\varepsilon^{\gamma (k+3)}} \| f \|_{L^p(\mathbf{T}_x^3)}
\]
holds for any $f \in L^p(\mathbf{T}_x^3)$.
\end{lemma}
\begin{proof}
This lemma is a direct application of Young's inequality. For $h \in L^p(\Omega_v)$, it holds that
\[
\partial_v^\alpha \Lambda_\varepsilon^v(h) = h \ast \big( \partial_{v_1}^{\alpha_1} \partial_{v_2}^{\alpha_2} \partial_{v_3}^{\alpha_3} D_{\varepsilon^{-1},v}^3 \big) = h \ast \bigg( \sum_{m \in \mathbf{Z}^3, \, |m_j| < \frac{1}{\varepsilon^\gamma}} (2 \pi i)^k m_1^{\alpha_1} m_2^{\alpha_2} m_3^{\alpha_3} \mathrm{e}^{2 \pi i m \cdot v} \bigg).
\]
Hence, by Young's inequality \cite[Theorem 1.2.12]{GraC} and then the triangle inequality, we can deduce that
\begin{align*}
\big\| \partial_v^\alpha \Lambda_\varepsilon^v(h) \big\|_{L^q(\Omega_v)} &\leq \bigg\| \sum_{m \in \mathbf{Z}^3, \, |m_j| < \frac{1}{\varepsilon^\gamma}} (2 \pi i)^k m_1^{\alpha_1} m_2^{\alpha_2} m_3^{\alpha_3}\mathrm{e}^{2 \pi i m \cdot v} \bigg\|_{L^r(\Omega_v)} \| h \|_{L^p(\Omega_v)} \\
&\leq \frac{(2 \pi)^k}{\varepsilon^{\gamma k}} \bigg( \sum_{n \in \mathbf{Z}, \, |n| < \frac{1}{\varepsilon^\gamma}} 1 \bigg)^3 \| h \|_{L^p(\Omega_v)} \leq \frac{2^3 (2 \pi)^k}{\varepsilon^{\gamma (k+3)}} \| h \|_{L^p(\Omega_v)}
\end{align*}
where $r = \frac{qp}{qp+p-q}$.
On the other hand, since 
\[
\sum_{m \in \mathbf{Z}^3, \, |m|<\frac{1}{\varepsilon^\gamma}} 1 < \bigg( \sum_{n \in \mathbf{Z}, \, |n|<\frac{1}{\varepsilon^\gamma}} 1 \bigg)^3,
\]
by exactly the same derivation above, we can show that the estimate
\[
\big\| \partial_x^\alpha \Lambda_\varepsilon^x(f) \big\|_{L^q(\mathbf{T}_x^3)} \leq \frac{2^3 (2 \pi)^k}{\varepsilon^{\gamma (k+3)}} \| f \|_{L^p(\mathbf{T}_x^3)}
\]
holds for any $f \in L^p(\mathbf{T}_x^3)$. This completes the proof of Lemma \ref{BerL}.
\end{proof}
\begin{remark} \label{CommuDC}
Let $p,q \in [1,\infty]$ and $\alpha \in \mathbf{N}_0^3$.
It can be easily verified from the definition that for any $g \in L^p( \mathbf{T}_x^3; L^q(\Omega_v) )$, it holds that
\[
\partial_x^\alpha \Lambda_\varepsilon (g) = \Lambda_\varepsilon(\partial_x^\alpha g) \quad \text{and} \quad \partial_v^\alpha \Lambda_\varepsilon (g) = \Lambda_\varepsilon(\partial_v^\alpha g),
\]
i.e., differential operators $\partial_x^\alpha, \partial_v^\alpha$ commute with the cutoff operator $\Lambda_\varepsilon$. 
\end{remark}

Due to relation (\ref{SqCtf}), Lemma \ref{BerL} has a further implication.

\begin{lemma} \label{Re:BerL}
Let $k \in \mathbf{N}_0$ and $\alpha \in \mathbf{N}_0^3$ with $|\alpha|=k$. Let $p,q,r \in [1,\infty]$ with $p \leq q$. There exists a constant $C(k)$, which depends only on $k$, such that the estimate
\[
\big\| \partial_v^\alpha \Lambda_\varepsilon^v(h) \big\|_{L^q(\Omega_v)} \leq \frac{C(k)}{\varepsilon^{\gamma (k+3)}} \| \Lambda_\varepsilon^v (h) \|_{L^p(\Omega_v)}
\]
holds for any $h \in L^r(\Omega_v)$ and the estimate
\[
\big\| \partial_x^\alpha \Lambda_\varepsilon^x(f) \big\|_{L^q(\mathbf{T}_x^3)} \leq \frac{C(k)}{\varepsilon^{\gamma (k+3)}} \| \Lambda_\varepsilon^x (f) \|_{L^p(\mathbf{T}_x^3)}
\]
holds for any $f \in L^r(\mathbf{T}_x^3)$. 
\end{lemma}
\begin{proof}
Let $h \in L^r(\Omega_v)$ and $f \in L^r(\mathbf{T}_x^3)$. 
Note that $\| \Lambda_\varepsilon^v (h) \|_{L^p(\Omega_v)}$ and $\| \Lambda_\varepsilon^x (f) \|_{L^p(\mathbf{T}_x^3)}$ are both finite regardless of the value $r$ as long as $r \in [1,\infty]$.
If $r \leq p$, then the finiteness of $\| \Lambda_\varepsilon^v (h) \|_{L^p(\Omega_v)}$ and $\| \Lambda_\varepsilon^x (f) \|_{L^p(\mathbf{T}_x^3)}$ can be shown by Lemma \ref{BerL}.
If $r > p$, then by H$\ddot{\text{o}}$lder's inequality we have that
\[
\| \Lambda_\varepsilon^v (h) \|_{L^p(\Omega_v)} \leq \| \Lambda_\varepsilon^v (h) \|_{L^r(\Omega_v)}, \quad \| \Lambda_\varepsilon^x (f) \|_{L^p(\mathbf{T}_x^3)} \leq \| \Lambda_\varepsilon^x (f) \|_{L_x^r(\mathbf{T}^3)}.
\]
The finiteness of $\| \Lambda_\varepsilon^v (h) \|_{L_v^r(\Omega)}$ and $\| \Lambda_\varepsilon^x (f) \|_{L^r(\mathbf{T}_x^3)}$ can be further guaranteed by Lemma \ref{BerL}.
Hence, it holds that $\Lambda_\varepsilon^v (h) \in L^p(\Omega_v)$ and $\Lambda_\varepsilon^x (f) \in L^p(\mathbf{T}_x^3)$ for any $1 \leq p \leq \infty$.
By relation (\ref{SqCtf}), we can view $\Lambda_\varepsilon^v (h)$ as $\Lambda_\varepsilon^v \big( \Lambda_\varepsilon^v (h) \big)$ and $\Lambda_\varepsilon^x (f)$ as $\Lambda_\varepsilon^x \big( \Lambda_\varepsilon^x (f) \big)$.
Therefore, by applying Lemma \ref{BerL} directly to $\Lambda_\varepsilon^v \big( \Lambda_\varepsilon^v (h) \big)$ and $\Lambda_\varepsilon^x \big( \Lambda_\varepsilon^x (f) \big)$, we obtain Lemma \ref{Re:BerL}.
\end{proof}

Let us note that for any $m,n \in \mathbf{Z}^3$, it holds that
\begin{align} \label{Orthexp}
\int_\Omega \mathrm{e}^{2 \pi i m \cdot v} \overline{\mathrm{e}^{2 \pi i n \cdot v}} \, dv = \delta_{mn} = \int_{\mathbf{T}^3} \mathrm{e}^{2 \pi i m \cdot x} \overline{\mathrm{e}^{2 \pi i n \cdot x}} \, dx
\end{align}
where $\overline{\mathrm{e}^{2 \pi i n \cdot v}}$ denotes the complex conjugate of $\mathrm{e}^{2 \pi i n \cdot v}$ and $\overline{\mathrm{e}^{2 \pi i n \cdot x}}$ denotes the complex conjugate of $\mathrm{e}^{2 \pi i n \cdot x}$.
As a result, we have that $\Lambda_\varepsilon^v(1) = \Lambda_\varepsilon^x(1) = 1$.
For each $1 \leq i \leq 3$, we then define that
\begin{align} \label{Cf:e012}
\mathrm{e}_0^\varepsilon := 1, \quad \mathrm{e}_{1,i}^\varepsilon := c_1^\varepsilon \Lambda_\varepsilon^{v_i} (v_i) \quad \text{and} \quad \mathrm{e}_{2,i}^\varepsilon :=  c_2^\varepsilon \Lambda_\varepsilon^{v_i} (v_i^2) - c_0^\varepsilon 
\end{align}
to be the cutoff in Fourier space version of $3$-dimensional Legendre polynomials in $\Omega$ up to power $2$.
We further set
\[
\mathrm{e}_1^\varepsilon := \big( \mathrm{e}_{1,1}^\varepsilon, \mathrm{e}_{1,2}^\varepsilon, \mathrm{e}_{1,3}^\varepsilon \big) \quad \text{and} \quad \mathrm{e}_2^\varepsilon := \sum_{i=1}^3 \mathrm{e}_{2,i}^\varepsilon.
\]
Constants $c_0^\varepsilon, c_1^\varepsilon, c_2^\varepsilon$ in definition (\ref{Cf:e012}) are specially chosen so that the set $\{ \mathrm{e}_0^\varepsilon, \mathrm{e}_{1,i}^\varepsilon, \mathrm{e}_2^\varepsilon \}_{1 \leq i \leq 3}$ is orthonormal, i.e., we determine constants $c_0^\varepsilon, c_1^\varepsilon, c_2^\varepsilon$ by requiring
\[
\int_\Omega \mathrm{e}_i^\varepsilon \mathrm{e}_j^\varepsilon \, dv = \delta_{ij}, \quad \int_\Omega \mathrm{e}_i^\varepsilon \mathrm{e}_{1,k}^\varepsilon \, dv = 0, \quad \int_\Omega \mathrm{e}_{1,k}^\varepsilon \mathrm{e}_{1,\ell}^\varepsilon \, dv = \delta_{k \ell} \quad \forall \; i,j \in \{0,2\}; \; k, \ell \in \{1,2,3\}.
\]
Since $v_i^s \in L^p(\Omega_v)$ trivially for any $1 < p < \infty$ and $s \in \mathbf{N}$, it holds that
\begin{align} \label{LpE:Cfv}
\| \Lambda_\varepsilon^{v_i}(v_i^s) - v_i^s \|_{L^p(\Omega_v)} \to 0 \quad \text{as} \quad \varepsilon \to 0
\end{align}
for any $1 < p < \infty$ and $s \in \mathbf{N}$, see e.g. \cite[Theorem 4.1.8]{GraC}.

For $h \in L^2( \mathbf{T}_x^3; L^2(\Omega_v) )$, we define the cutoff in Fourier space version of the macroscopic projection $\mathcal{P}$ acting on $h$ by
\begin{align} \label{defP}
\mathcal{P}^\varepsilon(h) := \rho^\varepsilon \mathrm{e}_0^\varepsilon + u^\varepsilon \cdot \mathrm{e}_1^\varepsilon + \theta^\varepsilon \mathrm{e}_2^\varepsilon
\end{align}
where
\[
\rho^\varepsilon := \int_\Omega h \mathrm{e}_0^\varepsilon \, dv, \quad \theta^\varepsilon := \int_\Omega h \mathrm{e}_2^\varepsilon \, dv
\]
and
\[
u_i^\varepsilon := \int_\Omega h \mathrm{e}_{1,i}^\varepsilon \, dv \quad ( i = 1,2,3 ), \quad u^\varepsilon := (u_1^\varepsilon, u_2^\varepsilon, u_3^\varepsilon).
\]
The cutoff in Fourier space version of the microscopic projection $\mathcal{L}$ in the approximate equation (\ref{AEtoBA}) is then defined as $\mathcal{L}^\varepsilon := I - \mathcal{P}^\varepsilon$.

For $f, h \in L^2( \mathbf{T}_x^3; L^2(\Omega_v) )$, we use the notation 
\[
\langle f, h \rangle_v := \int_\Omega f h \, dv
\]
to represent the inner product of $f$ and $h$ in $\Omega$ with respect to $v$.
As an end to this section, we would like to give an estimate to the $X$-norm of $\Lambda_\varepsilon (h)$ for $h \in H^1( \mathbf{T}_x^3; L^2(\Omega_v) )$.

\begin{lemma} \label{EXCut}
For any $h \in H^1( \mathbf{T}_x^3; L^2(\Omega_v) )$, it holds that
\[
\| \Lambda_\varepsilon (h) \|_X \leq \| h \|_X.
\]
\end{lemma}
\begin{proof}
Let $h \in L^2( \mathbf{T}_x^3; L^2(\Omega_v) )$.
Note that
\begin{align*}
\Lambda_\varepsilon (h) &= (h \ast D_{\varepsilon^{-1},v}^3) \ast \mathring{D}_{\varepsilon^{-1},x}^3 \\
&= \sum_{n \in \mathbf{Z}^3, \, |n| < \frac{1}{\varepsilon^\gamma}} \, \sum_{m \in \mathbf{Z}^3, \, |m_j| < \frac{1}{\varepsilon^\gamma}} \, \mathcal{F}_x\big( \mathcal{F}_v (h) \big) (n,m) \, \mathrm{e}^{2 \pi i (m \cdot v + n \cdot x)}.
\end{align*}
Due to the orthogonality (\ref{Orthexp}) and Plancherel's identity, see e.g. \cite[Proposition 3.2.7]{GraC}, we can deduce the $L^2$ estimate for $\Lambda_\varepsilon (h)$, i.e.,
\begin{equation} \label{EL2Cut}
\begin{split}
\int_{\mathbf{T}^3} \int_\Omega \big| \Lambda_\varepsilon (h) \big|^2 \, dv \, dx &= \sum_{n \in \mathbf{Z}^3, \, |n| < \frac{1}{\varepsilon^\gamma}} \, \sum_{m \in \mathbf{Z}^3, \, |m_j| < \frac{1}{\varepsilon^\gamma}} \, \big| \mathcal{F}_x\big( \mathcal{F}_v (h) \big) (n,m) \big|^2 \\
&\leq \sum_{n \in \mathbf{Z}^3} \sum_{m \in \mathbf{Z}^3} \big| \mathcal{F}_x\big( \mathcal{F}_v (h) \big) (n,m) \big|^2 = \int_{\mathbf{T}^3} \int_\Omega |h|^2 \, dv \, dx.
\end{split}
\end{equation}
Since the differential operator $\nabla_x$ commutes with the cut-off operator $\Lambda_\varepsilon$ (see Remark \ref{CommuDC}), by replacing $h$ in estimate (\ref{EL2Cut}) by $\nabla_x h$, we obtain Lemma \ref{EXCut}.
\end{proof}

\subsection{Proof of Lemma \ref{GloSol:AE}: Local existence and uniqueness} 
\label{sub:SAE}

In order to prove Theorem \ref{GloSol:AE}, we need a multiplication rule regarding the $X$-norm.

\begin{proposition} \label{MRX}
Let $\varepsilon \in (0,1)$ and $n \in \mathbf{N}$ with $n \geq 2$. Let $f_1, f_2, ..., f_n \in H^1( \mathbf{T}_x^3; L^2(\Omega_v) )$.
There exists a constant $C>0$, which is independent of $\varepsilon$ and $f_1, f_2, ..., f_n$, such that
\[
\Big\| \prod_{i=1}^n \Lambda_\varepsilon (f_i) \Big\|_X \leq C \varepsilon^{-6n \gamma} \prod_{i=1}^n \| \Lambda_\varepsilon (f_i) \|_X.
\]
\end{proposition}
\begin{proof}
By H$\ddot{\text{o}}$lder's inequality, we have that
\begin{align} \label{P5:E1}
\int_{\mathbf{T}^3} \Big\| \prod_{i=1}^n \Lambda_\varepsilon (f_i) \Big\|_{L^2(\Omega_v)}^2 \, dx \leq \left( \int_{\mathbf{T}^3} \Big\| \prod_{i=1}^n \Lambda_\varepsilon (f_i) \Big\|_{L^2(\Omega_v)}^6 \, dx \right)^{\frac{1}{3}}.
\end{align}
Then, by applying Minkowski's integral inequality to the right hand side of estimate (\ref{P5:E1}), we obtain that
\[
\left( \int_{\mathbf{T}^3} \Big\| \prod_{i=1}^n \Lambda_\varepsilon (f_i) \Big\|_{L^2(\Omega_v)}^6 \, dx \right)^{\frac{1}{3}} \leq \int_\Omega \Big\| \prod_{i=1}^n \Lambda_\varepsilon (f_i) \Big\|_{L^6(\mathbf{T}_x^3)}^2 \, dv.
\]
Since the Sobolev space $H^1(\mathbf{T}_x^3)$ is continuously embedded in $L^6(\mathbf{T}_x^3)$, we thus deduce that
\[
\int_{\mathbf{T}^3} \Big\| \prod_{i=1}^n \Lambda_\varepsilon (f_i) \Big\|_{L^2(\Omega_v)}^2 \, dx \lesssim \int_\Omega \Big\| \prod_{i=1}^n \Lambda_\varepsilon (f_i) \Big\|_{H^1(\mathbf{T}_x^3)}^2 \, dv.
\]

Next, we shall firstly focus on estimating the integral of $\big\| \prod_{i=1}^n \Lambda_\varepsilon (f_i) \big\|_{L^2(\mathbf{T}_x^3)}^2$ in $\Omega$ with respect to $v$.
By H$\ddot{\text{o}}$lder's inequality, we have that
\[
\Big\| \prod_{i=1}^n \Lambda_\varepsilon (f_i) \Big\|_{L^2(\mathbf{T}_x^3)} \leq \prod_{i=1}^n \| \Lambda_\varepsilon (f_i) \|_{L^{2n}(\mathbf{T}_x^3)}.
\]
By Lemma \ref{Re:BerL}, we see that
\[
\| \Lambda_\varepsilon (f_i) \|_{L^{2n}(\mathbf{T}_x^3)} \lesssim \varepsilon^{-3 \gamma} \| \Lambda_\varepsilon (f_i) \|_{L^1(\mathbf{T}_x^3)}, \quad \forall \; 1 \leq i \leq n.
\]
By considering H$\ddot{\text{o}}$lder's inequality and then Minkowski's integral inequality, we further deduce that
\begin{equation} \label{P5:E2}
\begin{split}
\int_\Omega \prod_{i=1}^n \| \Lambda_\varepsilon (f_i) \|_{L^{2n}(\mathbf{T}_x^3)}^2 \, dv &\lesssim \varepsilon^{-6n \gamma} \prod_{i=1}^n \left( \int_\Omega \| \Lambda_\varepsilon (f_i) \|_{L^1(\mathbf{T}_x^3)}^{2n} \, dv \right)^{\frac{1}{n}} \\
&\lesssim \varepsilon^{-6n \gamma} \prod_{i=1}^n \left( \int_{\mathbf{T}^3} \| \Lambda_\varepsilon (f_i) \|_{L^{2n}(\Omega_v)} \, dx \right)^2. \\
\end{split}
\end{equation}
By Lemma \ref{Re:BerL} again, we have that
\begin{align} \label{P5:E3}
\| \Lambda_\varepsilon (f_i) \|_{L^{2n}(\Omega_v)} \lesssim \varepsilon^{-3 \gamma} \| \Lambda_\varepsilon (f_i) \|_{L^2(\Omega_v)}, \quad \forall \; 1 \leq i \leq n.
\end{align}
Substituting estimates (\ref{P5:E3}) into estimate (\ref{P5:E2}) and then applying H$\ddot{\text{o}}$lder's inequality once more, we obtain that
\begin{align} \label{L2Lfh}
\int_\Omega \Big\| \prod_{i=1}^n \Lambda_\varepsilon (f_i) \Big\|_{L^2(\mathbf{T}_x^3)}^2 \, dv \lesssim \varepsilon^{-12n \gamma} \prod_{i=1}^n \int_{\mathbf{T}^3} \int_\Omega \big| \Lambda_\varepsilon (f_i) \big|^2 \, dv \, dx.
\end{align}

Since the differential operator $\nabla_x$ commutes with the cutoff operator $\Lambda_\varepsilon$ (see Remark \ref{CommuDC}), the estimate for the integral of $\big\| \nabla_x \big( \prod_{i=1}^n \Lambda_\varepsilon (f_i) \big) \big\|_{L^2(\mathbf{T}_x^3)}^2$ in $\Omega$ with respect to $v$ follows directly from inequality (\ref{L2Lfh}). By chain rule, we observe that
\begin{align*}
\Big\| \nabla_x \Big( \prod_{i=1}^n \Lambda_\varepsilon (f_i) \Big) \Big\|_{L^2(\mathbf{T}_x^3)}^2 &\lesssim \sum_{i=1}^n   \Big\| \Lambda_\varepsilon (\nabla_x f_i) \prod_{j \neq i} \Lambda_\varepsilon (f_j) \Big\|_{L^2(\mathbf{T}_x^3)}^2.
\end{align*}
Hence, by estimate (\ref{L2Lfh}) we can directly conclude that
\begin{align*}
&\int_\Omega \Big\| \Lambda_\varepsilon (\nabla_x f_i) \prod_{j \neq i} \Lambda_\varepsilon (f_j) \Big\|_{L^2(\mathbf{T}_x^3)}^2 \, dv \\
&\ \ \lesssim \varepsilon^{-12n \gamma} \bigg( \int_{\mathbf{T}^3} \int_\Omega \big| \Lambda_\varepsilon (\nabla_x f_i) \big|^2 \, dv \, dx \bigg) \prod_{j \neq i} \int_{\mathbf{T}^3} \int_\Omega \big| \Lambda_\varepsilon (f_j) \big|^2 \, dv \, dx
\end{align*}
for any $1 \leq i \leq n$.
This completes the proof of Proposition \ref{MRX}.
\end{proof}
\begin{remark} \label{Rk:AEMR}
We would like to emphasize that if we invoke Lemma \ref{Re:BerL} to estimate
\[
\big\| \nabla_x \Lambda_\varepsilon (f_i) \big\|_{L^{2n}(\mathbf{T}_x^3)} \lesssim \varepsilon^{-4 \gamma} \| \Lambda_\varepsilon (f_i) \|_{L^1(\mathbf{T}_x^3)}, \quad \forall \; 1 \leq i \leq n
\]
and then follow the derivation of inequality (\ref{L2Lfh}), we can indeed deduce that there exists a constant $C>0$, which is independent of $\varepsilon$ and $f_1, f_2, ..., f_n$, such that the estimate
\[
\Big\| \prod_{i=1}^n \Lambda_\varepsilon (f_i) \Big\|_X \leq C \varepsilon^{-7n \gamma} \prod_{i=1}^n \| \Lambda_\varepsilon (f_i) \|_{L^2( \mathbf{T}_x^3; L^2(\Omega_v) )}
\]
holds for any $f_1, f_2, ..., f_n \in L^p( \mathbf{T}_x^3; L^q(\Omega_v) )$ with $p,q \in [1,\infty]$.
\end{remark}

Furthermore, we would like to give an estimate regarding the microscopic projection $\mathcal{L}^\varepsilon$.

\begin{proposition} \label{BoLo}
Let $\varepsilon \in (0,1)$. Then $\mathcal{L}^\varepsilon : H^1( \mathbf{T}_x^3; L^2(\Omega_v) ) \to H^1( \mathbf{T}_x^3; L^2(\Omega_v) )$ is a bounded linear operator. For any $h \in H^1( \mathbf{T}_x^3; L^2(\Omega_v) )$, it holds that
\[
\| \mathcal{L}^\varepsilon (h) \|_X \leq \| h \|_X.
\]
\end{proposition}
\begin{proof}
Let $h \in H^1( \mathbf{T}_x^3; L^2(\Omega_v) )$. By the definition of the microscopic projection $\mathcal{L}^\varepsilon$, we have that
\begin{align} \label{Ex:Lh}
\mathcal{L}^\varepsilon(h) = h - \rho^\varepsilon \mathrm{e}_0^\varepsilon - u^\varepsilon \cdot \mathrm{e}_1^\varepsilon - \theta^\varepsilon \mathrm{e}_2^\varepsilon.
\end{align}
Since $\mathrm{e}_i^\varepsilon$ and $\mathrm{e}_j^\varepsilon$ are orthogonal for any $i \neq j$, by substituting $\mathcal{L}^\varepsilon (h)$ using expression (\ref{Ex:Lh}), we can deduce that
\begin{align*}
0 \leq \int_{\mathbf{T}^3} \int_\Omega \mathcal{L}^\varepsilon(h)^2 \, dv \, dx &= \int_{\mathbf{T}^3} \int_\Omega h^2 \, dv \, dx - \| \rho^\varepsilon \|_{L^2(\mathbf{T}_x^3)}^2 - \| u^\varepsilon \|_{L^2(\mathbf{T}_x^3)}^2 - \| \theta^\varepsilon \|_{L^2(\mathbf{T}_x^3)}^2 \\
&\leq \int_{\mathbf{T}^3} \int_\Omega h^2 \, dv \, dx.
\end{align*}
Continuing to manipulate expression (\ref{Ex:Lh}), it is easy to observe that the differential operator $\nabla_x$ commutes with the microscopic projection $\mathcal{L}^\varepsilon$, i.e., we have that
\begin{align} \label{Ex:GLh}
\nabla_x \mathcal{L}^\varepsilon(h) = \mathcal{L}^\varepsilon (\nabla_x h) = \nabla_x h - \mathrm{e}_0^\varepsilon \nabla_x \rho^\varepsilon - (\nabla_x u^\varepsilon) \cdot \mathrm{e}_1^\varepsilon - \mathrm{e}_2^\varepsilon \nabla_x \theta^\varepsilon.
\end{align}
Similarly, due to the orthogonality between $\mathrm{e}_i^\varepsilon$ and $\mathrm{e}_j^\varepsilon$ for $0 \leq i,j \leq 2$ with $i \neq j$, by substituting $\nabla_x \mathcal{L}^\varepsilon (h)$ using expression (\ref{Ex:GLh}), it can be deduced that
\begin{align*}
&\int_{\mathbf{T}^3} \int_\Omega \big| \nabla_x \mathcal{L}^\varepsilon(h) \big|^2 \, dv \, dx \\
&\ \ = \int_{\mathbf{T}^3} \int_\Omega | \nabla_x h |^2 \, dv \, dx - \| \nabla_x \rho^\varepsilon \|_{L^2(\mathbf{T}_x^3)}^2 - \| \nabla_x u^\varepsilon \|_{L^2(\mathbf{T}_x^3)}^2 - \| \nabla_x \theta^\varepsilon \|_{L^2(\mathbf{T}_x^3)}^2 \\
&\ \ \leq \int_{\mathbf{T}^3} \int_\Omega | \nabla_x h |^2 \, dv \, dx.
\end{align*}
This completes the proof of Proposition \ref{BoLo}.
\end{proof}

Now we are ready to give a proof to Lemma \ref{GloSol:AE}.

\begin{proof}[Proof of Lemma \ref{GloSol:AE} (Local existence):]
Let $f_{\varepsilon,0} \in H^1( \mathbf{T}_x^3; L^2(\Omega_v) )$.
Since we have properties
\[
\big\langle \Lambda_\varepsilon (h), g \big\rangle_v = \big\langle h, \Lambda_\varepsilon (g) \big\rangle_v \quad \text{and} \quad \Lambda_\varepsilon \big( \Lambda_\varepsilon (h) \big) = \Lambda_\varepsilon (h)
\]
for any $h, g \in L^2( \mathbf{T}_x^3; L^2(\Omega_v) )$, the global energy estimate (\ref{GloEE:ep}) can be derived by exactly the same proof of Lemma \ref{GSmBE}.
Hence, it is sufficient to construct a local solution formally to the approximate equation \eqref{AEtoBA} that is compactly supported in Fourier space.
For simplicity of notations, within this proof we denote $g_\varepsilon = \Lambda_\varepsilon (f_\varepsilon)$.
Taking the integration of our approximate equation (\ref{AEtoBA}) with respect to time, we have that 
\begin{equation}\label{IAmBE}
\begin{split}
g_\varepsilon(t) &= g_{\varepsilon, 0} - \int_0^t \left( \frac{1}{\varepsilon} \Lambda_\varepsilon \big( v \cdot \nabla_x g_\varepsilon \big) + \frac{1}{\varepsilon^2 \nu_\ast} \Lambda_\varepsilon \big( \mathcal{L}^\varepsilon(g_\varepsilon) \big) \right) \, ds \\
&\ \ + \int_0^t \left( \frac{\kappa}{\varepsilon \nu_\ast} \Lambda_\varepsilon \big( \mathcal{L}^\varepsilon( g_\varepsilon^2 ) \big) - \frac{\kappa^2}{\nu_\ast} \Lambda_\varepsilon( g_\varepsilon^3 ) \right) \, ds
\end{split}
\end{equation}
where $g_{\varepsilon, 0} = \Lambda_\varepsilon (f_{\varepsilon,0})$. Let us consider a sequence of functions $\{g_{\varepsilon, j}\}_{j \in \mathbf{N}_0}$ which is constructed inductively as follows. We set
\begin{equation} \label{def:gj}
\begin{split}
g_{\varepsilon, j+1}(t) 
&:= g_{\varepsilon, 0}
- \int_0^t \left( \frac{1}{\varepsilon} \Lambda_\varepsilon \big( v \cdot \nabla_x g_{\varepsilon, j}(s) \big) + \frac{1}{\varepsilon^2 \nu_\ast} \Lambda_\varepsilon \Big( \mathcal{L}^\varepsilon\big( g_{\varepsilon, j}(s) \big) \Big) \right) \, ds \\
&\ \ + \int_0^t \left( \frac{\kappa}{\varepsilon \nu_\ast} \Lambda_\varepsilon \Big( \mathcal{L}^\varepsilon\big( g_{\varepsilon, j}(s)^2 \big) \Big) - \frac{\kappa^2}{\nu_\ast} \Lambda_\varepsilon \big( g_{\varepsilon, j}(s)^3 \big) \right) \, ds
\end{split}
\end{equation}
for $j \geq 0$. The key idea here to construct a local solution is to show that the sequence $\{ g_{\varepsilon, j} \}_{j \in \mathbf{N}_0}$ is Cauchy in $L^\infty( [0,T]; H^1( \mathbf{T}_x^3; L^2(\Omega_v) ) )$ for some small $T>0$.
By Minkowski's integral inequality, for any $j \geq 1$ we have that
\begin{equation} \label{DXgj1}
\begin{split}
\| g_{\varepsilon, j+1} - g_{\varepsilon, j} \|_X &\leq \frac{1}{\varepsilon} \int_0^t \big\| \Lambda_\varepsilon \big( v \cdot \nabla_x ( g_{\varepsilon, j} - g_{\varepsilon, j-1} ) \big) \big\|_X \, ds \\
&\ \ + \frac{1}{\varepsilon^2 \nu_\ast} \int_0^t \big\| \Lambda_\varepsilon \big( \mathcal{L}^\varepsilon(g_{\varepsilon, j} - g_{\varepsilon, j-1}) \big) \big\|_X \, ds \\
&\ \ + \frac{\kappa}{\varepsilon \nu_\ast} \int_0^t \big\| \Lambda_\varepsilon \big( \mathcal{L}^\varepsilon( g_{\varepsilon, j}^2 - g_{\varepsilon, j-1}^2 ) \big) \big\|_X \, ds + \frac{\kappa^2}{\nu_\ast} \int_0^t \big\| \Lambda_\varepsilon ( g_{\varepsilon, j}^3 - g_{\varepsilon, j-1}^3 ) \big\|_X \, ds.
\end{split}
\end{equation}
By Lemma \ref{EXCut} and then Lemma \ref{Re:BerL}, we observe that
\begin{align} \label{XE:vdgg}
\big\| \Lambda_\varepsilon \big( v \cdot \nabla_x ( g_{\varepsilon, j} - g_{\varepsilon, j-1} ) \big) \big\|_X \leq \| v \cdot \nabla_x ( g_{\varepsilon, j} - g_{\varepsilon, j-1} ) \|_X \lesssim \varepsilon^{-4 \gamma} \| g_{\varepsilon, j} - g_{\varepsilon, j-1} \|_X.
\end{align}
By Lemma \ref{EXCut} and then Proposition \ref{BoLo}, we have that
\[
\big\| \Lambda_\varepsilon \big( \mathcal{L}^\varepsilon (g_{\varepsilon, j} - g_{\varepsilon, j-1}) \big) \big\|_X \leq \| \mathcal{L}^\varepsilon( g_{\varepsilon, j} - g_{\varepsilon, j-1} ) \|_X \leq \| g_{\varepsilon, j} - g_{\varepsilon, j-1} \|_X.
\]
Furthermore, by applying Lemma \ref{EXCut} with Proposition \ref{BoLo} again, we deduce by Proposition \ref{MRX} that
\begin{align*}
\big\| \Lambda_\varepsilon \big( \mathcal{L}^\varepsilon( g_{\varepsilon, j}^2 - g_{\varepsilon, j-1}^2 ) \big) \big\|_X \lesssim \varepsilon^{-12 \gamma} \big( \| g_{\varepsilon, j} \|_X + \| g_{\varepsilon, j-1} \|_X \big) \| g_{\varepsilon, j} - g_{\varepsilon, j-1} \|_X
\end{align*}
and 
\begin{align*}
\big\| \Lambda_\varepsilon( g_{\varepsilon, j}^3 - g_{\varepsilon, j-1}^3 ) \big\|_X \lesssim \varepsilon^{-18 \gamma} \big( \| g_{\varepsilon, j} \|_X^2 + \| g_{\varepsilon, j} \|_X \| g_{\varepsilon, j-1} \|_X + \| g_{\varepsilon, j-1} \|_X^2 \big) \| g_{\varepsilon, j} - g_{\varepsilon, j-1} \|_X.
\end{align*}
Substituting the above four estimates into estimate (\ref{DXgj1}), we then obtain that
\begin{equation} \label{EDfj}
\begin{split}
&\| g_{\varepsilon, j+1} - g_{\varepsilon, j} \|_X \\
&\ \ \lesssim \bigg( \frac{1}{\varepsilon^{4 \gamma +1}} + \frac{1}{\varepsilon^2 \nu_\ast} \bigg) \int_0^t \| g_{\varepsilon, j} - g_{\varepsilon, j-1} \|_X \, ds \\
&\ \ \ \ + \frac{\kappa}{\varepsilon^{12 \gamma + 1} \nu_\ast} \int_0^t \big( \| g_{\varepsilon, j} \|_X + \| g_{\varepsilon, j-1} \|_X \big) \| g_{\varepsilon, j} - g_{\varepsilon, j-1} \|_X \, ds \\
&\ \ \ \ + \frac{\kappa^2}{\varepsilon^{18 \gamma} \nu_\ast} \int_0^t \big( \| g_{\varepsilon, j} \|^2_X + \| g_{\varepsilon, j} \|_X \| g_{\varepsilon, j-1} \|_X + \| g_{\varepsilon, j-1} \|^2_X \big) \| g_{\varepsilon, j} - g_{\varepsilon, j-1} \|_X \, ds.
\end{split}
\end{equation} 

Next, we shall prove by induction that there exists $t>0$ sufficiently small so that for any $j \geq 1$, it holds that
\begin{equation} \label{UEfj}
\begin{split}
\| g_{\varepsilon, j} \|_{L_t^\infty X} &:= \sup_{0 \leq s \leq t} \| g_{\varepsilon, j}(s) \|_X \\
&\lesssim \| f_{\varepsilon, 0} \|_X + 2t E(\varepsilon) \big( \| f_{\varepsilon, 0} \|_X + \| f_{\varepsilon, 0} \|_X^2 + \| f_{\varepsilon, 0} \|_X^3 \big) =: C(f_{\varepsilon, 0}, t)
\end{split}
\end{equation}
where
\[
E(\varepsilon) := \frac{1}{\varepsilon^{4 \gamma +1}} + \frac{1}{\varepsilon^2 \nu_\ast} + \frac{\kappa}{\varepsilon^{12 \gamma + 1} \nu_\ast} + \frac{\kappa^2}{\varepsilon^{18 \gamma} \nu_\ast}
\]
and for any $j \geq 2$, it holds that
\begin{align} \label{DfjA}
\| g_{\varepsilon, j} - g_{\varepsilon, j-1} \|_{L_t^\infty X} \leq \frac{1}{2} \| g_{\varepsilon, j-1} - g_{\varepsilon, j-2} \|_{L_t^\infty X}.
\end{align}
Let $k \in \mathbf{N}$ with $k \geq 2$. Suppose that estimate (\ref{UEfj}) and estimate (\ref{DfjA}) hold simultaneously for all $1 \leq j \leq k$. Then by estimate (\ref{EDfj}), we have that
\begin{align*}
&\| g_{\varepsilon, k+1} - g_{\varepsilon, k} \|_{L_t^\infty X} \lesssim t E(\varepsilon) \| g_{\varepsilon, k} - g_{\varepsilon, k-1} \|_{L_t^\infty X} \\
&\ \ \times \Big( 1 + \| g_{\varepsilon, k} \|_{L_t^\infty X} + \| g_{\varepsilon, k-1} \|_{L_t^\infty X} + \| g_{\varepsilon, k} \|_{L_t^\infty X}^2 + \| g_{\varepsilon, k} \|_{L_t^\infty X} \| g_{\varepsilon, k-1} \|_{L_t^\infty X} + \| g_{\varepsilon, k-1} \|_{L_t^\infty X}^2 \Big).
\end{align*}
Using assumption (\ref{UEfj}) for cases where $j=k$ and $j=k-1$, we deduce that
\[
\| g_{\varepsilon, k+1} - g_{\varepsilon, k} \|_{L_t^\infty X} \lesssim t E(\varepsilon) \Big( 1 + 2 C(f_{\varepsilon, 0}, t) + 3 C(f_{\varepsilon, 0}, t)^2 \Big) \| g_{\varepsilon, k} - g_{\varepsilon, k-1} \|_{L_t^\infty X}.
\]
It is easy to observe that there exists $T_\ast > 0$ sufficiently small such that
\[
T_\ast E(\varepsilon) \Big( 1 + 2 C(f_{\varepsilon, 0}, T_\ast) + 3 C(f_{\varepsilon, 0}, T_\ast)^2 \Big) \leq \frac{1}{2}, 
\]
where $T_\ast$ is independent of $k$ and dependent only on $\varepsilon$, $\nu_\ast$, $\kappa$ and $f_{\varepsilon, 0}$. Thus, we obtain that
\begin{align} \label{Dgkp1}
\| g_{\varepsilon, k+1} - g_{\varepsilon, k} \|_{L_{T_\ast}^\infty X} \leq \frac{1}{2} \| g_{\varepsilon, k} - g_{\varepsilon, k-1} \|_{L_{T_\ast}^\infty X}.
\end{align}
Working directly with expression (\ref{def:gj}) in the case where $j=0$, we can show by similar derivations as in the above paragraph that
\[
\| g_{\varepsilon, 1} - g_{\varepsilon, 0} \|_{L_t^\infty X} \leq t E(\varepsilon) \big( \| f_{\varepsilon, 0} \|_X + \| f_{\varepsilon, 0} \|_X^2 + \| f_{\varepsilon, 0} \|_X^3 \big)
\]
for any $t \geq 0$.
Using assumption (\ref{DfjA}) together with estimate (\ref{Dgkp1}), we can control $\| g_{\varepsilon, j} - g_{\varepsilon, j-1} \|_{L_{T_\ast}^\infty X}$ for all $j$ from $2$ to $k+1$, i.e., we have that
\begin{align*}
\| g_{\varepsilon, k+1} \|_{L_{T_\ast}^\infty X} &\leq \| f_{\varepsilon, 0} \|_X + \sum_{j = 1}^{k+1} \| g_{\varepsilon, j} - g_{\varepsilon, j-1} \|_{L_{T_\ast}^\infty X} \leq \| f_{\varepsilon, 0} \|_X + \bigg( \sum_{j=1}^{k+1} \frac{1}{2^{j-1}} \bigg) \| g_{\varepsilon, 1} - g_{\varepsilon, 0} \|_{L_{T_\ast}^\infty X} \\
&\leq \| f_{\varepsilon, 0} \|_X + 2 \| g_{\varepsilon, 1} - g_{\varepsilon, 0} \|_{L_{T_\ast}^\infty X} \leq C(f_{\varepsilon, 0}, T_\ast).
\end{align*}
This completes the proof of the induction. 
By the contraction mapping principle, estimate (\ref{DfjA}) implies that the sequence $\{ g_{\varepsilon, j} \}_{\mathbf{N}_0}$ is indeed Cauchy in $L^\infty( [0,T_\ast]; H^1( \mathbf{T}_x^3; L^2(\Omega_v) ) )$. Taking the limit as $j \to \infty$, we obtain a local solution to the approximate equation (\ref{AEtoBA}).
\end{proof}
\begin{remark} \label{Ctf:fe}
In the proof of Lemma \ref{GloSol:AE}, the existence of a local solution $f_\varepsilon$ to the approximate equation (\ref{AEtoBA}) is concluded by the Banach fixed point theorem. Hence, for any $t \in [0,T_\ast]$, it holds that 
\[
f_\varepsilon(x,v,t) = \Lambda_\varepsilon\big( f_\varepsilon \big) (x,v,t).
\]
Since we have the global energy estimate (\ref{GloEE:ep}) which is independent of $\varepsilon$, by the standard continuous induction argument, we can extend this local solution $f_\varepsilon$ to a global solution $\widetilde{f_\varepsilon} \in L^\infty( [0,\infty); H^1( \mathbf{T}_x^3; L^2(\Omega_v) ) )$ satisfying
\[
\widetilde{f_\varepsilon}(x,v,t) = \Lambda_\varepsilon\big( \widetilde{f_\varepsilon} \big)(x,v,t), \quad \forall \; t \in [0,\infty).
\]
Hence, without loss of generality, we may always assume that $f_\varepsilon$ satisfies $f_\varepsilon = \Lambda_\varepsilon(f_\varepsilon)$ whenever we consider a global solution $f_\varepsilon$ constructed in Lemma \ref{GloSol:AE}.
\end{remark}
\begin{proof}[Proof of Lemma \ref{GloSol:AE} (Uniqueness):]
To prove the uniqueness of the global solution, we appeal the standard energy method as in the derivation of the global energy estimate.
Let $f_\varepsilon, h_\varepsilon \in L^\infty\big( [0,T]; H^1( \mathbf{T}_x^3; L^2(\Omega_v) ) \big)$ be two global solutions to the approximate equation (\ref{AEtoBA}) satisfying the same initial data
\[
f_\varepsilon \bigm|_{t=0} = \Lambda_\varepsilon(f_{\varepsilon, 0}) = h_\varepsilon \bigm|_{t=0}.
\]
Noting Remark \ref{Ctf:fe}, it can be easily observed that $f_\varepsilon - h_\varepsilon = \Lambda_\varepsilon(f_\varepsilon - h_\varepsilon)$ satisfies the equation
\begin{equation} \label{Uniq:AE}
\begin{split}
&\varepsilon^2 \partial_t (f_\varepsilon - h_\varepsilon) + \varepsilon \Lambda_\varepsilon\big( v \cdot \nabla_x (f_\varepsilon - h_\varepsilon) \big) + \frac{1}{\nu_\ast} \Lambda_\varepsilon\big( \mathcal{L}^\varepsilon(f_\varepsilon - h_\varepsilon) \big) \\
&\ \ = \frac{\varepsilon \kappa}{\nu_\ast} \Lambda_\varepsilon\big( \mathcal{L}^\varepsilon(f_\varepsilon^2 - h_\varepsilon^2) \big) - \frac{\varepsilon^2 \kappa^2}{\nu_\ast} \Lambda_\varepsilon(f_\varepsilon^3 - h_\varepsilon^3)
\end{split}
\end{equation}
with initial data $(f_\varepsilon - h_\varepsilon) \bigm|_{t=0} = 0$.
Taking the inner product of equation (\ref{Uniq:AE}) with $f_\varepsilon - h_\varepsilon$ in the $L^2( \mathbf{T}_x^3; L^2(\Omega_v) )$ sense, we deduce that
\begin{align*}
&\frac{1}{2} \frac{\mathrm{d}}{\mathrm{d} t} \| f_\varepsilon - h_\varepsilon \|_{L^2( \mathbf{T}_x^3; L^2(\Omega_v) )}^2 + \frac{1}{\varepsilon^2 \nu_\ast} \| \mathcal{L}^\varepsilon(f_\varepsilon - h_\varepsilon) \|_{L^2( \mathbf{T}_x^3; L^2(\Omega_v) )}^2 \\
&\ \ = \int_{\mathbf{T}^3} \int_\Omega \frac{\kappa}{\varepsilon \nu_\ast} (f_\varepsilon + h_\varepsilon) (f_\varepsilon - h_\varepsilon) \mathcal{L}^\varepsilon(f_\varepsilon - h_\varepsilon) - \frac{\kappa^2}{\nu_\ast} (f_\varepsilon^2 + f_\varepsilon h_\varepsilon + h_\varepsilon^2) (f_\varepsilon - h_\varepsilon)^2 \, dv \, dx \\
&\ \ \leq \frac{\kappa^2}{2 \nu_\ast} \int_{\mathbf{T}^3} \int_\Omega (f_\varepsilon + h_\varepsilon)^2 (f_\varepsilon - h_\varepsilon)^2 \, dv \, dx + \frac{1}{2 \varepsilon^2 \nu_\ast} \int_{\mathbf{T}^3} \int_\Omega \mathcal{L}^\varepsilon(f_\varepsilon - h_\varepsilon)^2 \, dv \, dx \\
&\ \ \ \ - \frac{\kappa^2}{\nu_\ast} \int_{\mathbf{T}^3} \int_\Omega (f_\varepsilon^2 + f_\varepsilon h_\varepsilon + h_\varepsilon^2) (f_\varepsilon - h_\varepsilon)^2 \, dv \, dx \\
&\ \ \leq \frac{1}{2 \varepsilon^2 \nu_\ast} \int_{\mathbf{T}^3} \int_\Omega \mathcal{L}^\varepsilon(f_\varepsilon - h_\varepsilon)^2 \, dv \, dx
\end{align*}
for any $t \geq 0$.
Hence, by the absorption principle, it holds that
\[
\| f_\varepsilon(t) - h_\varepsilon(t) \|_{L^2( \mathbf{T}_x^3; L^2(\Omega_v) )}^2 + \frac{1}{\varepsilon^2 \nu_\ast} \int_0^t \big\| \mathcal{L}^\varepsilon\big( f_\varepsilon - h_\varepsilon \big) (s) \big\|_{L^2( \mathbf{T}_x^3; L^2(\Omega_v) )}^2 \, ds \leq 0
\]
for any $t \geq 0$, i.e., $f_\varepsilon = h_\varepsilon$ almost everywhere.
This completes the proof of Lemma \ref{GloSol:AE}.
\end{proof}
\begin{remark} \label{SolAE:stg}
Different from the standard theory where one considers the velocity $v$ for all $v \in \mathbf{R}^3$, restricting $v$ to a periodic torus $\Omega$ allows us to control the $X$-norm of $\Lambda_\varepsilon\big( v \cdot \nabla_x \Lambda_\varepsilon(f_\varepsilon) \big)$, see estimate (\ref{XE:vdgg}).
As a result, for a global solution $f_\varepsilon$ constructed in Lemma \ref{GloSol:AE}, we can actually show that $\partial_t f_\varepsilon \in L^\infty( [0,\infty); H^1( \mathbf{T}_x^3; L^2(\Omega_v) ) )$, i.e., $f_\varepsilon$ satisfies the approximate equation (\ref{AEtoBA}) globally in the strong sense.
\end{remark}

As a direct application of Lemma \ref{GloSol:AE}, we have the following implication.

\begin{corollary} \label{Conv:RUT}
Let $\{ \varepsilon_n \}_{n \in \mathbf{N}} \subset (0,1)$ be a sequence that converges to zero as $n \to \infty$.
With respect to each $n$, let $f_{\varepsilon_n}$ be the unique global solution to the approximate Boltzmann equation with anomalous smoothing effect (\ref{AEtoBA}) constructed in Lemma \ref{GloSol:AE} with initial value $\Lambda_\varepsilon(f_{\varepsilon_n, 0})$.
Suppose that
\[
\sup_{n \in \mathbf{N}} \| f_{\varepsilon_n, 0} \|_X < \infty.
\]
Let 
\[
\rho^{\varepsilon_n} = \langle f_{\varepsilon_n}, \mathrm{e}_0^{\varepsilon_n} \rangle_v, \quad u^{\varepsilon_n} = \langle f_{\varepsilon_n}, \mathrm{e}_1^{\varepsilon_n} \rangle_v, \quad \theta^{\varepsilon_n} = \langle f_{\varepsilon_n}, \mathrm{e}_2^{\varepsilon_n} \rangle_v.
\]
Then, for each $n \in \mathbf{N}$, it holds that
\begin{align} \label{GEE:RUT}
\| \rho^{\varepsilon_n} \|_{L^\infty( [0,\infty); H^1(\mathbf{T}_x^3) )} + \| u^{\varepsilon_n} \|_{L^\infty( [0,\infty); H^1(\mathbf{T}_x^3) )} + \| \theta^{\varepsilon_n} \|_{L^\infty( [0,\infty); H^1(\mathbf{T}_x^3) )} \leq \sup_{n \in \mathbf{N}} \| f_{\varepsilon_n, 0} \|_X.
\end{align}
By suppressing subsequences, there exist 
\[
f \in L^\infty( [0,\infty); H^1( \mathbf{T}_x^3; L^2(\Omega_v) ) ) \quad \text{and} \quad \rho, u, \theta \in L^\infty( [0,\infty); H^1(\mathbf{T}_x^3) ) 
\]
such that
\[
f_{\varepsilon_n} \overset{\ast}{\rightharpoonup} f \quad \text{and} \quad (\rho^{\varepsilon_n}, u^{\varepsilon_n}, \theta^{\varepsilon_n}) \overset{\ast}{\rightharpoonup} (\rho, u, \theta) \quad \text{as} \quad n \to \infty.
\]
Moreover, we have the convergence
\begin{align} \label{Conv:Lf}
\mathcal{L}^{\varepsilon_n}(f_{\varepsilon_n}) \to 0 \quad \text{in} \quad L^2( [0,\infty); H^1( \mathbf{T}_x^3; L^2(\Omega_v) ) ) \quad \text{as} \quad n \to \infty.
\end{align}
\end{corollary}
\begin{proof}
Let $n \in \mathbf{N}$.
By considering the equality
\[
\langle f_{\varepsilon_n}, \mathrm{e}_i^{\varepsilon_n} \rangle_v = \big\langle \rho^{\varepsilon_n} \mathrm{e}_0^{\varepsilon_n} + u^{\varepsilon_n} \cdot \mathrm{e}_1^{\varepsilon_n} + \theta^{\varepsilon_n} \mathrm{e}_2^{\varepsilon_n} + \mathcal{L}^{\varepsilon_n} (f_{\varepsilon_n}), \mathrm{e}_i^{\varepsilon_n} \big\rangle_v
\]
for each $1 \leq i \leq 3$, the orthogonality of Legendre polynomials $\{\mathrm{e}_0^{\varepsilon_n}, \mathrm{e}_{1,i}^{\varepsilon_n}, \mathrm{e}_{2,j}^{\varepsilon_n} \}_{1 \leq i,j \leq 3}$ implies that
\[
\big\langle \mathcal{L}^{\varepsilon_n}(f_{\varepsilon_n}), \mathrm{e}_i^{\varepsilon_n} \big\rangle_v = 0 \quad \text{for any} \quad 1 \leq i \leq 3.
\]
As a result, it holds that
\begin{align*}
\int_\Omega f_{\varepsilon_n}^2 \, dv &= \big\langle \rho^{\varepsilon_n} \mathrm{e}_0^{\varepsilon_n} + u^{\varepsilon_n} \cdot \mathrm{e}_1^{\varepsilon_n} + \theta^{\varepsilon_n} \mathrm{e}_2^{\varepsilon_n} + \mathcal{L}^{\varepsilon_n}(f_{\varepsilon_n}), \; \rho^{\varepsilon_n} \mathrm{e}_0^{\varepsilon_n} + u^{\varepsilon_n} \cdot \mathrm{e}_1^{\varepsilon_n} + \theta^{\varepsilon_n} \mathrm{e}_2^{\varepsilon_n} + \mathcal{L}^{\varepsilon_n}(f_{\varepsilon_n}) \big\rangle_v \\
&= (\rho^{\varepsilon_n})^2 + |u^{\varepsilon_n}|^2 + (\theta^{\varepsilon_n})^2 + \mathcal{L}^{\varepsilon_n}(f_{\varepsilon_n})^2.
\end{align*}
Since $\nabla_x$ commutes with $\mathcal{L}^\varepsilon$, we also have that
\[
\int_\Omega (\nabla_x f_{\varepsilon_n})^2 \, dv = |\nabla_x \rho^{\varepsilon_n}|^2 + |\nabla_x u^{\varepsilon_n}|^2 + |\nabla_x \theta^{\varepsilon_n}|^2 + \big| \mathcal{L}^{\varepsilon_n}(\nabla_x f_{\varepsilon_n}) \big|^2.
\]
Hence, by Remark \ref{Ctf:fe}, we observe that estimate (\ref{GEE:RUT}) follows trivially from the global energy estimate (\ref{GloEE:ep}) and Lemma \ref{EXCut}.

By estimate (\ref{GEE:RUT}), the global energy estimate (\ref{GloEE:ep}) and Lemma \ref{EXCut}, it can be observed that sequences
\begin{equation}
\begin{array}{cc}
\big\{ \| f_{\varepsilon_n} \|_{L^\infty( [0,\infty); H^1( \mathbf{T}_x^3; L^2(\Omega_v) ) )} \big\}_{n \in \mathbf{N}}, & \big\{ \| \rho^{\varepsilon_n} \|_{L^\infty( [0,\infty); H^1(\mathbf{T}_x^3) )} \big\}_{n \in \mathbf{N}}, \\
\big\{ \| u^{\varepsilon_n} \|_{L^\infty( [0,\infty); H^1(\mathbf{T}_x^3) )} \big\}_{n \in \mathbf{N}}, & \big\{ \| \theta^{\varepsilon_n} \|_{L^\infty( [0,\infty); H^1(\mathbf{T}_x^3) )} \big\}_{n \in \mathbf{N}}
\end{array}
\end{equation}
are all uniformly bounded.
Since both $L^1( [0,\infty); H^{-1}( \mathbf{T}_x^3; L^2(\Omega_v) ) )$ and $L^1( [0,\infty); H^{-1}(\mathbf{T}_x^3) )$ are separable, by \cite[Corollary 3.30]{Brezis} we can conclude by suppressing subsequences that there exist
\[
f \in L^\infty\big( [0,\infty); H^1( \mathbf{T}_x^3; L^2(\Omega_v) ) \big) \quad \text{and} \quad \rho, u, \theta \in L^\infty( [0,\infty); H^1(\mathbf{T}_x^3) ) 
\]
such that
\[
f_{\varepsilon_n} \overset{\ast}{\rightharpoonup} f \quad \text{and} \quad (\rho^{\varepsilon_n}, u^{\varepsilon_n}, \theta^{\varepsilon_n}) \overset{\ast}{\rightharpoonup} (\rho, u, \theta) \quad \text{as} \quad n \to \infty.
\]
Finally, we also know from the global energy estimate (\ref{GloEE:ep}) and Lemma \ref{EXCut} that
\[
\int_0^\infty \big\| \mathcal{L}^{\varepsilon_n}(f_{\varepsilon_n}) \big\|_X^2 (s) \, ds = \int_0^\infty \mathcal{D}_{\varepsilon_n}\big( f_{\varepsilon_n} \big)^2 (s) \, ds \leq \varepsilon_n^2 \nu_\ast \sup_{n \in \mathbf{N}} \| f_{\varepsilon_n, 0} \|_X \to 0
\]
as $n \to \infty$, we obtain convergence (\ref{Conv:Lf}).
This completes the proof of Corollary \ref{Conv:RUT}.
\end{proof}
\begin{remark} \label{Bdd:RUT}
Following the context of Corollary \ref{Conv:RUT}, we can conclude by the uniform estimate (\ref{GEE:RUT}) that
\begin{align*}
&\| \rho \|_{L^\infty( [0,\infty); H^1(\mathbf{T}_x^3) )} \leq \underset{n \to \infty}{\operatorname{liminf}} \, \| \rho^{\varepsilon_n} \|_{L^\infty( [0,\infty); H^1(\mathbf{T}_x^3) )} \leq \sup_{n \in \mathbf{N}} \| f_{\varepsilon_n, 0} \|_X, \\
&\| u \|_{L^\infty( [0,\infty); H^1(\mathbf{T}_x^3) )} \leq \underset{n \to \infty}{\operatorname{liminf}} \, \| u^{\varepsilon_n} \|_{L^\infty( [0,\infty); H^1(\mathbf{T}_x^3) )} \leq \sup_{n \in \mathbf{N}} \| f_{\varepsilon_n, 0} \|_X, \\
&\| \theta \|_{L^\infty( [0,\infty); H^1(\mathbf{T}_x^3) )} \leq \underset{n \to \infty}{\operatorname{liminf}} \, \| \theta^{\varepsilon_n} \|_{L^\infty( [0,\infty); H^1(\mathbf{T}_x^3) )} \leq \sup_{n \in \mathbf{N}} \| f_{\varepsilon_n, 0} \|_X,
\end{align*}
since $(\rho^{\varepsilon_n}, u^{\varepsilon_n}, \theta^{\varepsilon_n}) \overset{\ast}{\rightharpoonup} (\rho, u, \theta)$ in the weak-$\ast$ topology 
\begin{align*}
\sigma\big( L^\infty( [0,\infty); H^1(\mathbf{T}_x^3) ), L^1( [0,\infty); H^{-1}(\mathbf{T}_x^3) ) \big)
\end{align*}
as $n \to \infty$; see e.g. \cite[Proposition 3.13]{Brezis}.
\end{remark}
\begin{remark} \label{Hire:fep}
We would like to emphasize that the solution $f_\varepsilon$ constructed in Lemma \ref{GloSol:AE} has higher regularity if time is considered locally. Indeed, to prove this claim we apply the differentiation $\partial_x^\alpha$ to the sequence $\{ g_{\varepsilon, j} \}_{j \in \mathbf{N}_0}$. Then, we switch the computation order of $\partial_x^\alpha$ with the microscopic projection $\mathcal{L}^\varepsilon$ and the cutoff operator $\Lambda_\varepsilon$. Since Lemma \ref{Re:BerL} allows us to estimate arbitrary number of times of differentiation of $\{ g_{\varepsilon, j} \}_{j \in \mathbf{N}_0}$, we can follow similar arguments as in the proof of the local existence part of Lemma \ref{GloSol:AE} to show that
\[
f_\varepsilon \in L^\infty( [0,T]; H^N( \mathbf{T}_x^3; L^2(\Omega_v) ) )
\]
for any fixed $T>0$ and $N \geq 1$.
\end{remark}

\section{Formal derivation of the Navier-Stokes-Fourier type system} \label{sec:ConNS} 

Let $\varepsilon \in (0,1)$. For simplicity of notations, we set $v_i^\varepsilon := \Lambda_\varepsilon^{v_i} (v_i)$ for $1 \leq i \leq 3$. Note that $v_i^\varepsilon$ is odd in $v_i$ with respect to the origin.
We let $v_\varepsilon := (v^\varepsilon_1, v^\varepsilon_2, v^\varepsilon_3)$.
For any $h \in L^2(\Omega_v)$, we define the notation $\langle h \rangle$ by 
\[
\langle h \rangle := \int_\Omega h \, dv.
\]
In addition, we define that
\[
v_\varepsilon^2 := \sum_{i=1}^3 \Lambda_\varepsilon^{v_i} (v_i^2).
\]
In contrast with $v_i^\varepsilon$, $\Lambda_\varepsilon^{v_i} (v_i^2)$ is even in $v_i$ with respect to the origin.

The basic setting of this chapter is as follows.
Let $f_0 \in H^1( \mathbf{T}_x^3; L^2(\Omega_v) )$. 
For each $\varepsilon \in (0,1)$, let $f_\varepsilon$ be the unique global solution to the approximate equation (\ref{AEtoBA}) satisfying $f_\varepsilon = \Lambda_\varepsilon(f_\varepsilon)$, $f_\varepsilon \bigm|_{t=0} = \Lambda_\varepsilon(f_0)$ and the global energy inequality
\begin{align*}
\sup_{t \geq 0} \mathcal{E}\big( f_\varepsilon \big)^2 (t) + \frac{1}{\varepsilon^2 \nu_\ast} \int_0^\infty \mathcal{D}_\varepsilon \big( f_\varepsilon \big)^2 (s) \, ds \leq \| f_0 \|_X^2.
\end{align*}
Taking the $L^2$ inner product with respect to $v$ of the approximate equation (\ref{AEtoBA}) with $\mathrm{e}_0^\varepsilon$, $\mathrm{e}_1^\varepsilon$ and $\mathrm{e}_2^\varepsilon$, we obtain that
\begin{eqnarray} \label{ipAE1v}
\left\{
\begin{array}{lcl}
\varepsilon \partial_{t} \big\langle \mathrm{e}_0^\varepsilon f_\varepsilon \big\rangle + \nabla_{x} \cdot \big\langle \mathrm{e}_0^\varepsilon v f_\varepsilon \big\rangle = - \frac{\varepsilon \kappa^2}{\nu_\ast} \big\langle \mathrm{e}_0^\varepsilon f_\varepsilon^3 \big\rangle, \\
\varepsilon \partial_{t} \big\langle \mathrm{e}_1^\varepsilon f_\varepsilon \big\rangle + \nabla_{x} \cdot \big\langle (\mathrm{e}_1^\varepsilon \otimes v) f_\varepsilon \big\rangle= - \frac{\varepsilon \kappa^2}{\nu_\ast} \big\langle \mathrm{e}_1^\varepsilon f_\varepsilon^3 \big\rangle, \\
\varepsilon \partial_t \big\langle \mathrm{e}_2^\varepsilon f_\varepsilon \big\rangle + \nabla_x \cdot \big\langle \mathrm{e}_2^\varepsilon v f_\varepsilon \big\rangle = - \frac{\varepsilon \kappa^2}{\nu_\ast} \big\langle \mathrm{e}_2^\varepsilon f_\varepsilon^3 \big\rangle.
\end{array}
\right.
\end{eqnarray}
Since $\mathrm{e}_0^\varepsilon = 1$ and
\[
\langle v, f_\varepsilon \rangle_v = \langle v, \Lambda_\varepsilon (f_\varepsilon) \rangle_v = \langle \Lambda_\varepsilon (v), f_\varepsilon \rangle_v,
\]
i.e., it holds that $\big\langle \mathrm{e}_0^\varepsilon v f_\varepsilon \big\rangle = \langle v_\varepsilon f_\varepsilon \rangle$.
By definitions of $\rho^\varepsilon$, $u^\varepsilon$ and $\theta^\varepsilon$, we can see that system (\ref{ipAE1v}) can be further rewritten as
\begin{eqnarray} \label{ruipAE}
\left\{
\begin{array}{lcl}
\varepsilon \partial_{t} \rho^\varepsilon + \frac{1}{c_1^\varepsilon} \nabla_{x} \cdot u^\varepsilon = - \frac{\varepsilon \kappa^2}{\nu_\ast} \langle f_\varepsilon^3 \rangle, \\
\varepsilon \partial_{t} u^\varepsilon + c_1^\varepsilon \nabla_{x} \cdot \big\langle (v_\varepsilon \otimes v) f_\varepsilon \big\rangle= - \frac{\varepsilon c_1^\varepsilon \kappa^2}{\nu_\ast} \big\langle v_\varepsilon f_\varepsilon^3 \big\rangle, \\
\varepsilon \partial_t \theta^\varepsilon + \nabla_x \cdot \big\langle \mathrm{e}_2^\varepsilon v f_\varepsilon \big\rangle = - \frac{\varepsilon \kappa^2}{\nu_\ast} \big\langle \mathrm{e}_2^\varepsilon f_\varepsilon^3 \big\rangle.
\end{array}
\right.
\end{eqnarray}

\subsection{Derivation of the divergence free condition} 
\label{sub:BS}

Taking the limit $\varepsilon \to 0$, the first equation of system (\ref{ruipAE}) implies that the limit function of $u^\varepsilon$ in the sense of Corollary \ref{Conv:RUT} is divergence free. 
Suppose that $\psi \in L^\infty( [0,\infty); L^2(\mathbf{T}_x^3) )$ and $\psi \in W^{1,1}( [0,\infty); L^2(\mathbf{T}_x^3) )$.
By the global estimate (\ref{GEE:RUT}), we have that
\begin{align} \label{E2:divfree}
\left| \int_0^\infty \int_{\mathbf{T}^3} \rho^{\varepsilon}(x,t) \big( \partial_t \psi \big) (x,t) \, dx \, dt \right| \leq \| f_0 \|_X \| \partial_t \psi \|_{L^1( [0, \infty); L^2(\mathbf{T}_x^3) )}.
\end{align}
By Lemma \ref{Re:BerL}, we see that
\begin{align} \label{E1:divfree}
\big| \langle f_\varepsilon^3 \rangle \big| \leq \| f_\varepsilon \|_{L^3(\Omega_v)}^3 \lesssim \frac{1}{\varepsilon^{9 \gamma}} \| f_\varepsilon \|_{L^1(\Omega_v)}^3.
\end{align}
Hence, by applying H$\ddot{\text{o}}$lder's inequality, followed by estimate (\ref{E1:divfree}) and then Minkowski's integral inequality and finally the Sobolev embedding of $H^1(\mathbf{T}_x^3)$ in $L^6(\mathbf{T}_x^3)$, it can be deduced that for any $t \in [0,\infty)$,
\begin{equation} \label{E:fep3}
\begin{split}
\left| \int_{\mathbf{T}^3} \langle f_\varepsilon^3 \rangle \psi(x,t) \, dx \right| &\lesssim \frac{1}{\varepsilon^{9 \gamma}} \left( \int_{\mathbf{T}^3} \| f_\varepsilon(x,t) \|_{L^1(\Omega_v)}^6 \, dx \right)^{\frac{1}{2}} \| \psi(t) \|_{L^2(\mathbf{T}_x^3)} \\
&\leq \frac{1}{\varepsilon^{9 \gamma}} \left( \int_{\Omega} \| f_\varepsilon(v,t) \|_{L^6(\mathbf{T}_x^3)} \, dv \right)^3 \| \psi(t) \|_{L^2(\mathbf{T}_x^3)} \\
&\lesssim \frac{1}{\varepsilon^{9 \gamma}} \left( \int_{\Omega} \| f_\varepsilon(v,t) \|_{H^1(\mathbf{T}_x^3)} \, dv \right)^3 \| \psi(t) \|_{L^2(\mathbf{T}_x^3)} \\
&\lesssim \frac{1}{\varepsilon^{9 \gamma}} \| f_0 \|_X^3 \| \psi(t) \|_{L^2(\mathbf{T}_x^3)}.
\end{split}
\end{equation}
Hence, by the first equation of system (\ref{ruipAE}), we deduce that
\begin{equation} \label{Est:tdivu}
\begin{split}
&\left| \int_0^\infty \int_{\mathbf{T}^3} \psi \operatorname{div} u^\varepsilon \, dx \, dt \right| \\
&\ \ \lesssim \varepsilon \left| \int_{\mathbf{T}^3} \rho^\varepsilon(0) \psi(0) \, dx \right| + \varepsilon \left| \int_0^\infty \int_{\mathbf{T}^3} \rho^\varepsilon \partial_t \psi \, dx \, dt \right| + \varepsilon \left| \int_0^\infty \int_{\mathbf{T}^3} \langle f_\varepsilon^3 \rangle \psi(x,t) \, dx \, dt \right| \\
&\ \ \lesssim \varepsilon \| f_0 \|_X \big( \| \psi \|_{L^\infty( [0,\infty); L^2(\mathbf{T}_x^3) )} + \| \partial_t \psi \|_{L^1( [0,\infty); L^2(\mathbf{T}_x^3) )} \big) + \varepsilon^{1 - 9 \gamma} \| f_0 \|_X^3 \| \psi \|_{L^1( [0,\infty); L^2(\mathbf{T}_x^3) )}.
\end{split}
\end{equation}
We take $\gamma < \frac{1}{9}$. Since $u^\varepsilon \overset{\ast}{\rightharpoonup} u$ in $L^\infty( [0,\infty); H^1(\mathbf{T}_x^3) )$, it holds that
\begin{align} \label{divu:cv}
\int_0^\infty \int_{\mathbf{T}^3} \psi \operatorname{div} (u^\varepsilon - u) \, dx \, dt = - \int_0^\infty \int_{\mathbf{T}^3} (u^\varepsilon - u) \cdot \nabla_x \psi \, dx \, dt \to 0
\end{align}
as $\varepsilon \to 0$ since $\nabla_x \psi \in L^1( [0,\infty); H^{-1}(\mathbf{T}^3) )$.
Combining (\ref{Est:tdivu}) with (\ref{divu:cv}), for any 
\[
\psi \in L^\infty( [0,\infty); H^1(\mathbf{T}_x^3) ) \cap W^{1,1}( [0,\infty); H^1(\mathbf{T}_x^3) ),
\]
we conclude that
\begin{align} \label{divu:0}
- \int_0^\infty \int_{\mathbf{T}^3} u \cdot \nabla_x \psi \, dx \, dt = \int_0^\infty \int_{\mathbf{T}^3} \psi \operatorname{div} u \, dx \, dt = 0. 
\end{align}
\begin{remark} \label{divuep:0}
It is easy to observe by (\ref{Est:tdivu}) that
\begin{equation*}
\nabla_x \cdot u^\varepsilon \to 0 \quad \text{as} \quad \varepsilon \to 0
\end{equation*}
in the sense of distributions.
\end{remark}
%

\subsection{Rewriting the key term $\varepsilon^{-1} \nabla_x \cdot \langle (v_\varepsilon \otimes v) f_\varepsilon \rangle$} 
\label{sub:ReMome}

We shall begin with the definition of the cutoff in Fourier space version of matrix $A$.
Note that $\operatorname{ker}(\mathcal{L}^\varepsilon) = \operatorname{span} \{ 1, v_\varepsilon, v_\varepsilon^2 \}$ where
\[
\operatorname{span} \{ 1, v_\varepsilon, v_\varepsilon^2 \} := \{ r_1 + r_2 v_1^\varepsilon + r_3 v_2^\varepsilon + r_4 v_3^\varepsilon + r_5 v_\varepsilon^2 \bigm| r_1, r_2, r_3, r_4, r_5 \in \mathbf{R} \}.
\]

\begin{lemma} \label{Ma:Aep}
Let $\varepsilon>0$ be sufficiently small. Then, there exist constants $a_\varepsilon$ and $b_\varepsilon$ such that the matrix
\[
A_\varepsilon := v_\varepsilon \otimes v - (a_\varepsilon v_\varepsilon^2 + b_\varepsilon) \mathrm{I} \in \mathrm{ker}^\perp(\mathcal{L}^\varepsilon)
\]
where $\mathrm{I}$ denotes the identity matrix and $\mathrm{ker}^\perp(\mathcal{L}^\varepsilon)$ represents the annihilator of $\mathrm{ker}(\mathcal{L}^\varepsilon)$, i.e., we can find constants $a_\varepsilon$ and $b_\varepsilon$ such that
\begin{align} \label{Con:aebe}
\langle A_\varepsilon, 1 \rangle_v = \langle A_\varepsilon, v_\varepsilon^2 \rangle_v = 0 \quad \text{and} \quad \langle A_\varepsilon, v_i^\varepsilon \rangle_v = 0 \quad \forall \quad 1 \leq i \leq 3.
\end{align}
In particular, it holds that
\[
a_\varepsilon \to \frac{1}{3} \quad \text{and} \quad b_\varepsilon \to 0 \quad \text{as} \quad \varepsilon \to 0.
\]
\end{lemma}
\begin{proof}
Note that for $i \neq j$, the function $v_i^\varepsilon v_j \Lambda_\varepsilon (v_k^2)$ is odd for any $1 \leq k \leq 3$, i.e., it holds that
\[
\big\langle v_i^\varepsilon v_j, \Lambda_\varepsilon (v_k^2) \big\rangle_v = 0 \quad \forall \; 1 \leq k \leq 3 \quad \text{for} \quad i \neq j.
\]
Moreover, the value of $\big\langle v_i^\varepsilon v_i, \Lambda_\varepsilon (v_j^2) \big\rangle_v$ is independent of $1 \leq i,j \leq 3$ as long as $i \neq j$ and the value of $\big\langle v_i^\varepsilon v_i, \Lambda_\varepsilon (v_i^2) \big\rangle_v$ is independent of $1 \leq i \leq 3$.
Since $v_i^\varepsilon \Lambda_\varepsilon (v_k^2)$ is odd in $v$ for any $1 \leq k \leq 3$, we have that
\[
\langle A_\varepsilon, v_i^\varepsilon \rangle_v = 0 \quad \forall \; 1 \leq k \leq 3.
\]
Hence, to find constants $a_\varepsilon$ and $b_\varepsilon$ which satisfy relation (\ref{Con:aebe}), it is sufficient to solve the system of equations 
\begin{align*}
3 a_\varepsilon \big( \big\langle \Lambda_\varepsilon (v_i^2), \Lambda_\varepsilon (v_i^2) \big\rangle_v + 2 \big\langle \Lambda_\varepsilon (v_i^2), \Lambda_\varepsilon (v_j^2) \big\rangle_v \big) + 3 b_\varepsilon \big\langle \Lambda_\varepsilon (v_i^2) \big\rangle &= \big\langle v_i^\varepsilon v_i, \Lambda_\varepsilon (v_i^2) \big\rangle_v + 2 \big\langle v_i^\varepsilon v_i, \Lambda_\varepsilon (v_j^2) \big\rangle_v, \\
3 a_\varepsilon \big\langle \Lambda_\varepsilon (v_i^2) \big\rangle + b_\varepsilon &= \langle v_i^\varepsilon v_i \rangle
\end{align*}
simultaneously for $i \neq j$. 

By convergence (\ref{LpE:Cfv}), we can easily deduce that
\begin{align*}
\big| \big\langle \Lambda_\varepsilon (v_i^2), \Lambda_\varepsilon (v_j^2) \big\rangle_v -  \langle v_i^2, v_j^2 \rangle_v \big| &\leq \big| \big\langle \Lambda_\varepsilon (v_i^2) - v_i^2, \Lambda_\varepsilon (v_j^2) \big\rangle_v \big| + \big| \big\langle v_i^2, \Lambda_\varepsilon (v_j^2) - v_j^2 \big\rangle_v \big| \\
&\leq \big\| \Lambda_\varepsilon (v_i^2) - v_i^2 \big\|_{L^2(\Omega_v)} \Big( \big\| \Lambda_\varepsilon (v_i^2) \big\|_{L^2(\Omega_v)} + \| v_i^2 \|_{L^2(\Omega_v)} \Big) \to 0
\end{align*}
and
\begin{align*}
\big| \big\langle \Lambda_\varepsilon (v_i^2) - v_i^2 \big\rangle \big| &\leq \big\| \Lambda_\varepsilon (v_i^2) - v_i^2 \big\|_{L^2(\Omega_v)} \to 0
\end{align*}
as $\varepsilon \to 0$ for any $1 \leq i,j \leq 3$.
We define that
\[
D_\varepsilon :=
\begin{pmatrix}
3 \big\langle \Lambda_\varepsilon (v_i^2), \Lambda_\varepsilon (v_i^2) \big\rangle_v + 6 \big\langle \Lambda_\varepsilon (v_i^2), \Lambda_\varepsilon (v_j^2) \big\rangle_v & 3 \big\langle \Lambda_\varepsilon (v_i^2) \big\rangle \\
3 \big\langle \Lambda_\varepsilon (v_i^2) \big\rangle & 1
\end{pmatrix}. 
\]
Since $\langle v_i^2, v_j^2 \rangle_v = \langle v_i^2 \rangle^2$ for $i \neq j$, we observe that
\begin{align*}
|\mathrm{det} (D_\varepsilon)| &= \big| 3 \big\langle \Lambda_\varepsilon (v_i^2), \Lambda_\varepsilon (v_i^2) \big\rangle_v + 6 \big\langle \Lambda_\varepsilon (v_i^2), \Lambda_\varepsilon (v_j^2) \big\rangle_v - 9 \big\langle \Lambda_\varepsilon (v_i^2) \big\rangle^2 \big| \\
&\to 3 | \langle v_i^2, v_i^2 \rangle_v - \langle v_i^2 \rangle^2 | = \frac{1}{60} \quad \text{as} \quad \varepsilon \to 0
\end{align*}
where $\mathrm{det} (D_\varepsilon)$ denotes the determinant of $D_\varepsilon$, i.e., the matrix $D_\varepsilon$ is invertible for $\varepsilon$ sufficiently small. 
As a result, if $\varepsilon$ is sufficiently small, then constants $a_\varepsilon$ and $b_\varepsilon$ can be uniquely determined by the inversion formula
\[
\binom{a_\varepsilon}{b_\varepsilon} = D_\varepsilon^{-1} \cdot \binom{ \big\langle v_i^\varepsilon v_i, \Lambda_\varepsilon (v_i^2) \big\rangle_v + 2 \big\langle v_i^\varepsilon v_i, \Lambda_\varepsilon (v_j^2) \big\rangle_v}{\langle v_i^\varepsilon v_i \rangle}
\]
where 
\[
D_\varepsilon^{-1} = \frac{1}{\mathrm{det}(D_\varepsilon)}
\begin{pmatrix}
1 & -3 \big\langle \Lambda_\varepsilon (v_i^2) \big\rangle \\
-3 \big\langle \Lambda_\varepsilon (v_i^2) \big\rangle & 3 \big\langle \Lambda_\varepsilon (v_i^2), \Lambda_\varepsilon (v_i^2) \big\rangle_v + 6 \big\langle \Lambda_\varepsilon (v_i^2), \Lambda_\varepsilon (v_j^2) \big\rangle_v
\end{pmatrix}.
\]

For $j \neq i$, we have that
\[
\big\langle v_i^\varepsilon v_i, \Lambda_\varepsilon (v_j^2) \big\rangle_v = \big\langle v_i^\varepsilon v_i, \Lambda_\varepsilon^{v_j} (v_j^2) \big\rangle_v = \big\langle \Lambda_\varepsilon^{v_i} (v_i^2) \big\rangle \langle v_i^\varepsilon v_i \rangle.
\]
Since 
\[
| \langle v_i^\varepsilon v_i - v_i^2 \rangle | \leq \| v_i \|_{L^\infty(\Omega_v)} \| v_i^\varepsilon - v_i \|_{L^2(\Omega_v)} \to 0
\]
as $\varepsilon \to 0$, we deduce by convergence (\ref{LpE:Cfv}) that
\[
\big| \big\langle v_i^\varepsilon v_i, \Lambda_\varepsilon (v_j^2) \big\rangle_v \big| \to \frac{1}{144} \quad \text{as} \quad \varepsilon \to 0.
\]
On the other hand, we have that
\begin{align*}
\big| \big\langle v_i^\varepsilon v_i, \Lambda_\varepsilon (v_i^2) - v_i^2 \big\rangle_v \big| \leq \| v_i \|_{L^\infty(\Omega_v)} \| v_i^\varepsilon \|_{L^2(\Omega_v)} \big\| \Lambda_\varepsilon (v_i^2) - v_i^2 \big\|_{L^2(\Omega_v)}.
\end{align*}
By convergence (\ref{LpE:Cfv}) again, it can be concluded that
\[
\big| \big\langle v_i^\varepsilon v_i, \Lambda_\varepsilon (v_i^2) \big\rangle_v \big| \to | \langle v_i^4 \rangle | = \frac{1}{80} \quad \text{as} \quad \varepsilon \to 0.
\]
As a result, we show that
\[
D_\varepsilon^{-1} \to 
\begin{pmatrix}
60 & -15 \\
-15 & \frac{19}{4}
\end{pmatrix}
\quad \text{and} \quad
\binom{ \big\langle v_i^\varepsilon v_i, \Lambda_\varepsilon (v_i^2) \big\rangle_v + 2 \big\langle v_i^\varepsilon v_i, \Lambda_\varepsilon (v_j^2) \big\rangle_v}{\langle v_i^\varepsilon v_i \rangle} \to \binom{\frac{19}{720}}{\frac{1}{12}}
\]
as $\varepsilon \to 0$, i.e., $a_\varepsilon \to \frac{1}{3}$ and $b_\varepsilon \to 0$ as $\varepsilon \to 0$.
\end{proof}

Next, we decompose $\nabla_x \cdot \big\langle (v_\varepsilon \otimes v) f_\varepsilon \big\rangle$ with respect to matrix $A_\varepsilon$, i.e., we have that
\begin{align*}
\nabla_x \cdot \big\langle (v_\varepsilon \otimes v) f_\varepsilon \big\rangle &= \nabla_x \cdot \langle A_\varepsilon f_\varepsilon \rangle + \nabla_x \cdot \big\langle f_\varepsilon (a_\varepsilon v_\varepsilon^2 + b_\varepsilon) I \big\rangle \\
&= \nabla_x \cdot \langle A_\varepsilon f_\varepsilon \rangle + \frac{a_\varepsilon}{c_2^\varepsilon} \nabla_x \theta^\varepsilon + \bigg( \frac{3 a_\varepsilon c_0^\varepsilon}{c_2^\varepsilon} + b_\varepsilon \bigg) \nabla_x \rho^\varepsilon.
\end{align*}
Since $A_\varepsilon \in \mathrm{ker}^\perp(\mathcal{L}^\varepsilon)$, it holds that
\begin{align} \label{InnP:ALfe}
\frac{1}{\varepsilon} \nabla_x \cdot \langle A_\varepsilon f_\varepsilon \rangle = \nabla_x \cdot \left\langle A_\varepsilon \frac{1}{\varepsilon} \mathcal{L}^\varepsilon(f_\varepsilon) \right\rangle.
\end{align}
In order to rewrite the right hand side of equation (\ref{InnP:ALfe}), we need the following tool.
\begin{lemma} \label{CommuCP}
For $h \in L^2( \mathbf{T}_x^3; L^2(\Omega_v) )$, it holds that
\[
\Lambda_\varepsilon \big( \mathcal{P}^\varepsilon (h) \big) = \mathcal{P}^\varepsilon \big( \Lambda_\varepsilon (h) \big) \quad \text{and} \quad \Lambda_\varepsilon \big( \mathcal{L}^\varepsilon (h) \big) = \mathcal{L}^\varepsilon \big( \Lambda_\varepsilon (h) \big),
\]
i.e., the cutoff in Fourier space operator $\Lambda_\varepsilon$ commutes with both the macroscopic projection $\mathcal{P}^\varepsilon$ and the microscopic projection $\mathcal{L}^\varepsilon$.
\end{lemma}
\begin{proof}
Let $1 \leq i \leq 3$ and $u_i^\varepsilon = \langle \mathrm{e}_{1,i}^\varepsilon h \rangle$.
By a direct calculation, we see that
\begin{align} \label{E1:CommuCP}
\big( (\mathrm{e}_{1,i}^\varepsilon u_i^\varepsilon) \ast D_{\varepsilon^{-1}, v}^3 \big) \ast \mathring{D}_{\varepsilon^{-1},x}^3 = \big( \mathrm{e}_{1,i}^\varepsilon \ast D_{\varepsilon^{-1},v}^3 \big) \big( u_i^\varepsilon \ast \mathring{D}_{\varepsilon^{-1},x}^3 \big).
\end{align}
Since $\mathrm{e}_{1,i}^\varepsilon \ast D_{\varepsilon^{-1},v}^3 = \Lambda_\varepsilon^v(\mathrm{e}_{1,i}^\varepsilon)$ is indeed $\mathrm{e}_{1,i}^\varepsilon$ itself, equality (\ref{E1:CommuCP}) reads as
\[
\Lambda_\varepsilon \big( \mathrm{e}_{1,i}^\varepsilon u_i^\varepsilon \big) = \mathrm{e}_{1,i}^\varepsilon \big( u_i^\varepsilon \ast \mathring{D}_{\varepsilon^{-1},x}^3 \big).
\]
Similarly, we can show that
\[
\Lambda_\varepsilon \big( \mathrm{e}_0^\varepsilon \rho^\varepsilon \big) = \mathrm{e}_0^\varepsilon \big( \rho^\varepsilon \ast \mathring{D}_{\varepsilon^{-1},x}^3 \big) \quad \text{and} \quad \Lambda_\varepsilon \big( \mathrm{e}_2^\varepsilon \theta^\varepsilon \big) = \mathrm{e}_2^\varepsilon \big( \theta^\varepsilon \ast \mathring{D}_{\varepsilon^{-1},x}^3 \big)
\]
with $\rho^\varepsilon = \langle \mathrm{e}_0^\varepsilon h \rangle$ and $\theta^\varepsilon = \langle \mathrm{e}_2^\varepsilon h \rangle$.
Hence, we obtain that
\[
\Lambda_\varepsilon \big( \mathcal{P}^\varepsilon (h) \big) = \mathrm{e}_0^\varepsilon \big( \rho^\varepsilon \ast \mathring{D}_{\varepsilon^{-1},x}^3 \big) + \mathrm{e}_1^\varepsilon \cdot \big( u^\varepsilon \ast \mathring{D}_{\varepsilon^{-1},x}^3 \big) +  \mathrm{e}_2^\varepsilon \big( \theta^\varepsilon \ast \mathring{D}_{\varepsilon^{-1},x}^3 \big).
\]

On the other hand, for $1 \leq i \leq 3$ we have that
\begin{align*}
\int_\Omega \mathrm{e}_{1,i}^\varepsilon \Big( \big( h \ast D_{\varepsilon^{-1},v}^3 \big) \ast \mathring{D}_{\varepsilon^{-1},x}^3 \Big) \, dv &= \int_\Omega \big( \mathrm{e}_{1,i}^\varepsilon \ast D_{\varepsilon^{-1},v}^3 \big) \big( h \ast \mathring{D}_{\varepsilon^{-1},x}^3 \big) \, dv \\
&= \int_\Omega \mathrm{e}_{1,i}^\varepsilon \big( h \ast \mathring{D}_{\varepsilon^{-1},x}^3 \big) \, dv = u_i^\varepsilon \ast \mathring{D}_{\varepsilon^{-1},x}^3.
\end{align*}
Similarly, we can show that 
\[
\int_\Omega \mathrm{e}_0^\varepsilon \Big( \big( h \ast D_{\varepsilon^{-1},v}^3 \big) \ast \mathring{D}_{\varepsilon^{-1},x}^3 \Big) \, dv = \rho^\varepsilon \ast \mathring{D}_{\varepsilon^{-1},x}^3
\]
and
\[
\int_\Omega \mathrm{e}_2^\varepsilon \Big( \big( h \ast D_{\varepsilon^{-1},v}^3 \big) \ast \mathring{D}_{\varepsilon^{-1},x}^3 \Big) \, dv = \theta^\varepsilon \ast \mathring{D}_{\varepsilon^{-1},x}^3.
\]
Hence, we have that
\begin{align*}
\mathcal{P}^\varepsilon \big( \Lambda_\varepsilon (h) \big) &= \mathcal{P}^\varepsilon \Big( \big( h \ast D_{\varepsilon^{-1},v}^3 \big) \ast \mathring{D}_{\varepsilon^{-1},x}^3 \Big) \\
&= \mathrm{e}_0^\varepsilon \big( \rho^\varepsilon \ast \mathring{D}_{\varepsilon^{-1},x}^3 \big) + \mathrm{e}_1^\varepsilon \cdot \big( u^\varepsilon \ast \mathring{D}_{\varepsilon^{-1},x}^3 \big) +  \mathrm{e}_2^\varepsilon \big( \theta^\varepsilon \ast \mathring{D}_{\varepsilon^{-1},x}^3 \big).
\end{align*}
This completes the proof of Lemma \ref{CommuCP}.
\end{proof}

Since $f_\varepsilon = \Lambda_\varepsilon (f_\varepsilon)$, by Lemma \ref{CommuCP} we can easily observe that 
\[
\Lambda_\varepsilon \big( \mathcal{P}^\varepsilon(f_\varepsilon) \big) = \mathcal{P}^\varepsilon(f_\varepsilon) \quad \text{and} \quad \Lambda_\varepsilon \big( \mathcal{L}^\varepsilon(f_\varepsilon) \big) = \mathcal{L}^\varepsilon(f_\varepsilon).
\]
As a result, our approximation equation (\ref{AEtoBA}) can be rewritten as 
\begin{align} \label{Re:Lf}
\mathcal{L}^\varepsilon(f_\varepsilon) = - \varepsilon^2 \nu_\ast \partial_t f_\varepsilon - \varepsilon \nu_\ast \Lambda_\varepsilon \big( v \cdot \nabla_x f_\varepsilon \big) + \varepsilon \kappa \mathcal{L}^\varepsilon\big( \Lambda_\varepsilon (f_\varepsilon^2) \big) - \varepsilon^2 \kappa^2 \Lambda_\varepsilon (f_\varepsilon^3).
\end{align}
By substituting expression (\ref{Re:Lf}) of $\mathcal{L}^\varepsilon(f_\varepsilon)$ into equation (\ref{InnP:ALfe}), we obtain that
\begin{equation} \label{ReApproEq}
\begin{split}
\frac{1}{\varepsilon} \nabla_x \cdot \langle A_\varepsilon f_\varepsilon \rangle &= -\varepsilon \nu_\ast \nabla_x \cdot \partial_t \langle A_\varepsilon f_\varepsilon \rangle - \nu_\ast \nabla_x \cdot \big\langle A_\varepsilon \Lambda_\varepsilon \big( v \cdot \nabla_x f_\varepsilon \big) \big\rangle + \kappa \nabla_x \cdot \big\langle A_\varepsilon  \mathcal{L}^\varepsilon\big( \Lambda_\varepsilon ( f_\varepsilon^2 ) \big) \big\rangle \\
&\ \ - \varepsilon \kappa^2 \nabla_x \cdot \big\langle A_\varepsilon \Lambda_\varepsilon ( f_\varepsilon^3 ) \big\rangle.
\end{split}
\end{equation}
Let $\Phi \in C^\infty\big( [0,T] \times \mathbf{T}^3 \big)$ with $0<T<\infty$.
Since
\[
\big| \big\langle A_\varepsilon f_\varepsilon \big\rangle \big| \leq \| A_\varepsilon \|_{L^2(\Omega_v)} \| f_\varepsilon \|_{L^2(\Omega_v)},
\]
by the global energy estimate (\ref{GloEE:ep}) we have that
\begin{align*}
\left| \int_0^T \int_{\mathbf{T}^3} \Big( \nabla_x \cdot \partial_t \big\langle A_\varepsilon f_\varepsilon \big\rangle \Big) \Phi \, dx \, dt \right| \leq T^{\frac{1}{2}} \| A_\varepsilon \|_{L^2(\Omega_v)} \| \partial_t \nabla_x \Phi \|_{L^2( [0,T]; L^2(\mathbf{T}_x^3) )} \| f_0 \|_X. 
\end{align*}
Since $\| A_\varepsilon \|_{L^2(\Omega_v)}$ converges to $\| A \|_{L^2(\Omega_v)}$ as $\varepsilon \to 0$, we can conclude that the term $\varepsilon \nu_\ast \nabla_x \cdot \partial_t \big\langle A_\varepsilon f_\varepsilon \big\rangle$ converges to zero in the sense of distributions as $\varepsilon \to 0$.
Similarly, by Plancherel's identity (see e.g. \cite[Proposition 3.2.7]{GraC}), we can deduce that
\begin{align*}
&\left| \int_0^T \int_{\mathbf{T}^3} \Big( \nabla_x \cdot \big\langle A_\varepsilon \Lambda_\varepsilon(f_\varepsilon^3) \big\rangle \Big) \Phi \, dx \, dt \right| \\
&\ \ \leq \| A_\varepsilon \|_{L^2(\Omega_v)} \| \nabla_x \Phi \|_{L^2( [0,T]; L^2(\mathbf{T}_x^3) )} \left( \int_0^T \| f_\varepsilon^3 \|_{L^2( \mathbf{T}_x^3; L^2(\Omega_v) )}^2 \, dt \right)^{\frac{1}{2}}. 
\end{align*}
By Lemma \ref{Re:BerL} and Minkowski's integral inequality, we have that
\begin{align*}
\| f_\varepsilon^3 \|_{L^2( \mathbf{T}_x^3; L^2(\Omega_v) )} &= \left( \int_\Omega \| f_\varepsilon \|_{L^6(\mathbf{T}_x^3)}^6 \, dv \right)^{\frac{1}{2}} \lesssim \varepsilon^{-9 \gamma} \left( \int_\Omega \| f_\varepsilon \|_{L^2(\mathbf{T}_x^3)}^6 \, dv \right)^{\frac{1}{2}} \\
&\leq \varepsilon^{- 9 \gamma} \left( \int_{\mathbf{T}^3} \| f_\varepsilon \|_{L^6(\Omega_v)}^2 \, dx \right)^{\frac{3}{2}} \lesssim \varepsilon^{-18 \gamma} \| f_\varepsilon \|_{L^2( \mathbf{T}_x^3; L^2(\Omega_v) )}^3.
\end{align*}
Hence, by the global energy estimate (\ref{GloEE:ep}), we obtain that
\[
\left( \int_0^T \| f_\varepsilon^3 \|_{L^2( \mathbf{T}_x^3; L^2(\Omega_v) )}^2 \, dt \right)^{\frac{1}{2}} \lesssim T^{\frac{1}{2}} \varepsilon^{-18 \gamma} \| f_0 \|_X^3.
\]
If $\gamma < \frac{1}{18}$, then the term $\varepsilon \kappa^2 \nabla_x \cdot \big\langle A_\varepsilon \Lambda_\varepsilon(f_\varepsilon^3) \big\rangle$ converges to zero in the sense of distributions as $\varepsilon \to 0$.
As a result, the convergence behavior of $\varepsilon^{-1} \nabla_x \cdot \big\langle A_\varepsilon f_\varepsilon \big\rangle$ is mainly governed by the middle two terms on the right hand side of equation (\ref{ReApproEq}).

\subsection{Derivation of the diffusion term $\Delta_x u^\varepsilon$} 
\label{Sub:Diff}

In this section, we shall investigate the convergence behavior of the term 
\[
\nabla_x \cdot \big\langle A_\varepsilon \Lambda_\varepsilon \big( v \cdot \nabla_x f_\varepsilon \big) \big\rangle.
\]
Since 
\[
\big\langle A_\varepsilon \Lambda_\varepsilon \big( v \cdot \nabla_x f_\varepsilon \big) \big\rangle = \big\langle A_\varepsilon, \Lambda_\varepsilon \big( v \cdot \nabla_x f_\varepsilon \big) \big\rangle_v = \big\langle \Lambda_\varepsilon^v  (A_\varepsilon), v \cdot \nabla_x f_\varepsilon \big\rangle_v,
\]
it is sufficient to consider the convergence behavior of $\nabla_x \cdot \big\langle \Lambda_\varepsilon^v (A_\varepsilon) \big( v \cdot \nabla_x f_\varepsilon \big) \big\rangle$.
By further decomposing $f_\varepsilon$ into the sum of $\mathcal{P}^\varepsilon(f_\varepsilon)$ and $\mathcal{L}^\varepsilon(f_\varepsilon)$, we see that the main contribution comes from the term $\nabla_x \cdot \big\langle \Lambda_\varepsilon^v (A_\varepsilon) \big( v \cdot \nabla_x \mathcal{P}^\varepsilon(f_\varepsilon) \big) \big\rangle$ since the term $\mathcal{L}^\varepsilon(f_\varepsilon)$ is of order $\mathcal{O}(\varepsilon)$ according to the global energy estimate (\ref{GloEE:ep}).
Indeed, by the Plancherel's identity (see e.g. \cite[Proposition 3.2.7.(1)]{GraC}), we have by convergence (\ref{LpE:Cfv}) that
\begin{align} \label{L2:Ctf:vev}
\big\| \Lambda_\varepsilon^v (v_i^\varepsilon v_i) - \Lambda_\varepsilon^v (v_i^2) \big\|_{L^2(\Omega_v)} \leq \| v_i^\varepsilon v_i - v_i^2 \|_{L^2(\Omega_v)} \leq \| v_i^\varepsilon - v_i \|_{L^4(\Omega_v)} \| v_i \|_{L^4(\Omega_v)} \to 0
\end{align}
as $\varepsilon \to 0$. 
Hence, by manipulating H$\ddot{\text{o}}$lder's inequality and the triangle inequality, we can deduce that
\begin{align} \label{L2:Ctf:Ae}
\big\| \Lambda_\varepsilon^v(A_\varepsilon) - A \big\|_{L^2(\Omega_v)} \to 0 \quad \text{as} \quad \varepsilon \to 0.
\end{align}
As a result, it holds that
\begin{align*}
\big| \big\langle \Lambda_\varepsilon^v (A_\varepsilon) \big( v \cdot \nabla_x \mathcal{L}^\varepsilon(f_\varepsilon) \big) \big\rangle \big| \leq \| v \|_{L^\infty(\Omega_v)} \big\| \Lambda_\varepsilon^v(A_\varepsilon) \big\|_{L^2(\Omega_v)} \| \nabla_x \mathcal{L}^\varepsilon(f_\varepsilon) \|_{L^2(\Omega_v)}.
\end{align*}
Then, by the global energy estimate (\ref{GloEE:ep}), we deduce that
\begin{align*}
&\left| \int_0^T \int_{\mathbf{T}^3} \Big( \nabla_x \cdot \big\langle \Lambda_\varepsilon^v (A_\varepsilon) \big( v \cdot \nabla_x \mathcal{L}^\varepsilon(f_\varepsilon) \big) \big\rangle \Big) \Phi \, dx \, dt \right| \\
&\ \ \lesssim \left( \int_0^T \mathcal{D}_\varepsilon\big( f_\varepsilon \big)^2 (s) \, ds \right)^{\frac{1}{2}} \| \nabla_x \Phi \|_{L^2( [0,T]; L^2(\mathbf{T}_x^3) )} \leq \varepsilon \sqrt{\nu_\ast} \| f_0 \|_X \| \nabla_x \Phi \|_{L^2( [0,T]; L^2(\mathbf{T}_x^3) )}
\end{align*}
for any $\Phi \in C^\infty( [0,T] \times \mathbf{T}^3 )$ with $0<T<\infty$, i.e., the term $\nabla_x \cdot \big\langle \Lambda_\varepsilon^v (A_\varepsilon) \big( v \cdot \nabla_x \mathcal{L}^\varepsilon(f_\varepsilon) \big) \big\rangle$ converges to zero in the sense of distributions as $\varepsilon \to 0$.

\begin{lemma}[(Derivation of the diffusion term)] \label{Diff}
For $1 \leq i \leq 3$, it holds that
\begin{align*}
\left\{ \nabla_x \cdot \big\langle \Lambda_\varepsilon^v (A_\varepsilon) \big( v \cdot \nabla_x \mathcal{P}^\varepsilon(f_\varepsilon) \big) \big\rangle \right\}_i &= \frac{\sqrt{3}}{72} \Delta_x u_i^\varepsilon - \frac{\sqrt{3}}{60} \partial_{x_i}^2 u_i^\varepsilon + D_{ns, i}^\varepsilon
\end{align*}
where $D_{ns, i}^\varepsilon$ is some remainder function of $x$ and $t$ which converges to zero in the sense of distributions as $\varepsilon \to 0$, i.e., for any $0<T<\infty$ and $\Phi \in C^\infty([0,T] \times \mathbf{T}^3)$, it holds that
\[
\left| \int_0^T \int_{\mathbf{T}^3} D_{ns,i}^\varepsilon (x,t) \Phi(x,t) \, dx \, dt \right| \to 0 \quad \text{as} \quad \varepsilon \to 0.
\]
\end{lemma}
\begin{proof}
By the definition of the macroscopic projection $\mathcal{P}^\varepsilon$, we have the expansion
\begin{align*}
v \cdot \nabla_x \mathcal{P}^\varepsilon(f_\varepsilon) &= v \cdot \nabla_x \big( \rho^\varepsilon \mathrm{e}_0^\varepsilon + u^\varepsilon \cdot \mathrm{e}_1^\varepsilon + \theta^\varepsilon \mathrm{e}_2^\varepsilon \big) \\
&= (\mathrm{e}_0^\varepsilon v) \cdot \nabla_x \rho^\varepsilon + c_1^\varepsilon \sum_{i,j} v_i v_j^\varepsilon \partial_i u_j^\varepsilon + (\mathrm{e}_2^\varepsilon v) \cdot \nabla_x \theta^\varepsilon \\
&= (\mathrm{e}_0^\varepsilon v) \cdot \nabla_x \rho^\varepsilon + c_1^\varepsilon (a_\varepsilon v_\varepsilon^2 + b_\varepsilon) \sum_{ij} \delta_{ij} \partial_i u_j^\varepsilon + c_1^\varepsilon \sum_{ij} (A_\varepsilon)_{ij} \partial_i u_j^\varepsilon + (\mathrm{e}_2^\varepsilon v) \cdot \nabla_x \theta^\varepsilon.
\end{align*}
Since $\mathrm{e}_0^\varepsilon = 1$ and $A_\varepsilon \in \mathrm{ker}^\perp(\mathcal{L}^\varepsilon)$, it holds that
\begin{align*}
&\Big\langle \Lambda_\varepsilon^v (A_\varepsilon), (\mathrm{e}_0^\varepsilon v) \cdot \nabla_x \rho^\varepsilon + c_1^\varepsilon (a_\varepsilon v_\varepsilon^2 + b_\varepsilon) \sum_{ij} \delta_{ij} \partial_i u_j^\varepsilon \Big\rangle_v \\
&\ \ = \Big\langle A_\varepsilon, v_\varepsilon \cdot \nabla_x \rho^\varepsilon + c_1^\varepsilon (a_\varepsilon v_\varepsilon^2 + b_\varepsilon) \sum_{ij} \delta_{ij} \partial_i u_j^\varepsilon \Big\rangle_v = 0.
\end{align*}
Since $\Lambda_\varepsilon^v (v_i^\varepsilon v_j) = v_i^\varepsilon v_j^\varepsilon$ for $i \neq j$, it can be easily deduced that $\Lambda_\varepsilon^v (v_i^\varepsilon v_j)$ is odd in $v$-variable for $i \neq j$.
On the other hand, since $v_i^\varepsilon$ is odd in $v_i$, we have that $v_i^\varepsilon v_i$ is even in $v_i$ and thus $\Lambda_\varepsilon^v (v_i^\varepsilon v_i)$ is even in $v$-variable.
As a result, we deduce that $\Lambda_\varepsilon^v(A_\varepsilon)_{ij} \mathrm{e}_2^\varepsilon v_k$ is always odd in $v$-variable for any $1 \leq i,j,k \leq 3$, i.e., it holds that
\[
\big\langle \Lambda_\varepsilon^v (A_\varepsilon), (\mathrm{e}_2^\varepsilon v) \cdot \nabla_x \theta^\varepsilon \big\rangle_v = 0.
\]
Hence, for any $1 \leq i \leq 3$, it holds that
\begin{align*}
\big\{ \nabla_x \cdot \big\langle \Lambda_\varepsilon^v (A_\varepsilon) \big( v \cdot \nabla_x \mathcal{P}^\varepsilon(f_\varepsilon) \big) \big\rangle \big\}_i = c_1^\varepsilon \sum_{j,k,\ell} \big\langle \Lambda_\varepsilon^v (A_\varepsilon)_{ij} (A_\varepsilon)_{k \ell} \big\rangle \partial_{x_j} \partial_{x_k} u_{\ell}^\varepsilon. 
\end{align*}

Considering the parity of $\Lambda_\varepsilon^v (A_\varepsilon)$ that we have discussed in the above paragraph, we can observe that $\Lambda_\varepsilon^v (A_\varepsilon)_{ij} (A_\varepsilon)_{k \ell}$ is odd in $v$-variable if any index among $i,j,k,\ell$ appears for an odd number of times.
Hence, we can deduce that $\big\langle \Lambda_\varepsilon^v (A_\varepsilon)_{ij} (A_\varepsilon)_{k \ell} \big\rangle$ does not vanish only in the following four cases where
\[
i=j=k=\ell \quad \text{or} \quad i=j, k=\ell, j \neq k \quad \text{or} \quad i=k,j=\ell, k \neq j \quad \text{or} \quad i=\ell,j=k, \ell \neq j,
\]
i.e., it holds that
\begin{equation} \label{E1:Diff}
\begin{split}
&\sum_{j,k,\ell} \big\langle \Lambda_\varepsilon^v (A_\varepsilon)_{ij} (A_\varepsilon)_{k \ell} \big\rangle \partial_{x_j} \partial_{x_k} u_{\ell}^\varepsilon \\
&\ \ =  \big\langle \Lambda_\varepsilon^v (A_\varepsilon)_{ii} (A_\varepsilon)_{ii} \big\rangle \partial_{x_i}^2 u_i^\varepsilon + \sum_{s \neq i} \big\langle \Lambda_\varepsilon^v (A_\varepsilon)_{ii} (A_\varepsilon)_{ss} \big\rangle \partial_{x_i} \partial_{x_s} u_s^\varepsilon \\
&\ \ \ \ + \sum_{s \neq i} \big\langle v_i^\varepsilon v_s^\varepsilon v_i^\varepsilon v_s \big\rangle \partial_{x_s} \partial_{x_i} u_s^\varepsilon + \sum_{s \neq i} \big\langle v_i^\varepsilon v_s^\varepsilon v_s^\varepsilon v_i \big\rangle \partial_{x_s}^2 u_i^\varepsilon.
\end{split}
\end{equation}
Since Lemma \ref{Ma:Aep} guarantees that $a_\varepsilon \to \frac{1}{3}$ and $b_\varepsilon \to 0$ as $\varepsilon \to 0$, by considering the difference with matrix $A$, we can rewrite the right hand side of equation (\ref{E1:Diff}) as 
\[
\sum_{j,k,\ell} \big\langle \Lambda_\varepsilon^v (A_\varepsilon)_{ij} (A_\varepsilon)_{k \ell} \big\rangle \partial_{x_j} \partial_{x_k} u_{\ell}^\varepsilon = \frac{1}{144} \Delta_x u_i^\varepsilon - \frac{1}{120} \partial_{x_i}^2 u_i^\varepsilon + D_{ns,i}^\varepsilon
\]
where 
\begin{align*}
D_{ns,i}^\varepsilon &= \langle A_{ii} A_{ss} \rangle \partial_{x_i} \big( \nabla_x \cdot u^\varepsilon \big) + \big\langle v_i^2 v_s^2 \big\rangle \partial_{x_i} \big( \nabla_x \cdot u^\varepsilon \big) + \big\langle \Lambda_\varepsilon^v (A_\varepsilon)_{ii} (A_\varepsilon)_{ii} - A_{ii} A_{ii} \big\rangle \partial_{x_i} \partial_{x_i} u_i^\varepsilon \\
&\ \ + \sum_{s \neq i} \big\langle \Lambda_\varepsilon^v (A_\varepsilon)_{ii} (A_\varepsilon)_{ss} - A_{ii} A_{ss} \big\rangle \partial_{x_i} \partial_{x_s} u_s^\varepsilon + \sum_{s \neq i} \big\langle (v_i^\varepsilon)^2 v_s^\varepsilon v_s - v_i^2 v_s^2 \big\rangle \partial_{x_s} \partial_{x_i} u_s^\varepsilon \\
&\ \ + \sum_{s \neq i} \big\langle (v_s^\varepsilon)^2 v_i^\varepsilon v_i - v_i^2 v_s^2 \big\rangle \partial_{x_s}^2 u_i^\varepsilon.
\end{align*}

Manipulating with the triangle inequality and H$\ddot{\text{o}}$lder's inequality, we can show by convergence (\ref{LpE:Cfv}) and (\ref{L2:Ctf:Ae}) that if $i \neq s$, then all of the following three terms
\[
\big| \big\langle \Lambda_\varepsilon^v (A_\varepsilon)_{ii} (A_\varepsilon)_{ii} - A_{ii} A_{ii} \big\rangle \big|, \quad \big| \big\langle \Lambda_\varepsilon^v (A_\varepsilon)_{ii} (A_\varepsilon)_{ss} - A_{ii} A_{ss} \big\rangle \big|, \quad \big| \big\langle (v_i^\varepsilon)^2 v_s^\varepsilon v_s - v_i^2 v_s^2 \big\rangle \big|
\]
converges to zero as $\varepsilon \to 0$.
Since $\| u^\varepsilon \|_{L^\infty( [0,\infty); H^1(\mathbf{T}_x^3) )}$ is controlled by $\| f_0 \|_X$ according to Corollary \ref{Conv:RUT}, for any fixed $T>0$, it can be easily deduced that
\[
\left| \int_0^T \int_{\mathbf{T}^3} D_{ns,i}^\varepsilon (x,t) \Phi(x,t) \, dx \,dt \right| \to 0 \quad \text{as} \quad \varepsilon \to 0
\]
for any $\Phi \in C^\infty([0,T] \times \mathbf{T}^3)$, i.e., $D_{ns,i}^\varepsilon$ is a remainder term that converges to zero as $\varepsilon \to 0$ in the sense of distributions. Therefore, we can conclude that 
\[
\big\{ \nabla_x \cdot \big\langle \Lambda_\varepsilon^v (A_\varepsilon) \big( v \cdot \nabla_x \mathcal{P}^\varepsilon(f_\varepsilon) \big) \big\rangle \big\}_i
\]
shares the same limit with $\frac{\sqrt{3}}{72} \Delta_x u_i^\varepsilon - \frac{\sqrt{3}}{60} \partial_{x_i}^2 u_i^\varepsilon$ as $\varepsilon \to 0$ for any $1 \leq i \leq 3$.
This completes the proof of Lemma \ref{Diff}.
\end{proof}

\subsection{Derivation of the transportation term $u^\varepsilon \cdot \nabla_x u^\varepsilon$} 
\label{Sub:Nonli}

In this section, we shall investigate the convergence behavior of the term 
\[
\nabla_x \cdot \big\langle A_\varepsilon  \mathcal{L}^\varepsilon\big( \Lambda_\varepsilon ( f_\varepsilon^2 ) \big) \big\rangle.
\]
Since $A_\varepsilon \in \mathrm{ker}^\perp(\mathcal{L}^\varepsilon)$ (see Lemma \ref{Ma:Aep}), it holds that
\[
\big\langle A_\varepsilon \mathcal{L}^\varepsilon\big( \Lambda_\varepsilon(f_\varepsilon^2) \big) \big\rangle = \langle A_\varepsilon \Lambda_\varepsilon(f_\varepsilon^2) \rangle.
\]
Hence, it is sufficient to investigate the convergence behavior of $\nabla_x \cdot \big\langle \Lambda_\varepsilon(A_\varepsilon) f_\varepsilon^2 \big\rangle$.
We next decompose 
\[
f_\varepsilon^2 = \big( \mathcal{P}^\varepsilon(f_\varepsilon) + \mathcal{L}^\varepsilon(f_\varepsilon) \big)^2 = \mathcal{P}^\varepsilon(f_\varepsilon)^2 + 2 \mathcal{P}^\varepsilon(f_\varepsilon) \mathcal{L}^\varepsilon(f_\varepsilon) + \mathcal{L}^\varepsilon(f_\varepsilon)^2.
\]
Since $\mathcal{L}^\varepsilon(f_\varepsilon)$ is of order $\mathcal{O}(\varepsilon)$ according to the global energy estimate (\ref{GloEE:ep}), the main contribution is governed by
\[
\nabla_x \cdot \big\langle \Lambda_\varepsilon(A_\varepsilon) \mathcal{P}^\varepsilon(f_\varepsilon)^2 \big\rangle
\]
provided that $\varepsilon$ is sufficiently small.
Indeed, by Lemma \ref{Re:BerL}, we observe that
\begin{align*}
\big| \big\langle \Lambda_\varepsilon(A_\varepsilon) \mathcal{P}^\varepsilon(f_\varepsilon) \mathcal{L}^\varepsilon(f_\varepsilon) \big\rangle \big| &\leq \big\| \Lambda_\varepsilon^v(A_\varepsilon) \big\|_{L^\infty(\Omega_v)} \| \mathcal{P}^\varepsilon(f_\varepsilon) \|_{L^2(\Omega_v)} \| \mathcal{L}^\varepsilon(f_\varepsilon) \|_{L^2(\Omega_v)} \\
&\lesssim \varepsilon^{-3 \gamma} \| \Lambda_\varepsilon^v(A_\varepsilon) \|_{L^2(\Omega_v)} \| \mathcal{P}^\varepsilon(f_\varepsilon) \|_{L^2(\Omega_v)} \| \mathcal{L}^\varepsilon(f_\varepsilon) \|_{L^2(\Omega_v)} 
\end{align*}
and
\begin{align*}
\big| \big\langle \Lambda_\varepsilon(A_\varepsilon) \mathcal{L}^\varepsilon(f_\varepsilon)^2 \big\rangle \big| &\leq \big\| \Lambda_\varepsilon^v(A_\varepsilon) \big\|_{L^\infty(\Omega_v)} \| \mathcal{L}^\varepsilon(f_\varepsilon) \|_{L^2(\Omega_v)}^2 \\
&\lesssim \varepsilon^{-3 \gamma} \| \Lambda_\varepsilon^v(A_\varepsilon) \|_{L^2(\Omega_v)} \| \mathcal{L}^\varepsilon(f_\varepsilon) \|_{L^2(\Omega_v)}^2.
\end{align*}
Thus, by the global energy estimate (\ref{GloEE:ep}), we can deduce that
\begin{align*}
&\left| \int_0^T \int_{\mathbf{T}^3} \Big( \nabla_x \cdot \big\langle \Lambda_\varepsilon(A_\varepsilon) \mathcal{P}^\varepsilon(f_\varepsilon) \mathcal{L}^\varepsilon(f_\varepsilon) \big\rangle \Big) \Phi \, dx \, dt \right| \\
&\ \ \lesssim \varepsilon^{-3 \gamma} \sqrt{T} \| \Lambda_\varepsilon^v(A_\varepsilon) \|_{L^2(\Omega_v)} \sup_{t \in [0,T]} \mathcal{E} \big( f_\varepsilon \big) (t) \left( \int_0^T \mathcal{D}_\varepsilon\big( f_\varepsilon \big)^2 (t) \, dt \right)^{\frac{1}{2}} \| \nabla_x \Phi \|_{L^\infty( [0,T]; L^\infty(\mathbf{T}_x^3) )} \\
&\ \ \leq \varepsilon^{1 - 3 \gamma} \sqrt{\nu_\ast T} \| \Lambda_\varepsilon^v(A_\varepsilon) \|_{L^2(\Omega_v)} \| f_0 \|_X^2 \| \nabla_x \Phi \|_{L^\infty( [0,T]; L^\infty(\mathbf{T}_x^3) )}
\end{align*}
and
\begin{align*}
&\left| \int_0^T \int_{\mathbf{T}^3} \Big( \nabla_x \cdot \big\langle \Lambda_\varepsilon(A_\varepsilon) \mathcal{L}^\varepsilon(f_\varepsilon)^2 \big\rangle \Big) \Phi \, dx \, dt \right| \\
&\ \ \lesssim \varepsilon^{-3 \gamma} \| \Lambda_\varepsilon^v(A_\varepsilon) \|_{L^2(\Omega_v)} \int_0^T \mathcal{D}_\varepsilon\big( f_\varepsilon \big)^2 (t) \, dt \| \nabla_x \Phi \|_{L^\infty( [0,T]; L^\infty(\mathbf{T}_x^3) )} \\
&\ \ \leq \varepsilon^{2 - 3 \gamma} \nu_\ast \| \Lambda_\varepsilon^v(A_\varepsilon) \|_{L^2(\Omega_v)} \| f_0 \|_X^2 \| \nabla_x \Phi \|_{L^\infty( [0,T]; L^\infty(\mathbf{T}_x^3) )}
\end{align*}
for any $\Phi \in C^\infty( [0,T] \times \mathbf{T}_x^3 )$ with $0<T<\infty$.
Recalling convergence (\ref{L2:Ctf:Ae}), we can conclude that if $\gamma < \frac{2}{3}$, then both 
\[
\nabla_x \cdot \big\langle \Lambda_\varepsilon(A_\varepsilon) \mathcal{P}^\varepsilon(f_\varepsilon) \mathcal{L}^\varepsilon(f_\varepsilon) \big\rangle \quad \text{and} \quad \nabla_x \cdot \big\langle \Lambda_\varepsilon(A_\varepsilon) \mathcal{L}^\varepsilon(f_\varepsilon)^2 \big\rangle
\]
converges to zero in the sense of distributions as $\varepsilon \to 0$.

\begin{lemma}[Derivation of nonlinear terms] \label{Dtp}
For $1 \leq i \leq 3$, it holds that
\[
\left\{ \nabla_x \cdot \big\langle \Lambda_\varepsilon(A_\varepsilon) \mathcal{P}^\varepsilon (f_\varepsilon)^2 \big\rangle \right\}_i = \frac{1}{6} \left\{ \nabla_x \cdot ( u^\varepsilon \otimes u^\varepsilon ) \right\}_i - \frac{1}{10} \partial_{x_i} (u_i^\varepsilon)^2 + \frac{2}{5} \partial_{x_i} |u^\varepsilon|^2 + N_{ns,i}^\varepsilon
\]
where $N_{ns, i}^\varepsilon$ is some remainder function of $x$ and $t$ which converges to zero in the sense of distributions as $\varepsilon \to 0$, i.e., for any $0<T<\infty$ and $\Phi \in C^\infty([0,T] \times \mathbf{T}^3)$, it holds that
\[
\left| \int_0^T \int_{\mathbf{T}^3} N_{ns,i}^\varepsilon (x,t) \Phi(x,t) \, dx \, dt \right| \to 0 \quad \text{as} \quad \varepsilon \to 0.
\]
\end{lemma}
\begin{proof}
Firstly, we expand $\mathcal{P}^\varepsilon(f_\varepsilon)^2$ by considering the definition of $\mathcal{P}^\varepsilon(f_\varepsilon)$. Noting that $\mathrm{e}_0^\varepsilon = 1$, we have that
\begin{equation} \label{Exp:Pf2}
\begin{split}
\mathcal{P}^\varepsilon(f_\varepsilon)^2 &= \Big( \rho^\varepsilon \mathrm{e}_0^\varepsilon + \sum_i u_i^\varepsilon \mathrm{e}_{1,i}^\varepsilon + \theta^\varepsilon \mathrm{e}_2^\varepsilon \Big)^2 \\
&= (c_1^\varepsilon)^2 \sum_{i,j} (A_\varepsilon)_{ij} u_i^\varepsilon u_j^\varepsilon + \big( \theta^\varepsilon \mathrm{e}_2^\varepsilon \big)^2 + 2 \mathrm{e}_2^\varepsilon \theta^\varepsilon (\mathrm{e}_1^\varepsilon \cdot u^\varepsilon) + J^\varepsilon + K^\varepsilon
\end{split}
\end{equation}
where 
\begin{align*}
J^\varepsilon := (\rho^\varepsilon)^2 + 2 \rho^\varepsilon (\mathrm{e}_1^\varepsilon \cdot u^\varepsilon) + 2 \rho^\varepsilon \theta^\varepsilon \mathrm{e}_2^\varepsilon + (c_1^\varepsilon)^2 \left( \frac{a_\varepsilon}{c_2^\varepsilon} \mathrm{e}_2^\varepsilon + \frac{3  c_0^\varepsilon a_\varepsilon}{c_2^\varepsilon} + b_\varepsilon \right) |u^\varepsilon|^2
\end{align*}
and
\begin{align*}
K^\varepsilon := (c_1^\varepsilon)^2 \sum_{i \neq j} v_i^\varepsilon (v_j^\varepsilon - v_j) u_i^\varepsilon u_j^\varepsilon.
\end{align*}
It is easy to see that $J^\varepsilon$ is a linear combination of $1$, $v_i^\varepsilon$ and $v_\varepsilon^2$, i.e., $J^\varepsilon \in \mathrm{ker}(\mathcal{L}^\varepsilon)$.
Since $A_\varepsilon$ is a matrix in $v$ only, it holds that
\[
\big\langle \Lambda_\varepsilon(A_\varepsilon) J^\varepsilon \big\rangle = \big\langle \Lambda_\varepsilon^v(A_\varepsilon) J^\varepsilon \big\rangle = \big\langle A_\varepsilon, \Lambda_\varepsilon^v(J^\varepsilon) \big\rangle_v = \big\langle A_\varepsilon, J^\varepsilon \big\rangle_v = 0.
\]
Moreover, for $i \neq j$ we observe that
\begin{align*}
\big\langle \Lambda_\varepsilon(A_\varepsilon) v_i^\varepsilon (v_j^\varepsilon - v_j) \big\rangle &= \big\langle \Lambda_\varepsilon(A_\varepsilon), v_i^\varepsilon (v_j^\varepsilon - v_j) \big\rangle_v = \big\langle A_\varepsilon, \Lambda_\varepsilon \big( v_i^\varepsilon (v_j^\varepsilon - v_j) \big) \big\rangle_v \\
&= \big\langle A_\varepsilon, \Lambda_\varepsilon^{v_i}(v_i^\varepsilon) \Lambda_\varepsilon^{v_j}(v_j^\varepsilon - v_j) \big\rangle_v = 0
\end{align*}
as $\Lambda_\varepsilon^{v_j}(v_j^\varepsilon - v_j) = v_j^\varepsilon - v_j^\varepsilon = 0$. Hence, we obtain that 
\[
\big\langle \Lambda_\varepsilon(A_\varepsilon) K^\varepsilon \big\rangle = 0.
\]
For any $1 \leq i \leq 3$, $(A_\varepsilon)_{ii}$ is even in $v$-variable, correspondingly $\Lambda_\varepsilon(A_\varepsilon)_{ii}$ is also even in $v$-variable.
In the case when $i \neq j$, it holds that
\[
\Lambda_\varepsilon(A_\varepsilon)_{ij} = \Lambda_\varepsilon(v_i^\varepsilon v_j) = v_i^\varepsilon v_j^\varepsilon
\]
with $v_i^\varepsilon$ being odd in $v_i$ and $v_j^\varepsilon$ being odd in $v_j$.
Since $\mathrm{e}_{1,k}^\varepsilon$ is odd in $v$-variable for any $1 \leq k \leq 3$ and $\mathrm{e}_2^\varepsilon$ is even in $v$-variable, we conclude that the term $\Lambda_\varepsilon(A_\varepsilon)_{ij} \mathrm{e}_2^\varepsilon \mathrm{e}_{1,k}^\varepsilon$ is always odd in $v$-variable for any $1 \leq i,j,k \leq 3$, i.e.,
\[
\big\langle \Lambda_\varepsilon(A_\varepsilon) \mathrm{e}_2^\varepsilon \mathrm{e}_{1,k}^\varepsilon \big\rangle = 0 \quad \forall \; 1 \leq k \leq 3.
\]
As a result, for any $1 \leq i \leq 3$, the convergence behavior of $\left\{ \nabla_x \cdot \big\langle \Lambda_\varepsilon(A_\varepsilon) \mathcal{P}^\varepsilon(f^\varepsilon)^2 \big\rangle \right\}_i$ is governed by
\[
(c_1^\varepsilon)^2 \sum_{j,k,\ell} \big\langle \Lambda_\varepsilon(A_\varepsilon)_{ji} (A_\varepsilon)_{k \ell} \big\rangle \partial_{x_j} \big( u_k^\varepsilon u_\ell^\varepsilon \big) \quad \text{and} \quad \sum_j \big\langle \Lambda_\varepsilon(A_\varepsilon)_{ji} \big( \mathrm{e}_2^\varepsilon \big)^2 \big\rangle \partial_{x_j} \big( \theta^\varepsilon \big)^2.
\]

By expanding $\Lambda_\varepsilon(A_\varepsilon)_{ji} (A_\varepsilon)_{k \ell}$ using the definition of matrix $A_\varepsilon$, we have that
\begin{align*}
\big\langle \Lambda_\varepsilon(A_\varepsilon)_{ji} (A_\varepsilon)_{k \ell} \big\rangle &= \big\langle \big( \Lambda_\varepsilon(v_j^\varepsilon v_i) - (a_\varepsilon v_\varepsilon^2 + b_\varepsilon) \delta_{ji} \big) \big( v_k^\varepsilon v_\ell - (a_\varepsilon v_\varepsilon^2 + b_\varepsilon) \delta_{k \ell} \big) \big\rangle \\
&= \big\langle \Lambda_\varepsilon(v_j^\varepsilon v_i) v_k^\varepsilon v_\ell \big\rangle - \big\langle \big( a_\varepsilon v_\varepsilon^2 + b_\varepsilon) \delta_{ji} v_k^\varepsilon v_\ell \big\rangle - \big\langle (a_\varepsilon v_\varepsilon^2 + b_\varepsilon) \delta_{k \ell} \Lambda_\varepsilon(v_j^\varepsilon v_i) \big\rangle \\
&\ \ + \big\langle (a_\varepsilon v_\varepsilon^2 + b_\varepsilon)^2 \delta_{ji} \delta_{k \ell} \big\rangle.
\end{align*}
It is obvious to observe that
\begin{align*}
\sum_{j,k,\ell} \big\langle (a_\varepsilon v_\varepsilon^2 + b_\varepsilon) \delta_{ji} v_k^\varepsilon v_\ell \big\rangle \partial_{x_j} \big( u_k^\varepsilon u_\ell^\varepsilon \big) &= \sum_{k,\ell} \big\langle (a_\varepsilon v_\varepsilon^2 + b_\varepsilon) v_k^\varepsilon v_\ell \big\rangle \partial_{x_i} \big( u_k^\varepsilon u_\ell^\varepsilon \big).
\end{align*}
Since $a_\varepsilon v_\varepsilon^2 + b_\varepsilon$ is even in $v$-variable, $k \neq \ell$ would imply that $\big\langle (a_\varepsilon v_\varepsilon^2 + b_\varepsilon) v_k^\varepsilon v_\ell \big\rangle = 0$.
Thus, it can be deduced that 
\begin{align*}
\sum_{j,k,\ell} \big\langle (a_\varepsilon v_\varepsilon^2 + b_\varepsilon) \delta_{ji} v_k^\varepsilon v_\ell \big\rangle \partial_{x_j} \big( u_k^\varepsilon u_\ell^\varepsilon \big) &= \sum_k \big\langle (a_\varepsilon v_\varepsilon^2 + b_\varepsilon) v_k^\varepsilon v_k \big\rangle \partial_{x_i} \big( u_k^\varepsilon \big)^2 \\
&= \sum_k \left\langle v_k^2 \frac{|v|^2}{3} \right\rangle \partial_{x_i} \big( u_k^\varepsilon \big)^2 + r_1^\varepsilon
\end{align*}
where
\[
r_1^\varepsilon := \sum_k \left\langle (a_\varepsilon v_\varepsilon^2 + b_\varepsilon) v_k^\varepsilon v_k - \frac{|v|^2}{3} v_k^2 \right\rangle \partial_{x_i} \big( u_k^\varepsilon \big)^2.
\]
Let $0<T<\infty$ and $\Phi \in C^\infty( [0,T] \times \mathbf{T}^3 )$.
For any $t \in [0,T]$, we have that
\begin{align*}
\left| \int_{\mathbf{T}^3} r_1^\varepsilon \Phi \, dx \right| &\leq \sum_k \left| \left\langle (a_\varepsilon v_\varepsilon^2 + b_\varepsilon) v_k^\varepsilon v_k - \frac{|v|^2}{3} v_k^2 \right\rangle \right| \cdot \left| \int_{\mathbf{T}^3} \big( u_k^\varepsilon \big)^2 \big( \partial_{x_i} \Phi \big) \, dx \right| \\
&\leq \| \partial_{x_i} \Phi \|_{L^\infty(\mathbf{T}_x^3)} \sum_k \left| \left\langle (a_\varepsilon v_\varepsilon^2 + b_\varepsilon) v_k^\varepsilon v_k - \frac{|v|^2}{3} v_k^2 \right\rangle \right| \cdot \| u_k^\varepsilon \|_{L^2(\mathbf{T}_x^3)}^2.
\end{align*}
Since
\[
\left| \left\langle (a_\varepsilon v_\varepsilon^2 + b_\varepsilon) v_k^\varepsilon v_k - \frac{|v|^2}{3} v_k^2 \right\rangle \right| \to 0 \quad \text{as} \quad \varepsilon \to 0
\]
for any $1 \leq k \leq 3$, we can show that
\[
\left| \int_0^T \int_{\mathbf{T}^3} r_1^\varepsilon \Phi \, dx \, dt \right| \to 0 \quad \text{as} \quad \varepsilon \to 0.
\]
Since values of 
\[
\left\langle v_k^2 \frac{|v|^2}{3} \right\rangle \quad \text{and} \quad \left\langle (a_\varepsilon v_\varepsilon^2 + b_\varepsilon) v_k^\varepsilon v_k - \frac{|v|^2}{3} v_k^2 \right\rangle
\]
are independent of $1 \leq k \leq 3$, we conclude that
\[
\sum_{j,k,\ell} \big\langle (a_\varepsilon v_\varepsilon^2 + b_\varepsilon) \delta_{ji} v_k^\varepsilon v_\ell \big\rangle \partial_{x_j} \big( u_k^\varepsilon u_\ell^\varepsilon \big) = \left\langle v_i^2 \frac{|v|^2}{3} \right\rangle \partial_{x_i} | u^\varepsilon |^2 + r_1^\varepsilon
\]
with
\begin{align} \label{r2eps}
r_1^\varepsilon := \left\langle (a_\varepsilon v_\varepsilon^2 + b_\varepsilon) v_k^\varepsilon v_k - \frac{|v|^2}{3} v_k^2 \right\rangle \partial_{x_i} | u^\varepsilon |^2.
\end{align}

In the mean time, since 
\[
\big\langle (a_\varepsilon v_\varepsilon^2 + b_\varepsilon) \Lambda_\varepsilon(v_i^\varepsilon v_i) \big\rangle = \big\langle a_\varepsilon v_\varepsilon^2 + b_\varepsilon, \Lambda_\varepsilon(v_i^\varepsilon v_i) \big\rangle_v = \big\langle \Lambda_\varepsilon(a_\varepsilon v_\varepsilon^2 + b_\varepsilon), v_i^\varepsilon v_i \big\rangle_v = \big\langle a_\varepsilon v_\varepsilon^2 + b_\varepsilon, v_i^\varepsilon v_i \big\rangle_v,
\]
by exactly the same derivation as above, we can show that
\[
\sum_{j,k,\ell} \big\langle (a_\varepsilon v_\varepsilon^2 + b_\varepsilon) \delta_{k \ell} \Lambda_\varepsilon(v_j^\varepsilon v_i) \big\rangle \partial_{x_j} \big( u_k^\varepsilon u_\ell^\varepsilon \big) = \left\langle v_i^2 \frac{|v|^2}{3} \right\rangle \partial_{x_i} | u^\varepsilon |^2 + r_1^\varepsilon
\]
where $r_1^\varepsilon$ is defined by expression (\ref{r2eps}).
Furthermore, we can show that
\begin{align*}
\sum_{j,k,\ell} \big\langle (a_\varepsilon v_\varepsilon^2 + b_\varepsilon)^2 \delta_{ji} \delta_{k \ell} \big\rangle \partial_{x_j} \big( u_k^\varepsilon u_\ell^\varepsilon \big) &= \left\langle \frac{|v|^4}{9} \right\rangle \partial_{x_i} |u^\varepsilon|^2 + r_2^\varepsilon
\end{align*}
where
\[
r_2^\varepsilon := \left\langle (a_\varepsilon v_\varepsilon^2 + b_\varepsilon)^2 - \frac{|v|^4}{9} \right\rangle \partial_{x_i} |u^\varepsilon|^2.
\]
Since 
\[
\left\| a_\varepsilon v_\varepsilon^2 + b_\varepsilon - \frac{|v|^2}{3} \right\|_{L^2(\Omega_v)} \to 0 \quad \text{as} \quad \varepsilon \to 0,
\]
we can easily show that
\[
\left| \int_0^T \int_{\mathbf{T}^3} r_2^\varepsilon \Phi \, dx \, dt \right| \to 0 \quad \text{as} \quad \varepsilon \to 0.
\]

On the other hand, analogous to what we have already discussed in the proof of Lemma \ref{Diff}, the value of $\big\langle \Lambda_\varepsilon(v_j^\varepsilon v_i) v_k^\varepsilon v_\ell \big\rangle$ is not zero only in the following four cases where
\[
i=j=k=\ell \quad \text{or} \quad i=j, k=\ell, j \neq k \quad \text{or} \quad i=k,j=\ell, k \neq j \quad \text{or} \quad i=\ell,j=k, \ell \neq j.
\]
Hence, we have that
\begin{align*}
&\sum_{j,k,\ell} \big\langle \Lambda_\varepsilon(v_j^\varepsilon v_i) v_k^\varepsilon v_\ell \big\rangle \partial_{x_j} \big( u_k^\varepsilon u_\ell^\varepsilon \big) \\
&\ \ = \big\langle \Lambda_\varepsilon(v_i^\varepsilon v_i) v_i^\varepsilon v_i \big\rangle \partial_{x_i} \big( u_i^\varepsilon \big)^2 + \sum_{k \neq i} \big\langle \Lambda_\varepsilon(v_i^\varepsilon v_i) v_k^\varepsilon v_k \big\rangle \partial_{x_i} \big( u_k^\varepsilon \big)^2 \\
&\ \ \ \ + \sum_{k \neq j} \big\langle v_j^\varepsilon (v_k^\varepsilon)^2 v_j \big\rangle \partial_{x_j} \big( u_k^\varepsilon u_j^\varepsilon \big) + \sum_{j \neq i} \big\langle v_j^\varepsilon v_i^\varepsilon v_j^\varepsilon v_i \big\rangle \partial_{x_j} \big( u_j^\varepsilon u_i^\varepsilon \big) \\
&\ \ = \frac{1}{144} \big( \nabla_x |u^\varepsilon|^2 \big)_i + \frac{1}{72} \big( \nabla_x \cdot ( u^\varepsilon \otimes u^\varepsilon ) \big)_i - \frac{1}{120} \partial_{x_i} \big( u_i^\varepsilon \big)^2 + r_3^\varepsilon
\end{align*}
where
\begin{align*}
r_3^\varepsilon &:= \big\langle \Lambda_\varepsilon(v_i^\varepsilon v_i) v_i^\varepsilon v_i - v_i^4 \big\rangle \partial_{x_i} \big( u_i^\varepsilon \big)^2 + \sum_{k \neq i} \big\langle \Lambda_\varepsilon(v_i^\varepsilon v_i) v_k^\varepsilon v_k - v_i^2 v_k^2 \big\rangle \partial_{x_i} \big( u_k^\varepsilon \big)^2 \\
&\ \ + \sum_{k \neq j} \big\langle v_j^\varepsilon (v_k^\varepsilon)^2 v_j - v_j^2 v_k^2 \big\rangle \partial_{x_k} \big( u_k^\varepsilon u_i^\varepsilon \big) + \sum_{j \neq i} \big\langle v_j^\varepsilon v_i^\varepsilon v_j^\varepsilon v_i - v_i^2 v_j^2 \big\rangle \partial_{x_j} \big( u_i^\varepsilon u_j^\varepsilon \big).
\end{align*}
By the triangle inequality and H$\ddot{\text{o}}$lder's inequality, we can deduce that
\begin{align*}
\big| \big\langle \Lambda_\varepsilon(v_i^\varepsilon v_i) v_i^\varepsilon v_i - v_i^4 \big\rangle \big| &\lesssim \int_\Omega \big| \Lambda_\varepsilon(v_i^\varepsilon v_i) v_i^\varepsilon - v_i^3 \big| \, dv \lesssim \int_\Omega \big| \Lambda_\varepsilon(v_i^\varepsilon v_i) - v_i^2 \big| \cdot |v_i^\varepsilon| \, dv + \int_\Omega |v_i^\varepsilon - v_i| \, dv \\
&\leq \big\| \Lambda_\varepsilon(v_i^\varepsilon v_i) - v_i^2 \big\|_{L^2(\Omega_v)} \| v_i^\varepsilon \|_{L^2(\Omega_v)} + \| v_i^\varepsilon - v_i \|_{L^2(\Omega_v)}.
\end{align*}
Since $\| v_i^\varepsilon - v_i \|_{L^2(\Omega_v)}$ tends to zero as $\varepsilon \to 0$, $\| v_i^\varepsilon \|_{L^2(\Omega_v)}$ is certainly bounded for $\varepsilon$ sufficiently small.
Then, by Plancherel's identity, see e.g. \cite[Proposition 3.2.7]{GraC}, we observe that
\begin{align*}
\big\| \Lambda_\varepsilon(v_i^\varepsilon v_i) - v_i^2 \big\|_{L^2(\Omega_v)} &\leq \big\| \Lambda_\varepsilon(v_i^\varepsilon v_i) - \Lambda_\varepsilon(v_i^2) \big\|_{L^2(\Omega_v)} + \big\| \Lambda_\varepsilon(v_i^2) - v_i^2 \big\|_{L^2(\Omega_v)} \\
&\lesssim \| v_i^\varepsilon - v_i \|_{L^2(\Omega_v)} + \big\| \Lambda_\varepsilon(v_i^2) - v_i^2 \big\|_{L^2(\Omega_v)} \to 0 
\end{align*}
as $\varepsilon \to 0$. Thus, we conclude that $\big| \big\langle \Lambda_\varepsilon(v_i^\varepsilon v_i) v_i^\varepsilon v_i - v_i^4 \big\rangle \big| \to 0$ as $\varepsilon \to 0$. By almost the same derivation, we also have that $\big| \big\langle \Lambda_\varepsilon(v_i^\varepsilon v_i) v_k^\varepsilon v_k - v_i^2 v_k^2 \big\rangle \big| \to 0$ as $\varepsilon \to 0$ for $k \neq i$.
On the other hand, by convergence (\ref{LpE:Cfv}), we have that 
\begin{align*}
\big| \big\langle v_j^\varepsilon (v_k^\varepsilon)^2 v_j - v_j^2 v_k^2 \big\rangle \big| &\lesssim \int_\Omega \big| v_j^\varepsilon (v_k^\varepsilon)^2 - v_j v_k^2 \big| \, dv \\
&\lesssim \| v_j^\varepsilon - v_j \|_{L^2(\Omega_v)} \| v_k^\varepsilon \|_{L^4(\Omega_v)}^2 + \| v_k^\varepsilon - v_k \|_{L^2(\Omega_v)} (\| v_k^\varepsilon \|_{L^2(\Omega_v)} + 1) \to 0
\end{align*}
as $\varepsilon \to 0$. 
As a result, we can then show that
\[
\left| \int_0^T \int_{\mathbf{T}^3} r_3^\varepsilon \Phi \, dx \, dt \right| \to 0 \quad \text{as} \quad \varepsilon \to 0
\]
for any $0<T<\infty$ and $\Phi \in C^\infty([0,T] \times \mathbf{T}^3)$.

Finally, we note that if $j \neq i$, then $v_j^\varepsilon v_i^\varepsilon \big( \mathrm{e}_2^\varepsilon \big)^2$ is always odd in $v$-variable, in this case where we have that
\[
\big\langle \Lambda_\varepsilon(A_\varepsilon)_{ji} \big( \mathrm{e}_2^\varepsilon \big)^2 \big\rangle = 0.
\]
Hence,
\[
\sum_j \big\langle \Lambda_\varepsilon(A_\varepsilon)_{ji} \big( \mathrm{e}_2^\varepsilon \big)^2 \big\rangle \partial_{x_j} \big( \theta^\varepsilon \big)^2 = 60 \left\langle A_{ii} \left( |v|^2 - \frac{1}{4} \right)^2 \right\rangle \partial_{x_i} \big( \theta^\varepsilon \big)^2 + r_4^\varepsilon = r_4^\varepsilon
\]
where
\[
r_4^\varepsilon := \left\langle \Lambda_\varepsilon(A_\varepsilon)_{ii} \big( \mathrm{e}_2^\varepsilon \big)^2 - 60 A_{ii} \left( |v|^2 - \frac{1}{4} \right)^2 \right\rangle \partial_{x_i} \big( \theta^\varepsilon \big)^2.
\]
Since $\| \mathrm{e}_2^\varepsilon - \mathrm{e}_2 \|_{L^4(\Omega_v)}$ and $\| \mathrm{e}_2^\varepsilon - \mathrm{e}_2 \|_{L^2(\Omega_v)}$ converges to zero as $\varepsilon \to 0$, we can deduce that
\[
\left| \left\langle \Lambda_\varepsilon(A_\varepsilon)_{ii} \big( \mathrm{e}_2^\varepsilon \big)^2 - 60 A_{ii} \left( |v|^2 - \frac{1}{4} \right)^2 \right\rangle \right| \to 0 \quad \text{as} \quad \varepsilon \to 0.
\]
Hence, we can show that
\[
\left| \int_0^T \int_{\mathbf{T}^3} r_4^\varepsilon \Phi \, dx \, dt \right| \to 0 \quad \text{as} \quad \varepsilon \to 0
\]
for any $0<T<\infty$ and $\Phi \in C^\infty([0,T] \times \mathbf{T}^3)$. By setting 
\[
N_{ns,i}^\varepsilon := \sum_{i=1}^4 r_i^\varepsilon, 
\]
we obtain Lemma \ref{Dtp}.
\end{proof}

\subsection{Derivation of the equation of $\theta^\varepsilon$} 
\label{Sub:thetaeq}

In this section, we rewrite the third equation of system (\ref{ruipAE}) in terms of $\rho^\varepsilon$, $u^\varepsilon$ and $\theta^\varepsilon$. 
Analogous to the usage of matrix $A_\varepsilon$ in the derivation of the second equation of system (\ref{Re:Conser}), in this case we need a vector $B_\varepsilon(v) \in \mathrm{ker}^\perp(\mathcal{L}^\varepsilon)$ to rewrite the term $\varepsilon^{-1} \nabla_x \cdot \langle \mathrm{e}_2^\varepsilon v f_\varepsilon \rangle$.

\begin{lemma} \label{vec:Bep}
There exists a constant $c_\varepsilon$ such that the vector 
\[
B_\varepsilon(v) := v (v_\varepsilon^2 - c_\varepsilon) \in \mathrm{ker}^\perp(\mathcal{L}^\varepsilon).
\]
In particular,
\[
c_\varepsilon \to \frac{19}{60} \quad \text{as} \quad \varepsilon \to 0.
\]
\end{lemma}
\begin{proof}
Indeed, since $B_\varepsilon$ is always odd in $v$-variable, it trivially holds that
\[
\langle B_\varepsilon, 1 \rangle_v = \langle B_\varepsilon, v_\varepsilon^2 \rangle_v = 0.
\]
Since $v_\varepsilon^2 - c_\varepsilon$ is even in $v$-variable, in the case when $i \neq j$, we have that
\[
\langle v_i (v_\varepsilon^2 - c_\varepsilon), v_j^\varepsilon \rangle_v = 0.
\]
Hence, it is sufficient to solve the equation
\[
c_\varepsilon \langle v_i v_i^\varepsilon \rangle = \langle v_i v_i^\varepsilon v_\varepsilon^2 \rangle
\]
for $c_\varepsilon$.
Since $\langle v_i v_i^\varepsilon \rangle = \langle v_i^\varepsilon, v_i^\varepsilon \rangle_v$, we observe that
\[
| \langle v_i v_i^\varepsilon \rangle | \leq \| v_i^\varepsilon \|_{L^2(\Omega_v)}^2 \to \| v_i \|_{L^2(\Omega_v)}^2 = \frac{1}{12} \quad \text{as} \quad \varepsilon \to 0.
\]
On the other hand, by the triangle inequality and H$\ddot{\text{o}}$lder's inequality, we deduce that 
\begin{align*}
| \langle v_i v_i^\varepsilon v_\varepsilon^2 - v_i^2 |v|^2 \rangle | &\lesssim \int_\Omega \big| v_i^\varepsilon v_\varepsilon^2 - v_i |v|^2 \big| \, dv \leq \int_\Omega |v_i^\varepsilon - v_i| \cdot |v_\varepsilon^2| \, dv + \int_\Omega |v_i| \cdot \big| v_\varepsilon^2 - |v|^2 \big| \, dv \\
&\leq \| v_i^\varepsilon - v_i \|_{L^2(\Omega_v)} \| v_\varepsilon^2 \|_{L^2(\Omega_v)} + \| v_i \|_{L^2(\Omega_v)} \big\| v_\varepsilon^2 - |v|^2 \big\|_{L^2(\Omega_v)}.
\end{align*}
Since
\[
\big\| v_\varepsilon^2 - |v|^2 \big\|_{L^2(\Omega_v)} \leq \sum_{i=1}^3 \| \Lambda_\varepsilon^{v_i} (v_i^2) - v_i^2 \|_{L^2(\Omega_v)} \to 0 \quad \text{as} \quad \varepsilon \to 0.
\]
and $\| v_i^\varepsilon - v_i \|_{L^2(\Omega_v)} \to 0$ as $\varepsilon \to 0$, we obtain that
\[
| \langle v_i v_i^\varepsilon v_\varepsilon^2 \rangle| \to | \langle v_i^2 |v|^2 \rangle | = \frac{19}{720} \quad \text{as} \quad \varepsilon \to 0.
\]
This completes the proof of Lemma \ref{vec:Bep}.
\end{proof}

Making use of this vector $B_\varepsilon$, we rewrite
\begin{align*}
\frac{1}{\varepsilon} \nabla_x \cdot \big\langle \mathrm{e}_2^\varepsilon v f_\varepsilon \big\rangle &= \frac{c_2^\varepsilon}{\varepsilon} \nabla_x \cdot \Big\langle v \Big( v_\varepsilon^2 - \frac{3 c_0^\varepsilon}{c_2^\varepsilon} \Big) f_\varepsilon \Big\rangle \\
&= \frac{c_2^\varepsilon}{\varepsilon} \nabla_x \cdot \big\langle B_\varepsilon f_\varepsilon \big\rangle + \frac{c_2^\varepsilon c_\varepsilon - 3 c_0^\varepsilon}{\varepsilon} \nabla_x \cdot \langle v f_\varepsilon \rangle.
\end{align*}
Since $f_\varepsilon = \Lambda_\varepsilon(f_\varepsilon)$, we have that 
\[
\langle v f_\varepsilon \rangle = \frac{1}{c_1^\varepsilon} u^\varepsilon.
\]
Since $B_\varepsilon \in \mathrm{ker}^\perp(\mathcal{L}^\varepsilon)$, using expression (\ref{Re:Lf}) for $\mathcal{L}^\varepsilon(f_\varepsilon)$, we deduce that
\begin{equation} \label{MC:thetaeq}
\begin{split}
&\frac{1}{\varepsilon} \nabla_x \cdot \big\langle B_\varepsilon f_\varepsilon \big\rangle = \nabla_x \cdot \Big\langle B_\varepsilon \frac{1}{\varepsilon} \mathcal{L}^\varepsilon(f_\varepsilon) \Big\rangle \\
&\ \ = - \varepsilon \nu_\ast \nabla_x \cdot \partial_t \big\langle B_\varepsilon f_\varepsilon \big\rangle - \nu_\ast \nabla_x \cdot \big\langle B_\varepsilon \Lambda_\varepsilon \big( v \cdot \nabla_x f_\varepsilon \big) \big\rangle + \kappa \nabla_x \cdot \big\langle B_\varepsilon \mathcal{L}^\varepsilon\big( \Lambda_\varepsilon (f_\varepsilon^2) \big) \big\rangle \\
&\quad \quad \quad - \varepsilon \kappa^2 \nabla_x \cdot \big\langle B_\varepsilon \Lambda_\varepsilon (f_\varepsilon^3) \big\rangle.
\end{split}
\end{equation}
By exactly the same proof as in the case of rewriting $\varepsilon^{-1} \nabla_x \cdot \langle A_\varepsilon f_\varepsilon \rangle$ in terms of equality (\ref{ReApproEq}), we can show that if $\gamma < \frac{1}{18}$, then the term $\varepsilon \nabla_x \cdot \partial_t \langle B_\varepsilon f_\varepsilon \rangle$ and the term $\varepsilon \nabla_x \cdot \langle B_\varepsilon \Lambda_\varepsilon(f_\varepsilon^3) \rangle$ both converge to zero in the sense of distributions as $\varepsilon \to 0$.
As a result, the convergence behavior of $\varepsilon^{-1} \nabla_x \cdot \langle B_\varepsilon f_\varepsilon \rangle$ is mainly dominated by the middle two terms on the right hand side of equation (\ref{MC:thetaeq}).

\begin{lemma} \label{thetaeq:Diff}
It holds that
\begin{align*}
\nabla_x \cdot \langle B_\varepsilon \Lambda_\varepsilon ( v \cdot \nabla_x f_\varepsilon ) \rangle = \frac{97 \sqrt{5}}{12600} \Delta_x \theta^\varepsilon + R_{\theta,D}^\varepsilon
\end{align*}
where $R_{\theta,D}^\varepsilon$ is some remainder function of $x$ and $t$ which converges to zero in the sense of distributions as $\varepsilon \to 0$, i.e., for any $0<T<\infty$ and $\Phi \in C^\infty( [0,T] \times \mathbf{T}^3 )$, it holds that
\[
\left| \int_0^T \int_{\mathbf{T}^3} \big( R_{\theta,D}(\varepsilon) \big) (x,t) \Phi(x,t) \, dx \, dt \right| \to 0 \quad \text{as} \quad \varepsilon \to 0.
\]
\end{lemma}
\begin{proof}
By decomposing $f_\varepsilon$ into the sum of $\mathcal{P}^\varepsilon(f_\varepsilon)$ and $\mathcal{L}^\varepsilon(f_\varepsilon)$, we have that
\begin{align*}
v \cdot \nabla_x f_\varepsilon &= v \cdot \nabla_x \rho^\varepsilon + c_1^\varepsilon \sum_{1\leq i,j \leq 3} v_i^\varepsilon v_j \partial_{x_j} u_i^\varepsilon + (v \mathrm{e}_2^\varepsilon) \cdot \nabla_x \theta^\varepsilon + v \cdot \nabla_x \mathcal{L}^\varepsilon(f_\varepsilon) \\
&= c_1^\varepsilon A_\varepsilon : \nabla_x u^\varepsilon + c_2^\varepsilon B_\varepsilon \cdot \nabla_x \theta^\varepsilon + v \cdot \nabla_x \mathcal{L}^\varepsilon(f_\varepsilon) + M^\varepsilon
\end{align*}
where
\[
M^\varepsilon := v \cdot \nabla_x \rho^\varepsilon + c_1^\varepsilon (a_\varepsilon v_\varepsilon^2 + b_\varepsilon) \nabla_x \cdot u^\varepsilon + (c_2^\varepsilon c_\varepsilon - 3 c_0^\varepsilon) v \cdot \nabla_x \theta^\varepsilon.
\]
Since the differentiation $\nabla_x$ commutes with the cutoff operator $\Lambda_\varepsilon$ and $f_\varepsilon = \Lambda_\varepsilon(f_\varepsilon)$, we observe that
\[
\Lambda_\varepsilon(M^\varepsilon) = v_\varepsilon \cdot \nabla_x \rho^\varepsilon + c_1^\varepsilon (a_\varepsilon v_\varepsilon^2 + b_\varepsilon) \nabla_x \cdot u^\varepsilon + (c_2^\varepsilon c_\varepsilon - 3 c_0^\varepsilon) v_\varepsilon \cdot \nabla_x \theta^\varepsilon \in \mathrm{ker}^\perp(\mathcal{L}^\varepsilon),
\]
which further implies that
\[
\langle B_\varepsilon \Lambda_\varepsilon(M^\varepsilon) \rangle = 0.
\]

Let $0<T<\infty$ and $\Phi \in C^\infty( [0,T] \times \mathbf{T}_x^3 )$. For any $t \in [0,T]$, we have that
\begin{align*}
\left| \int_{\mathbf{T}^3} \Big( \nabla_x \cdot \big\langle \Lambda_\varepsilon(B_\varepsilon) \big( v \cdot \nabla_x \mathcal{L}^\varepsilon(f_\varepsilon) \big) \big\rangle \Big) \Phi \, dx \right| \leq \int_{\mathbf{T}^3} \big| \big\langle \Lambda_\varepsilon(B_\varepsilon) \big( v \cdot \nabla_x \mathcal{L}^\varepsilon(f_\varepsilon) \big) \big\rangle \big| \cdot \big| \nabla_x \Phi \big| \, dx.
\end{align*}
Note that
\[
\big| \big\langle \Lambda_\varepsilon(B_\varepsilon) \big( v \cdot \nabla_x \mathcal{L}^\varepsilon(f_\varepsilon) \big) \big\rangle \big| \lesssim \| \Lambda_\varepsilon(B_\varepsilon) \|_{L^2(\Omega_v)} \big\| \nabla_x \mathcal{L}^\varepsilon(f_\varepsilon) \big\|_{L^2(\Omega_v)}.
\]
By Plancherel's identity, see e.g. \cite[Proposition 3.2.7]{GraC}, we deduce that
\[
\| \Lambda_\varepsilon(B_\varepsilon) \|_{L^2(\Omega_v)} \leq \| B_\varepsilon \|_{L^2(\Omega_v)} \lesssim \| v_\varepsilon^2 - c_\varepsilon \|_{L^2(\Omega_v)} \to \Big\| |v|^2 - \frac{19}{60} \Big\|_{L^2(\Omega_v)}
\]
as $\varepsilon \to 0$.
Since the $L^2$ norm of $|v|^2 - \frac{19}{60}$ in $\Omega$ is finite and independent of $\varepsilon$, it holds that
\[
\int_{\mathbf{T}^3} \big| \big\langle \Lambda_\varepsilon(B_\varepsilon) \big( v \cdot \nabla_x \mathcal{L}^\varepsilon(f_\varepsilon) \big) \big\rangle \big| \cdot \big| \nabla_x \Phi \big| \, dx \lesssim \big\| \nabla_x \mathcal{L}^\varepsilon(f_\varepsilon) \big\|_{L^2\big( \mathbf{T}_x^3; L^2(\Omega_v) \big)} \| \nabla_x \Phi \|_{L^2(\mathbf{T}_x^3)}.
\]
Hence, we obtain by Corollary \ref{Conv:RUT} that
\begin{align*}
&\left| \int_0^T \int_{\mathbf{T}^3} \Big( \nabla_x \cdot \big\langle \Lambda_\varepsilon(B_\varepsilon) \big( v \cdot \nabla_x \mathcal{L}^\varepsilon(f_\varepsilon) \big) \big\rangle \Big) \Phi \, dx \, dt \right| \\
&\ \ \lesssim \left( \int_0^T \| \mathcal{L}^\varepsilon(f_\varepsilon) \|_X \, dt \right)^{\frac{1}{2}} \| \nabla_x \Phi \|_{L^2\big( [0,T]; L^2(\mathbf{T}_x^3) \big)} \to 0 \quad \text{as} \quad \varepsilon \to 0.
\end{align*}

Since $\Lambda_\varepsilon(B_\varepsilon)_k (A_\varepsilon)_{ij}$ is odd in $v$-variable for any $1 \leq i,j,k \leq 3$, it holds that
\[
\big\langle B_\varepsilon \Lambda_\varepsilon\big( A_\varepsilon : \nabla_x u^\varepsilon \big) \big\rangle = 0.
\]
Thus, the convergence behavior of $\nabla_x \cdot \big\langle B_\varepsilon \Lambda_\varepsilon\big( v \cdot \nabla_x f_\varepsilon \big) \big\rangle$ is governed by 
\[
c_2^\varepsilon \nabla_x \cdot \big\langle \Lambda_\varepsilon(B_\varepsilon) \big( B_\varepsilon \cdot \nabla_x \theta^\varepsilon \big) \big\rangle.
\]
Since $\Lambda_\varepsilon(B_\varepsilon)_i$ is odd with respect to $v_i$ and even with respect to $v_j$ for $j \neq i$, the term $\Lambda_\varepsilon(B_\varepsilon)_i (B_\varepsilon)_j$ is always odd in $v$-variable for $i \neq j$.
Thus, 
\[
\big\langle \Lambda_\varepsilon(B_\varepsilon) \big( B_\varepsilon \cdot \nabla_x \theta^\varepsilon \big) \big\rangle_i = \big\langle \Lambda_\varepsilon(B_\varepsilon)_i (B_\varepsilon)_i \big\rangle \partial_{x_i} \theta^\varepsilon
\]
for any $1 \leq i \leq 3$ and
\[
\nabla_x \cdot \big\langle \Lambda_\varepsilon(B_\varepsilon) \big( B_\varepsilon \cdot \nabla_x \theta^\varepsilon \big) \big\rangle = \big\langle \Lambda_\varepsilon(B_\varepsilon)_i (B_\varepsilon)_i \big\rangle \Delta_x \theta^\varepsilon.
\]
By Plancherel's identity, see e.g. \cite[Proposition 3.2.7]{GraC}, it can be easily shown that
\begin{align*}
\big| \big\langle \Lambda_\varepsilon(B_\varepsilon)_i (B_\varepsilon)_i \big\rangle \big| \to | \langle B_i^2 \rangle | = \frac{97}{75600}
\end{align*}
as $\varepsilon \to 0$.
This completes the proof of Lemma \ref{thetaeq:Diff}.
\end{proof}

We next rewrite the term $\nabla_x \cdot \big\langle B_\varepsilon \mathcal{L}^\varepsilon\big( \Lambda_\varepsilon(f_\varepsilon^2) \big) \big\rangle$ in terms of $u^\varepsilon$ and $\theta^\varepsilon$.
Since $B_\varepsilon \in \mathrm{ker}^\perp(\mathcal{L}^\varepsilon)$, we have that
\[
\nabla_x \cdot \big\langle B_\varepsilon \mathcal{L}^\varepsilon\big( \Lambda_\varepsilon(f_\varepsilon^2) \big) \big\rangle = \nabla_x \cdot \big\langle B_\varepsilon \Lambda_\varepsilon(f_\varepsilon^2) \big\rangle.
\]

\begin{lemma} \label{thetaeq:Tr}
It holds that
\[
\nabla_x \cdot \big\langle B_\varepsilon \Lambda_\varepsilon(f_\varepsilon^2) \big\rangle = \frac{97 \sqrt{15}}{3150} \operatorname{div} (u^\varepsilon \theta^\varepsilon) + R_{\theta, U}^\varepsilon
\]
where $R_{\theta, U}$ is some remainder function of $x$ and $t$ which converges to zero in the sense of distributions as $\varepsilon \to 0$, i.e., for any $0<T<\infty$ and $\Phi \in C^\infty( [0,T] \times \mathbf{T}^3 )$, it holds that
\[
\left| \int_0^T \int_{\mathbf{T}^3} R_{\theta, U}^\varepsilon(x,t) \Phi(x,t) \, dx \, dt \right| \to 0 \quad \text{as} \quad \varepsilon \to 0.
\]
\end{lemma}
\begin{proof}
Firstly, we expand 
\[
f_\varepsilon^2 = \mathcal{P}^\varepsilon(f_\varepsilon)^2 + 2 \mathcal{P}^\varepsilon(f_\varepsilon) \mathcal{L}^\varepsilon(f_\varepsilon) + \mathcal{L}^\varepsilon(f_\varepsilon)^2.
\]
Let $0<T<\infty$ and $\Phi \in C^\infty( [0,T] \times \mathbf{T}^3 )$. 
Since the macroscopic projection $\mathcal{P}^\varepsilon$ commutes with the cutoff operator $\Lambda_\varepsilon$ and $f_\varepsilon = \Lambda_\varepsilon(f_\varepsilon)$, by Lemma \ref{Re:BerL} we deduce that
\begin{align*}
\big| \big\langle \Lambda_\varepsilon(B_\varepsilon) \mathcal{L}^\varepsilon(f_\varepsilon)^2 \big\rangle \big| \leq \| \Lambda_\varepsilon(B_\varepsilon) \|_{L^2(\Omega_v)} \| \mathcal{L}^\varepsilon(f_\varepsilon) \|_{L^4(\Omega_v)}^2 \lesssim \varepsilon^{- 6 \gamma} \| \Lambda_\varepsilon(B_\varepsilon) \|_{L^2(\Omega_v)} \| \mathcal{L}^\varepsilon(f_\varepsilon) \|_{L^2(\Omega_v)}^2.
\end{align*}
By the global energy estimate (\ref{GloEE:ep}), it holds that
\begin{align*}
&\left| \int_0^T \int_{\mathbf{T}^3} \Big( \nabla_x \cdot \big\langle \Lambda_\varepsilon(B_\varepsilon) \mathcal{L}^\varepsilon(f_\varepsilon)^2 \big\rangle \Big) \Phi \, dx \, dt \right| \\
&\ \ \lesssim \varepsilon^{-6 \gamma} \| \Lambda_\varepsilon(B_\varepsilon) \|_{L^2(\Omega_v)} \| \nabla_x \Phi \|_{L^\infty(\mathbf{T}_x^3)} \int_0^T \| \mathcal{L}^\varepsilon(f_\varepsilon) \|_X^2 \, dt \\
&\ \ \leq \varepsilon^{2 - 6 \gamma} \| \Lambda_\varepsilon(B_\varepsilon) \|_{L^2(\Omega_v)} \| \nabla_x \Phi \|_{L^\infty(\mathbf{T}_x^3)} \| f_0 \|_X^2.
\end{align*}
Since $\| \Lambda_\varepsilon(B_\varepsilon) \|_{L^2(\Omega_v)}$ is bounded for $\varepsilon$ sufficiently small (see proof of Lemma \ref{thetaeq:Diff}), if $\gamma < \frac{2}{6}$ then we observe that
\[
\left| \int_0^T \int_{\mathbf{T}^3} \Big( \nabla_x \cdot \big\langle \Lambda_\varepsilon(B_\varepsilon) \mathcal{L}^\varepsilon(f_\varepsilon)^2 \big\rangle \Big) \Phi \, dx \, dt \right| \to 0 \quad \text{as} \quad \varepsilon \to 0.
\]
This shows that $\nabla_x \cdot \big\langle \Lambda_\varepsilon(B_\varepsilon) \mathcal{L}^\varepsilon(f_\varepsilon)^2 \big\rangle$ converges to zero in the sense of distributions if $\gamma < \frac{2}{6}$.
Furthermore, by Lemma \ref{Re:BerL} again we have that
\begin{align*}
\big| \big\langle \Lambda_\varepsilon(B_\varepsilon) \mathcal{P}^\varepsilon(f_\varepsilon) \mathcal{L}^\varepsilon(f_\varepsilon) \big\rangle \big| &\leq \| \Lambda_\varepsilon(B_\varepsilon) \|_{L^2(\Omega_v)} \| \mathcal{P}^\varepsilon(f_\varepsilon) \|_{L^4(\Omega_v)} \| \mathcal{L}^\varepsilon(f_\varepsilon) \|_{L^4(\Omega_v)} \\
&\lesssim \varepsilon^{-6 \gamma} \| \Lambda_\varepsilon(B_\varepsilon) \|_{L^2(\Omega_v)} \| \mathcal{P}^\varepsilon(f_\varepsilon) \|_{L^2(\Omega_v)} \| \mathcal{L}^\varepsilon(f_\varepsilon) \|_{L^2(\Omega_v)}.
\end{align*}
Since $\| \mathcal{P}^\varepsilon(f_\varepsilon) \|_X$ is controlled by $\| f_\varepsilon \|_X$ (see Proposition \ref{BoLo}), by the global energy estimate (\ref{GloEE:ep}) again, we can deduce that
\begin{align*}
&\left| \int_0^T \int_{\mathbf{T}^3} \Big( \nabla_x \cdot \big\langle \Lambda_\varepsilon(B_\varepsilon) \mathcal{P}^\varepsilon(f_\varepsilon) \mathcal{L}^\varepsilon(f_\varepsilon) \big\rangle \Big) \Phi \, dx \, dt \right| \\
&\ \ \lesssim \varepsilon^{-6 \gamma} \| \Lambda_\varepsilon(B_\varepsilon) \|_{L^2(\Omega_v)} \| \nabla_x \Phi \|_{L^\infty(\mathbf{T}_x^3)} \left( \int_0^T \| \mathcal{P}^\varepsilon(f_\varepsilon) \|_X^2 \, dt \right)^{\frac{1}{2}} \left( \int_0^T \| \mathcal{L}^\varepsilon(f_\varepsilon) \|_X^2 \, dt \right)^{\frac{1}{2}} \\
&\ \ \leq \varepsilon^{1 - 6 \gamma} T \| \Lambda_\varepsilon(B_\varepsilon) \|_{L^2(\Omega_v)} \| \nabla_x \Phi \|_{L^\infty(\mathbf{T}_x^3)} \| f_0 \|_X^2.
\end{align*}
In this case, if $\gamma < \frac{1}{6}$, we have that
\[
\left| \int_0^T \int_{\mathbf{T}^3} \Big( \nabla_x \cdot \big\langle \Lambda_\varepsilon(B_\varepsilon) \mathcal{P}^\varepsilon(f_\varepsilon) \mathcal{L}^\varepsilon(f_\varepsilon) \big\rangle \Big) \Phi \, dx \, dt \right| \to 0 \quad \text{as} \quad \varepsilon \to 0.
\]
This shows that $\nabla_x \cdot \big\langle \Lambda_\varepsilon(B_\varepsilon) \mathcal{P}^\varepsilon(f_\varepsilon) \mathcal{L}^\varepsilon(f_\varepsilon) \big\rangle$ also converges to zero in the sense of distributions if $\gamma < \frac{1}{6}$.
As a result, the convergence behavior of $\nabla_x \cdot \big\langle B_\varepsilon \Lambda_\varepsilon(f_\varepsilon^2) \big\rangle$ is dominated by the term $\nabla_x \cdot \big\langle B_\varepsilon \Lambda_\varepsilon\big( \mathcal{P}^\varepsilon(f_\varepsilon)^2 \big) \big\rangle$.

Next, we directly work with the expansion (\ref{Exp:Pf2}) of $\mathcal{P}^\varepsilon(f_\varepsilon)^2$. 
Since
\[
\Lambda_\varepsilon(J^\varepsilon) = \Lambda_\varepsilon^x\big( (\rho^\varepsilon)^2 \big) + 2 \mathrm{e}_1^\varepsilon \cdot \Lambda_\varepsilon^x(\rho^\varepsilon u^\varepsilon) + 2 \mathrm{e}_2^\varepsilon \Lambda_\varepsilon^x(\rho^\varepsilon \theta^\varepsilon) + (c_1^\varepsilon)^2 \left( \frac{a_\varepsilon}{c_2^\varepsilon} \mathrm{e}_2^\varepsilon + \frac{3  c_0^\varepsilon a_\varepsilon}{c_2^\varepsilon} + b_\varepsilon \right) \Lambda_\varepsilon(|u^\varepsilon|^2)
\]
belongs to $\mathrm{ker}^\perp(\mathcal{L}^\varepsilon)$, it holds that 
\[
\big\langle B_\varepsilon \Lambda_\varepsilon(J^\varepsilon) \big\rangle = 0.
\]
Similar as in the proof of Lemma \ref{Dtp}, we have that
\[
\Lambda_\varepsilon\big( v_i^\varepsilon (v_j^\varepsilon - v_j) u_i^\varepsilon u_j^\varepsilon \big) = v_i^\varepsilon \Lambda_\varepsilon^{v_j}(v_j^\varepsilon - v_j) \Lambda_\varepsilon^x(u_i^\varepsilon u_j^\varepsilon) = 0
\]
for $i \neq j$. Hence, 
\[
\big\langle B_\varepsilon \Lambda_\varepsilon(K^\varepsilon) \big\rangle = 0.
\]
As we have already discussed in the proof of Lemma \ref{thetaeq:Diff}, $\Lambda_\varepsilon(B_\varepsilon)_k (A_\varepsilon)_{ij}$ is always odd in $v$-variable for any $1 \leq i,j,k \leq 3$, due to this fact, we deduce that
\[
\big\langle B_\varepsilon \Lambda_\varepsilon\big( A_\varepsilon : u^\varepsilon \otimes u^\varepsilon \big) \big\rangle = 0.
\]
Since $\mathrm{e}_2^\varepsilon$ is even in $v$-variable, certainly $(\mathrm{e}_2^\varepsilon)^2$ is also even in $v$-variable. Since $\Lambda_\varepsilon(B_\varepsilon)$ is odd in $v$-variable,
\[
\big\langle \Lambda_\varepsilon(B_\varepsilon) (\mathrm{e}_2^\varepsilon \theta^\varepsilon)^2 \big\rangle = 0.
\]
Note that
\[
\mathrm{e}_1^\varepsilon \mathrm{e}_2^\varepsilon = c_1^\varepsilon (v_\varepsilon - v) ( c_2^\varepsilon v_\varepsilon^2 - 3 c_0^\varepsilon ) + c_1^\varepsilon c_2^\varepsilon B_\varepsilon + (c_1^\varepsilon c_2^\varepsilon c_\varepsilon - 3 c_0^\varepsilon c_1^\varepsilon ) v.
\]
Since 
\[
\big| \big\langle \Lambda_\varepsilon(B_\varepsilon) (v_\varepsilon - v) (c_2^\varepsilon v_\varepsilon^2 - 3 c_0^\varepsilon) \big\rangle \big| \to 0 \quad \text{as} \quad \varepsilon \to 0,
\]
we can easily show that the term $\nabla_x \cdot \big\langle \Lambda_\varepsilon(B_\varepsilon) (v_\varepsilon - v) (c_2^\varepsilon v_\varepsilon^2 - 3 c_0^\varepsilon) \cdot (\theta^\varepsilon u^\varepsilon)\big\rangle$ converges to zero in the sense of distributions.
Since $\Lambda_\varepsilon(B_\varepsilon)_i$ is odd with respect to $v_i$ for any $1 \leq i \leq 3$, we observe that
\[
\nabla_x \cdot \big\langle \Lambda_\varepsilon(B_\varepsilon) \theta^\varepsilon \big( B_\varepsilon \cdot u^\varepsilon \big) \big\rangle = \big\langle \Lambda_\varepsilon(B_\varepsilon)_i (B_\varepsilon)_i \big\rangle \nabla_x \cdot (\theta^\varepsilon u^\varepsilon)
\]
and
\[
\nabla_x \cdot \big\langle \Lambda_\varepsilon(B_\varepsilon) \theta^\varepsilon (v \cdot u^\varepsilon) \big\rangle = \big\langle \Lambda_\varepsilon(B_\varepsilon)_i v_i \big\rangle \nabla_x \cdot (\theta^\varepsilon u^\varepsilon).
\]
Finally, by noting that
\[
\big\langle \Lambda_\varepsilon(B_\varepsilon)_i (B_\varepsilon)_i \big\rangle \to \frac{97}{75600} \quad \text{and} \quad \big\langle \Lambda_\varepsilon(B_\varepsilon)_i v_i \big\rangle \to 0
\]
as $\varepsilon \to 0$, we obtain Lemma \ref{thetaeq:Tr}.
\end{proof}

\section{The hydrodynamic limit: Proof of Theorem \ref{MCT}} 
\label{Sec:ConNSF}

We set 
\[
\mu_1^\varepsilon := \frac{a_\varepsilon c_1^\varepsilon}{c_2^\varepsilon}, \quad \mu_2^\varepsilon := c_1^\varepsilon \Big( \frac{3 a_\varepsilon c_0^\varepsilon}{c_2^\varepsilon} + b_\varepsilon \Big), \quad \mu_3^\varepsilon := \frac{\mu_5^\varepsilon}{c_1^\varepsilon}, \quad \mu_4^\varepsilon := \frac{\mu_2^\varepsilon}{c_1^\varepsilon} + \mu_1^\varepsilon \mu_3^\varepsilon, \quad \mu_5^\varepsilon := c_2^\varepsilon c_\varepsilon - 3 c_0^\varepsilon.
\]

\subsection{Convergence behavior of the forcing term} 
\label{Sub:ConFor}

We next expand
\[
f_\varepsilon^3 = \mathcal{P}^\varepsilon(f_\varepsilon)^3 + 3 \mathcal{P}^\varepsilon(f_\varepsilon)^2 \mathcal{L}^\varepsilon(f_\varepsilon) + 3 \mathcal{P}^\varepsilon(f_\varepsilon) \mathcal{L}^\varepsilon(f_\varepsilon)^2 + \mathcal{L}^\varepsilon(f_\varepsilon)^3
\]
to rewrite $\big\langle \mathrm{e}_i^\varepsilon f_\varepsilon^3 \big\rangle$ in terms of $\rho^\varepsilon$, $u^\varepsilon$ and $\theta^\varepsilon$.
We set 
\begin{align*}
\mathcal{R}_{\mathcal{L}^\varepsilon, 3} := 3 \mathcal{P}^\varepsilon(f_\varepsilon)^2 \mathcal{L}^\varepsilon(f_\varepsilon) + 3 \mathcal{P}^\varepsilon(f_\varepsilon) \mathcal{L}^\varepsilon(f_\varepsilon)^2 + \mathcal{L}^\varepsilon(f_\varepsilon)^3.
\end{align*}
\begin{lemma} \label{PuCov:L3}
Let $\varepsilon \in (0,1)$ be sufficiently small. Let $0 \leq t_1 \leq t_2 < \infty$.
Then, for any $\Phi \in L^\infty( [0,\infty); L^2(\mathbf{T}_x^3) )$ and $i \in \{0, 1, 2\}$, it holds that
\begin{align*}
\left| \int_{t_1}^{t_2} \int_{\mathbf{T}^3} \big\langle \mathrm{e}_i^\varepsilon \mathcal{R}_{\mathcal{L}^\varepsilon, 3} \big\rangle \Phi \, dx \, dt \right| \lesssim \sqrt{\nu_\ast (t_2 - t_1)} \varepsilon^{1 - 18 \gamma} \| f_0 \|_X^3 \| \Phi \|_{L^\infty( [0,\infty); L^2(\mathbf{T}_x^3) )}.
\end{align*}
\end{lemma}
\begin{proof}
Since the cutoff operator $\Lambda_\varepsilon$ commutes with both the macroscopic projection $\mathcal{P}^\varepsilon$ and the microscopic projection $\mathcal{L}^\varepsilon$, we observe by Lemma \ref{Re:BerL} that
\[
\| \mathcal{L}^\varepsilon(f_\varepsilon)^3 \|_{L^2(\Omega_v)}^2 = \| \mathcal{L}^\varepsilon(f_\varepsilon) \|_{L^6(\Omega_v)}^6 \lesssim \varepsilon^{-18 \gamma} \| \mathcal{L}^\varepsilon(f_\varepsilon) \|_{L^1(\Omega_v)}^6.
\]
By Minkowski's integral inequality and then the Sobolev embedding $H^1(\mathbf{T}_x^3) \hookrightarrow L^6(\mathbf{T}_x^3)$, we obtain that
\begin{align*}
\left( \int_{\mathbf{T}^3} \| \mathcal{L}^\varepsilon(f_\varepsilon) \|_{L^1(\Omega_v)}^6 \, dx \right)^{\frac{1}{6}} \leq \int_\Omega \| \mathcal{L}^\varepsilon(f_\varepsilon) \|_{H^1(\mathbf{T}_x^3)} \, dv \leq \| \mathcal{L}^\varepsilon(f_\varepsilon) \|_X.
\end{align*}
Hence, by Proposition \ref{BoLo} and the global energy estimate (\ref{GloEE:ep}), we deduce that
\begin{align*}
&\left| \int_{t_1}^{t_2} \int_{\mathbf{T}^3} \big\langle \mathrm{e}_i^\varepsilon \mathcal{L}^\varepsilon(f_\varepsilon)^3 \big\rangle \Phi \, dx \, dt \right| \\
&\ \ \lesssim \varepsilon^{-9 \gamma} \| \mathrm{e}_i^\varepsilon \|_{L^2(\Omega_v)} \| \Phi \|_{L^\infty( [t_1, t_2]; L^2(\mathbf{T}_x^3) )} \int_{t_1}^{t_2} \| \mathcal{L}^\varepsilon(f_\varepsilon) \|_X^3 \, dt \\
&\ \ \leq \varepsilon^{- 9 \gamma} \| \mathrm{e}_i^\varepsilon \|_{L^2(\Omega_v)} \| \Phi \|_{L^\infty( [0,T]; L^2(\mathbf{T}_x^3) )} \sup_{t \in [t_1, t_2]} \big\| \mathcal{L}^\varepsilon\big( f_\varepsilon \big) (t) \big\|_X^2 \int_{t_1}^{t_2} \| \mathcal{L}^\varepsilon(f_\varepsilon) \|_X \, dt \\
&\ \ \leq \sqrt{\nu_\ast (t_2 - t_1)} \varepsilon^{1 - 9 \gamma} \| \mathrm{e}_i^\varepsilon \|_{L^2(\Omega_v)} \| \Phi \|_{L^\infty( [0,T]; L^2(\mathbf{T}_x^3) )} \| f_0 \|_X^3.
\end{align*}
For $\varepsilon$ sufficiently small, we may assume without loss of generality that
\begin{align*}
\| \mathrm{e}_i^\varepsilon \|_{L^2(\Omega_v)} < \| \mathrm{e}_i \|_{L^2(\Omega_v)} + 1 < \infty
\end{align*}
since $\| \mathrm{e}_i^\varepsilon \|_{L^2(\Omega_v)}$ converges to $\| \mathrm{e}_i \|_{L^2(\Omega_v)}$ as $\varepsilon \to 0$.

On the other hand, we observe by H$\ddot{\text{o}}$lder's inequality that
\[
\big| \big\langle \mathrm{e}_i^\varepsilon \mathcal{P}^\varepsilon(f_\varepsilon) \mathcal{L}^\varepsilon(f_\varepsilon)^2 \big\rangle \big| \leq \| \mathrm{e}_i^\varepsilon \|_{L^2(\Omega_v)} \| \mathcal{P}^\varepsilon(f_\varepsilon) \mathcal{L}^\varepsilon(f_\varepsilon)^2 \|_{L^2(\Omega_v)}.
\]
Hence, by applying Proposition \ref{MRX} and the global energy inequality (\ref{GloEE:ep}), we deduce that
\begin{align*}
&\left| \int_{t_1}^{t_2} \int_{\mathbf{T}^3} \big\langle \mathrm{e}_i^\varepsilon \mathcal{P}^\varepsilon(f_\varepsilon) \mathcal{L}^\varepsilon(f_\varepsilon)^2 \big\rangle \Phi \, dx \, dt \right| \\
&\ \ \leq \| \mathrm{e}_i^\varepsilon \|_{L^2(\Omega_v)} \| \Phi \|_{L^\infty( [t_1, t_2]; L^2(\mathbf{T}_x^3) )} \int_{t_1}^{t_2} \| \mathcal{P}^\varepsilon(f_\varepsilon) \mathcal{L}^\varepsilon(f_\varepsilon)^2 \|_{L^2( \mathbf{T}_x^3; L^2(\Omega_v) )} \, dt \\
&\ \ \lesssim \varepsilon^{-18 \gamma} \| \mathrm{e}_i^\varepsilon \|_{L^2(\Omega_v)} \| \Phi \|_{L^\infty( [0,T]; L^2(\mathbf{T}_x^3) )} \sup_{t \in [t_1, t_2]} \| \mathcal{P}^\varepsilon(f_\varepsilon) \|_X \sup_{t \in [t_1, t_2]} \| \mathcal{L}^\varepsilon(f_\varepsilon) \|_X \int_{t_1}^{t_2} \| \mathcal{L}^\varepsilon(f_\varepsilon) \|_X \, dt \\
&\ \ \lesssim \sqrt{\nu_\ast (t_2 - t_1)} \varepsilon^{1 - 18 \gamma} \| \mathrm{e}_i^\varepsilon \|_{L^2(\Omega_v)} \| \Phi \|_{L^\infty( [0,T]; L^2(\mathbf{T}_x^3) )} \| f_0 \|_X^3,
\end{align*}
provided that $\varepsilon$ is sufficiently small.
Analogously, we estimate by H$\ddot{\text{o}}$lder's inequality that
\[
\big| \big\langle \mathrm{e}_i^\varepsilon \mathcal{P}^\varepsilon(f_\varepsilon)^2 \mathcal{L}^\varepsilon(f_\varepsilon) \big\rangle \big| \leq \| \mathrm{e}_i^\varepsilon \|_{L^2(\Omega_v)} \| \mathcal{P}^\varepsilon(f_\varepsilon)^2 \mathcal{L}^\varepsilon(f_\varepsilon) \|_{L^2(\Omega_v)}.
\]
Thus, in this case it holds that
\begin{align*}
&\left| \int_{t_1}^{t_2} \int_{\mathbf{T}^3} \big\langle \mathrm{e}_i^\varepsilon \mathcal{P}^\varepsilon(f_\varepsilon)^2 \mathcal{L}^\varepsilon(f_\varepsilon) \big\rangle \Phi \, dx \, dt \right| \\
&\ \ \lesssim \varepsilon^{-18 \gamma} \| \mathrm{e}_i^\varepsilon \|_{L^2(\Omega_v)} \| \Phi \|_{L^\infty( [t_1, t_2]; L^2(\mathbf{T}_x^3) )} \left( \sup_{t \in [t_1, t_2]} \| \mathcal{P}^\varepsilon(f_\varepsilon) \|_X \right)^2 \int_{t_1}^{t_2} \| \mathcal{L}^\varepsilon(f_\varepsilon) \|_X \, dt \\
&\ \ \leq \varepsilon^{1 - 18 \gamma} \sqrt{\nu_\ast (t_2 - t_1)} \| \mathrm{e}_i^\varepsilon \|_{L^2(\Omega_v)} \| \Phi \|_{L^\infty( [0,T]; L^2(\mathbf{T}_x^3) )} \| f_0 \|_X^3
\end{align*}
if $\varepsilon$ is sufficiently small.
This completes the proof of Lemma \ref{PuCov:L3}.
\end{proof}
\begin{remark} \label{WeCov:L3}
Lemma \ref{PuCov:L3} also says that for any $\Phi \in C^\infty( [0,T] \times \mathbf{T}^3 )$ with $0<T<\infty$ and $i \in \{0, 1, 2\}$,
\begin{align*}
\left| \int_0^T \int_{\mathbf{T}^3} \big\langle \mathrm{e}_i^\varepsilon \mathcal{R}_{\mathcal{L}^\varepsilon, 3} \big\rangle \Phi \, dx \, dt \right| \to 0 \quad \text{as} \quad \varepsilon \to 0
\end{align*}
if $\gamma < \frac{1}{18}$, i.e., $\big\langle \mathrm{e}_i^\varepsilon \mathcal{R}_{\mathcal{L}^\varepsilon, 3} \big\rangle$ converges to zero in the sense of distributions as $\varepsilon \to 0$ if $\gamma < \frac{1}{18}$.
\end{remark}

By a direct calculation, we have that
\begin{align*}
\mathcal{P}^\varepsilon(f_\varepsilon)^3 &= (\rho^\varepsilon)^3 + (\mathrm{e}_1^\varepsilon \cdot u^\varepsilon)^3 + (\mathrm{e}_2^\varepsilon \theta^\varepsilon)^3 + 3 (\rho^\varepsilon)^2 (\mathrm{e}_1^\varepsilon \cdot u^\varepsilon) + 3 \rho^\varepsilon (\mathrm{e}_1^\varepsilon \cdot u^\varepsilon)^2 + 3 (\mathrm{e}_1^\varepsilon \cdot u^\varepsilon)^2 (\mathrm{e}_2^\varepsilon \theta^\varepsilon) \\
&\ \ + 3 (\mathrm{e}_1^\varepsilon \cdot u^\varepsilon) (\mathrm{e}_2^\varepsilon \theta^\varepsilon)^2 + 3 (\rho^\varepsilon)^2 (\mathrm{e}_2^\varepsilon \theta^\varepsilon) + 3 \rho^\varepsilon (\mathrm{e}_2^\varepsilon \theta^\varepsilon)^2 + 6 \rho^\varepsilon \mathrm{e}_2^\varepsilon \theta^\varepsilon (\mathrm{e}_1^\varepsilon \cdot u^\varepsilon).
\end{align*}
Since $\mathrm{e}_1^\varepsilon$ is odd and $\mathrm{e}_2^\varepsilon$ is even with respect to $0$ in $v$-variable, we observe that
\begin{align*}
\big\langle \mathrm{e}_1^\varepsilon \mathcal{P}^\varepsilon(f_\varepsilon)^3 \big\rangle &= \big\langle \mathrm{e}_1^\varepsilon (\mathrm{e}_1^\varepsilon \cdot u^\varepsilon)^3 \big\rangle + 3 (\rho^\varepsilon)^2 \big\langle \mathrm{e}_1^\varepsilon (\mathrm{e}_1^\varepsilon \cdot u^\varepsilon) \big\rangle + 3 (\theta^\varepsilon)^2 \big\langle \mathrm{e}_1^\varepsilon (\mathrm{e}_2^\varepsilon)^2 (\mathrm{e}_1^\varepsilon \cdot u^\varepsilon) \big\rangle + 6 \rho^\varepsilon \theta^\varepsilon \big\langle \mathrm{e}_1^\varepsilon \mathrm{e}_2^\varepsilon (\mathrm{e}_1^\varepsilon \cdot u^\varepsilon) \big\rangle.
\end{align*}
For any $1 \leq i \leq 3$, it holds that
\begin{align*}
\big\langle \mathrm{e}_{1,i}^\varepsilon (\mathrm{e}_1^\varepsilon \cdot u^\varepsilon)^3 \big\rangle &= \big\langle (\mathrm{e}_{1,i}^\varepsilon)^4 \big\rangle (u_i^\varepsilon)^3 + 3 \big\langle (\mathrm{e}_{1,i}^\varepsilon \mathrm{e}_{1,j}^\varepsilon)^2 \big\rangle u_i^\varepsilon (u_j^\varepsilon)^2 + 3 \big\langle (\mathrm{e}_{1,i}^\varepsilon \mathrm{e}_{1,k}^\varepsilon)^2 \big\rangle u_i^\varepsilon (u_k^\varepsilon)^2 \\
&= \frac{9}{5} (u_i^\varepsilon)^3 + 3 u_i^\varepsilon (u_j^\varepsilon)^2 + 3 u_i^\varepsilon (u_k^\varepsilon)^2 + R_{3,i}^{\varepsilon,1}(x,t)
\end{align*}
where
\[
R_{3,i}^{\varepsilon,1}(x,t) := \big\langle (\mathrm{e}_{1,i}^\varepsilon)^4 - \mathrm{e}_{1,i}^4 \big\rangle (u_i^\varepsilon)^3 + 3 \big\langle (\mathrm{e}_{1,i}^\varepsilon \mathrm{e}_{1,j}^\varepsilon)^2 - \mathrm{e}_{1,i}^2 \mathrm{e}_{1,j}^2 \big\rangle u_i^\varepsilon (u_j^\varepsilon)^2 + 3 \big\langle (\mathrm{e}_{1,i}^\varepsilon \mathrm{e}_{1,k}^\varepsilon)^2 - \mathrm{e}_{1,i}^2 \mathrm{e}_{1,k}^2 \big\rangle u_i^\varepsilon (u_k^\varepsilon)^2.
\]
Furthermore, we rewrite
\begin{equation*}
 \begin{aligned}
 3 (\rho^\varepsilon)^2 \big\langle \mathrm{e}_{1,i}^\varepsilon (\mathrm{e}_1^\varepsilon \cdot u^\varepsilon) \big\rangle &=& &3 (\rho^\varepsilon)^2 u_i^\varepsilon + R_{3,i}^{\varepsilon,2},& \quad R_{3,i}^{\varepsilon,2} &:=& &3 \big\langle \mathrm{e}_{1,i}^\varepsilon \mathrm{e}_{1,i}^\varepsilon - \mathrm{e}_{1,i}^2 \big\rangle u_i^\varepsilon (\rho^\varepsilon)^2,& \\
 3 (\theta^\varepsilon)^2 \big\langle \mathrm{e}_{1,i}^\varepsilon (\mathrm{e}_2^\varepsilon)^2 (\mathrm{e}_1^\varepsilon \cdot u^\varepsilon) \big\rangle &=& &\frac{75}{7} (\theta^\varepsilon)^2 u_i^\varepsilon + R_{3,i}^{\varepsilon,3},& \quad R_{3,i}^{\varepsilon,3} &:=& &3 \big\langle (\mathrm{e}_{1,i}^\varepsilon \mathrm{e}_2^\varepsilon)^2 - (\mathrm{e}_{1,i} \mathrm{e}_2)^2 \big\rangle u_i^\varepsilon (\theta^\varepsilon)^2,& \\
 6 \rho^\varepsilon \theta^\varepsilon \big\langle \mathrm{e}_{1,i}^\varepsilon \mathrm{e}_2^\varepsilon (\mathrm{e}_1^\varepsilon \cdot u^\varepsilon) \big\rangle &=& &\frac{12 \sqrt{5}}{5} \rho^\varepsilon \theta^\varepsilon u_i^\varepsilon + R_{3,i}^{\varepsilon,4},& \quad R_{3,i}^{\varepsilon,4} &:=& &6 \big\langle (\mathrm{e}_{1,i}^\varepsilon)^2 \mathrm{e}_2^\varepsilon - \mathrm{e}_{1,i}^2 \mathrm{e}_2 \big\rangle u_i^\varepsilon \rho^\varepsilon \theta^\varepsilon&
 \end{aligned}
\end{equation*}
Hence, for any $1 \leq i \leq 3$, 
\begin{align*}
\big\langle \mathrm{e}_{1,i}^\varepsilon \mathcal{P}^\varepsilon(f_\varepsilon)^3 \big\rangle = F_i^\varepsilon + R_{3,i}^\varepsilon
\end{align*}
where
\[
F_i^\varepsilon := - \frac{6}{5} (u_i^\varepsilon)^3 + 3 u_i^\varepsilon |u^\varepsilon|^2 + 3 (\rho^\varepsilon)^2 u_i^\varepsilon + \frac{75}{7} (\theta^\varepsilon)^2 u_i^\varepsilon + \frac{12 \sqrt{5}}{5} \rho^\varepsilon \theta^\varepsilon u_i^\varepsilon
\]
and
\[
R_{3,i}^\varepsilon := \sum_{j=1}^4 R_{3,i}^{\varepsilon,j}.
\]

Analogously, we have that
\begin{align*}
\big\langle \mathrm{e}_2^\varepsilon \mathcal{P}^\varepsilon(f_\varepsilon)^3 \big\rangle &= \big\langle (\mathrm{e}_2^\varepsilon)^4 \big\rangle (\theta^\varepsilon)^3 + 3 \rho^\varepsilon \big\langle \mathrm{e}_2^\varepsilon (\mathrm{e}_1^\varepsilon \cdot u^\varepsilon)^2 \big\rangle + 3 \theta^\varepsilon \big\langle (\mathrm{e}_2^\varepsilon)^2 (\mathrm{e}_1^\varepsilon \cdot u^\varepsilon)^2 \big\rangle + 3 (\rho^\varepsilon)^2 \theta^\varepsilon \big\langle (\mathrm{e}_2^\varepsilon)^2 \big\rangle \\
&\ \ + 3 \rho^\varepsilon (\theta^\varepsilon)^2 \big\langle (\mathrm{e}_2^\varepsilon)^3 \big\rangle.
\end{align*}
We rewrite
\begin{equation*}
 \begin{aligned}
 \big\langle (\mathrm{e}_2^\varepsilon)^4 \big\rangle (\theta^\varepsilon)^3 &=& &\frac{171}{7} (\theta^\varepsilon)^3 + Q_3^{\varepsilon,1},& \quad Q_3^{\varepsilon,1} &:=& &\big\langle (\mathrm{e}_2^\varepsilon)^4 - \mathrm{e}_2^4 \big\rangle (\theta^\varepsilon)^3,& \\
 3 \rho^\varepsilon \big\langle \mathrm{e}_2^\varepsilon (\mathrm{e}_1^\varepsilon \cdot u^\varepsilon)^2 \big\rangle &=& &\frac{6 \sqrt{5}}{5} \rho^\varepsilon |u^\varepsilon|^2 + Q_3^{\varepsilon,2},& \quad Q_3^{\varepsilon,2} &:=& &3 \big\langle \mathrm{e}_2^\varepsilon (\mathrm{e}_{1,i}^\varepsilon)^2 - \mathrm{e}_2 \mathrm{e}_{1,i}^2 \big\rangle \rho^\varepsilon |u^\varepsilon|^2,& \\
 3 \theta^\varepsilon \big\langle (\mathrm{e}_2^\varepsilon)^2 (\mathrm{e}_1^\varepsilon \cdot u^\varepsilon)^2 \big\rangle &=& &\frac{15}{7} \theta^\varepsilon |u^\varepsilon|^2 + Q_3^{\varepsilon,3},& \quad Q_3^{\varepsilon,3} &:=& &3 \big\langle (\mathrm{e}_2^\varepsilon)^2 (\mathrm{e}_{1,i}^\varepsilon)^2 - \mathrm{e}_2^2 \mathrm{e}_{1,i}^2 \big\rangle \theta^\varepsilon |u^\varepsilon|^2,& \\
 3 (\rho^\varepsilon)^2 \theta^\varepsilon \big\langle (\mathrm{e}_2^\varepsilon)^2 \big\rangle &=& &9 \theta^\varepsilon (\rho^\varepsilon)^2 + Q_3^{\varepsilon,4},& \quad Q_3^{\varepsilon,4} &:=& &3 \big\langle (\mathrm{e}_2^\varepsilon)^2 - \mathrm{e}_2^2 \big\rangle (\rho^\varepsilon)^2 \theta^\varepsilon,& \\
 3 \rho^\varepsilon (\theta^\varepsilon)^2 \big\langle (\mathrm{e}_2^\varepsilon)^3 \big\rangle &=& &\frac{18 \sqrt{5}}{7} \rho^\varepsilon (\theta^\varepsilon)^2 + Q_3^{\varepsilon,5},& \quad Q_3^{\varepsilon,5} &:=& &3 \big\langle (\mathrm{e}_2^\varepsilon)^3 - \mathrm{e}_2^3 \big\rangle \rho^\varepsilon (\theta^\varepsilon)^2.&
 \end{aligned}
\end{equation*}
Thus, we obtain that
\begin{align*}
\big\langle \mathrm{e}_2^\varepsilon \mathcal{P}^\varepsilon(f_\varepsilon)^3 \big\rangle = G_\varepsilon + Q_3^\varepsilon
\end{align*}
where
\[
G_\varepsilon:= \frac{171}{7} (\theta^\varepsilon)^3 + \frac{6 \sqrt{5}}{5} \rho^\varepsilon |u^\varepsilon|^2 + \frac{15}{7} \theta^\varepsilon |u^\varepsilon|^2 + 9 \theta^\varepsilon (\rho^\varepsilon)^2 + \frac{18 \sqrt{5}}{7} \rho^\varepsilon (\theta^\varepsilon)^2
\]
and
\[
Q_3^\varepsilon := \sum_{j=1}^5 Q_3^{\varepsilon, j}.
\]

On the other hand, we also have that
\begin{align*}
\big\langle \mathcal{P}^\varepsilon(f_\varepsilon)^3 \big\rangle = (\rho^\varepsilon)^3 + \big\langle (\mathrm{e}_2^\varepsilon \theta^\varepsilon)^3 \big\rangle + 3 \rho^\varepsilon \big\langle (\mathrm{e}_1^\varepsilon \cdot u^\varepsilon)^2 \big\rangle + 3 \theta^\varepsilon \big\langle (\mathrm{e}_1^\varepsilon \cdot u^\varepsilon)^2 \mathrm{e}_2^\varepsilon \big\rangle + 3 \rho^\varepsilon \big\langle (\mathrm{e}_2^\varepsilon \theta^\varepsilon)^2 \big\rangle.
\end{align*}
Then, we rewrite
\begin{equation*}
 \begin{aligned}
 \big\langle (\mathrm{e}_2^\varepsilon \theta^\varepsilon)^3 \big\rangle &=& &\frac{6 \sqrt{5}}{7} (\theta^\varepsilon)^3 + W_3^{\varepsilon,1},& \quad W_3^{\varepsilon,1} &:=& &\big\langle (\mathrm{e}_2^\varepsilon)^3 - \mathrm{e}_2^3 \big\rangle (\theta^\varepsilon)^3,& \\
 3 \rho^\varepsilon \big\langle (\mathrm{e}_1^\varepsilon \cdot u^\varepsilon)^2 \big\rangle &=& &3 \rho^\varepsilon |u^\varepsilon|^2 + W_3^{\varepsilon,2},& \quad W_3^{\varepsilon,2} &:=& &3 \big\langle (\mathrm{e}_{1,i}^\varepsilon)^2 - \mathrm{e}_{1,i}^2 \big\rangle \rho^\varepsilon |u^\varepsilon|^2,& \\
 3 \theta^\varepsilon \big\langle (\mathrm{e}_1^\varepsilon \cdot u^\varepsilon)^2 \mathrm{e}_2^\varepsilon \big\rangle &=& &\frac{6 \sqrt{5}}{5} \theta^\varepsilon |u^\varepsilon|^2 + W_3^{\varepsilon,3},& \quad W_3^{\varepsilon,3} &:=& &3 \big\langle \mathrm{e}_2^\varepsilon (\mathrm{e}_{1,i}^\varepsilon)^2 - \mathrm{e}_2 \mathrm{e}_{1,i}^2 \big\rangle \theta^\varepsilon |u^\varepsilon|^2,& \\
 3 \rho^\varepsilon \big\langle (\mathrm{e}_2^\varepsilon \theta^\varepsilon)^2 \big\rangle &=& &9 \rho^\varepsilon (\theta^\varepsilon)^2 + W_3^{\varepsilon,4},& \quad W_3^{\varepsilon,4} &:=& &3 \big\langle (\mathrm{e}_2^\varepsilon)^2 - \mathrm{e}_2^2 \big\rangle \rho^\varepsilon (\theta^\varepsilon)^2.&
 \end{aligned}
\end{equation*}
Hence, we have that
\begin{align*}
\big\langle \mathcal{P}^\varepsilon(f_\varepsilon)^3 \big\rangle = E_\varepsilon + W_3^\varepsilon
\end{align*}
where
\[
E_\varepsilon := (\rho^\varepsilon)^3 + \frac{6 \sqrt{5}}{7} (\theta^\varepsilon)^3 + 3 \rho^\varepsilon |u^\varepsilon|^2 + \frac{6 \sqrt{5}}{5} \theta^\varepsilon |u^\varepsilon|^2 + 9 \rho^\varepsilon (\theta^\varepsilon)^2
\]
and
\[
W_3^\varepsilon := \sum_{j=1}^4 W_3^{\varepsilon, j}.
\]

\begin{remark} \label{L2Bd:RUT}
By the Sobolev embedding $H^1(\mathbf{T}_x^3) \hookrightarrow L^6(\mathbf{T}_x^3)$ and Corollary \ref{Conv:RUT}, we have that
\begin{equation} \label{LfL2:RUT}
\begin{split}
\| \rho^\varepsilon u^\varepsilon \theta^\varepsilon \|_{L^\infty( [0,\infty); L^2(\mathbf{T}_x^3) )} &\leq \| \rho^\varepsilon \|_{L^\infty( [0,\infty); H^1(\mathbf{T}_x^3) )} \| u^\varepsilon \|_{L^\infty( [0,\infty); H^1(\mathbf{T}_x^3) )} \| \theta^\varepsilon \|_{L^\infty( [0,\infty); H^1(\mathbf{T}_x^3) )} \\
&\leq \| f_0 \|_X^3.
\end{split}
\end{equation}
By analogous estimates as (\ref{LfL2:RUT}), we can deduce that $\{ F_i^\varepsilon \}_\varepsilon$ ($\forall \; 1 \leq i \leq 3$), $\{ G_\varepsilon \}_\varepsilon$ and $\{ E_\varepsilon \}_\varepsilon$ are all uniformly bounded in $L^\infty( [0,\infty); L^2(\mathbf{T}_x^3) )$.
Moreover, it can be easily observed that $R_{3,i}^\varepsilon$ ($\forall \; 1 \leq i \leq 3$), $Q_3^\varepsilon$ and $W_3^\varepsilon$ all converge to $0$ in norm $L^\infty( [0,\infty); L^2( \mathbf{T}_x^3; L^2(\Omega_v) ) )$ as $\varepsilon \to 0$.
\end{remark}
%

\subsection{Application of the $L^2$ Helmholtz projection on the main equation} 
\label{Sub:ConNSF}

We define that 
\[
F^\varepsilon := (F_1^\varepsilon, F_2^\varepsilon, F_3^\varepsilon), \quad H^\varepsilon := (H_1^\varepsilon, H_2^\varepsilon, H_3^\varepsilon), \quad J^\varepsilon := (J_1^\varepsilon, J_2^\varepsilon, J_3^\varepsilon)
\]
with
\[
J_i^\varepsilon := - \frac{1}{10} \partial_{x_i}^2 u_i^\varepsilon \quad \text{and} \quad H_i^\varepsilon := \frac{\sqrt{3} \kappa}{5} \partial_{x_i} (u_i^\varepsilon)^2 \quad \forall \; 1 \leq i \leq 3.
\]
System (\ref{ruipAE}) now turns into
\begin{eqnarray} \label{Re:Conser}
\left\{
\begin{array}{lcl}
\partial_{t} \rho^\varepsilon + \frac{1}{c_1^\varepsilon \varepsilon} \nabla_{x} \cdot u^\varepsilon = - \frac{\kappa^2}{\nu_\ast} E_\varepsilon + \mathcal{O}_{0,\varepsilon}, \\
\partial_{t} u^\varepsilon + \frac{1}{\varepsilon} \nabla_x \big( \mu_1^\varepsilon \theta^\varepsilon + \mu_2^\varepsilon \rho^\varepsilon \big) - \frac{\nu_\ast}{12} \Delta_x u^\varepsilon + \frac{\sqrt{3} \kappa}{3} \nabla_x \cdot (u^\varepsilon \otimes u^\varepsilon) + \frac{4 \sqrt{3} \kappa}{5} \nabla_x |u^\varepsilon|^2 \\
\ \ = - \frac{\kappa^2}{\nu_\ast} F^\varepsilon + H^\varepsilon + \nu_\ast J^\varepsilon + \mathcal{O}_1^\varepsilon, \\
\partial_t \theta^\varepsilon + \frac{\mu_5^\varepsilon}{c_1^\varepsilon \varepsilon} \nabla_x \cdot u^\varepsilon - \frac{97 \nu_\ast}{420} \Delta_x \theta^\varepsilon + \frac{97 \sqrt{3} \kappa}{105} \operatorname{div} (u^\varepsilon \theta^\varepsilon) = - \frac{\kappa^2}{\nu_\ast} G_\varepsilon + \mathcal{O}_{2,\varepsilon}
\end{array}
\right.
\end{eqnarray}
where $\mathcal{O}_{0,\varepsilon}, \mathcal{O}_1^\varepsilon, \mathcal{O}_{2,\varepsilon}$ are remainder terms that converge to zero in the sense of distributions as $\varepsilon \to 0$.
Rearrange the second equation system (\ref{Re:Conser}), we observe that
\begin{equation} \label{Boussi:ep}
\begin{split}
&\nabla_x (\mu_2^\varepsilon \rho^\varepsilon + \mu_1^\varepsilon \theta^\varepsilon) \\
&\ \ \approx \varepsilon \left\{ - \partial_t u^\varepsilon + \frac{\nu_\ast}{12} \Delta_x u^\varepsilon - \frac{\sqrt{3} \kappa}{3} \nabla_x \cdot (u^\varepsilon \otimes u^\varepsilon) - \frac{4 \sqrt{3} \kappa}{5} \nabla_x |u^\varepsilon|^2 - \frac{\kappa^2}{\nu_\ast} F^\varepsilon + H^\varepsilon + \nu_\ast J^\varepsilon + \mathcal{O}_1^\varepsilon \right\}.
\end{split}
\end{equation}
It can be easily shown that the right hand side of equation (\ref{Boussi:ep}) converges to zero in the sense of distributions as $\varepsilon \to 0$. Combining Corollary \ref{Conv:RUT} with the facts that $\mu_1^\varepsilon \to \frac{\sqrt{15}}{45}$ and $\mu_2^\varepsilon \to \frac{\sqrt{3}}{6}$ as $\varepsilon \to 0$, we obtain the Boussinesq relation
\begin{align} \label{Boussi:Eq}
\nabla_x (3 \sqrt{5} \rho + 2 \theta) = 0.
\end{align}

To get rid of the term $\nabla_x (3 \sqrt{5} \rho^\varepsilon + 2 \theta^\varepsilon)$ whose coefficient is a constant multiple of $\varepsilon^{-1}$ in the second equation of system (\ref{Re:Conser}), we consider the Helmholtz decomposition for $L^2(\mathbf{T}_x^3)^3$. For any $h \in L^2(\mathbf{T}_x^3)^3$, there exists a unique decomposition of the form $h = h_0 + \nabla_x \pi$ where 
\begin{align*}
h_0 \in L_\sigma^2(\mathbf{T}_x^3) &:= \{ f \in L^2(\mathbf{T}_x^3)^3 \bigm| \operatorname{div} f = 0 \; \; \text{in} \; \;  \mathbf{T}^3 \}, \\
\nabla_x \pi \in G^2(\mathbf{T}_x^3) &:= \{ \nabla_x \pi \in L^2(\mathbf{T}_x^3)^3 \bigm| \pi \in L^2(\mathbf{T}_x^3) \}.
\end{align*}
Moreover, the estimate
\begin{align} \label{L2Est:Helm}
\| h_0 \|_{L^2(\mathbf{T}_x^3)} + \| \nabla_x \pi \|_{L^2(\mathbf{T}_x^3)} \leq 2 \| h \|_{L^2(\mathbf{T}_x^3)}
\end{align}
holds. The Helmholtz projection, denoted by $\mathbb{P}$, is the projection that maps $h$ to $h_0$, i.e., we have that $\mathbb{P}(h) = h_0$. 
Let us recall that the Helmholtz projection on a torus is given by 
\begin{align} \label{L2HelmDe:T}
\mathbb{P}\big( h \big) (x) = \sum_{k \in \mathbf{Z}^3, \, k \neq 0} \left( \widehat{h}(k) - \frac{\widehat{h}(k) \cdot k}{|k|^2} k \right) \mathrm{e}^{2 \pi i k \cdot x}, \quad \widehat{h}(k) = \int_{\mathbf{T}^3} h(x) \mathrm{e}^{- 2 \pi i k \cdot x} \, dx
\end{align}
for $h \in L^2(\mathbf{T}_x^3)^3$.
Moreover, $\mathbb{P}$ has properties
\[
\mathbb{P}(f_0) = f_0 \quad \forall \; f_0 \in L_\sigma^2(\mathbf{T}_x^3) \quad \text{and} \quad \mathbb{P}(\nabla_x p) = 0 \quad \forall \; \nabla_x p \in G^2(\mathbf{T}_x^3).
\]
Suppose that $h \in H^1(\mathbf{T}_x^3)^3$ and $h = h_0 + \nabla_x \pi$ be the Helmholtz decomposition of $h$ in $L^2(\mathbf{T}_x^3)$. It can be easily observed from (\ref{L2HelmDe:T}) that the differentiation $\partial_x$ commutes with $\mathbb{P}$.
As a result, $\partial_x h = \partial_x h_0 + \nabla_x (\partial_x h)$ is indeed the Helmholtz decomposition of $\partial_x h$ in $L^2(\mathbf{T}_x^3)^3$.
By applying the Helmholtz projection $\mathbb{P}$ to both sides of the second equation in system (\ref{Re:Conser}), we obtain that
\begin{equation} \label{HeReinNS}
\begin{split}
\partial_t \mathbb{P} (u^\varepsilon) - \frac{\nu_\ast}{12} \Delta_x \mathbb{P} (u^\varepsilon) + \frac{\sqrt{3} \kappa}{3} \mathbb{P} \big( \nabla_x \cdot (u^\varepsilon \otimes u^\varepsilon) \big) = - \frac{\kappa^2}{\nu_\ast} \mathbb{P} (F^\varepsilon) + \mathbb{P}(H^\varepsilon) + \mathbb{P}(\nu_\ast J^\varepsilon) + \mathbb{P}(\mathcal{O}_1^\varepsilon).
\end{split}
\end{equation}

Since the sequence $\{ u^\varepsilon \}_\varepsilon$ is uniformly bounded in $L^\infty( [0,\infty); H^1(\mathbf{T}_x^3) )$ (see Corollary \ref{Conv:RUT}), by interchanging the differentiation $\partial_x$ with the Helmholtz projection $\mathbb{P}$, we observe by estimate (\ref{L2Est:Helm}) that $\{ \mathbb{P}(u^\varepsilon) \}_\varepsilon$ is also uniformly bounded in $L^\infty( [0,\infty); H^1(\mathbf{T}_x^3) )$.
Hence, by suppressing subsequences again, it can be concluded that there exists $\widetilde{u} \in L^\infty([0, \infty); H^1(\mathbf{T}_x^3) )$ such that $\mathbb{P}(u^\varepsilon) \overset{\ast}{\rightharpoonup} \widetilde{u}$, i.e., for any $\Psi \in L^1([0, \infty); H^{-1}(\mathbf{T}_x^3) )$, it holds that
\begin{align} \label{Cv:ut:A}
\int_0^\infty \int_{\mathbf{T}^3} \mathbb{P}(u^\varepsilon) \cdot \Psi \, dx \, dt \to \int_0^\infty \int_{\mathbf{T}^3} \widetilde{u} \cdot \Psi \, dx \, dt \quad \text{as} \quad \varepsilon \to 0.
\end{align}
On the other hand, for any $\Psi \in L^1([0, \infty); L^2(\mathbf{T}_x^3) )$, it holds that
\begin{align} \label{Cv:Pu:A}
\int_0^\infty \int_{\mathbf{T}^3} \mathbb{P}(u^\varepsilon - u) \cdot \Psi \, dx \, dt = \int_0^\infty \int_{\mathbf{T}^3} (u^\varepsilon - u) \cdot \mathbb{P}(\Psi) \, dx \, dt \to 0 \quad \text{as} \quad \varepsilon \to 0.
\end{align}
Combining convergence (\ref{Cv:ut:A}) and (\ref{Cv:Pu:A}), we deduce that
\begin{align*}
\int_0^\infty \int_{\mathbf{T}^3} \widetilde{u} \cdot \Psi \, dx \, dt = \int_0^\infty \int_{\mathbf{T}^3} \mathbb{P}(u) \cdot \Psi \, dx \, dt, \quad \forall \; \Psi \in L^1( [0, \infty); L^2(\mathbf{T}_x^3) ).
\end{align*}
Since there exists $\pi \in H^1(\mathbf{T}_x^3)$ such that $\mathbb{P}(u) = u - \nabla_x \pi$, we conclude by (\ref{divu:0}) that 
\begin{align} \label{Qu:0}
\int_0^\infty \int_{\mathbf{T}^3} (\widetilde{u} - u) \cdot \Psi \, dx \, dt = 0, \quad \forall \; \Psi \in L^\infty( [0,\infty); H^1(\mathbf{T}_x^3) ) \cap W^{1,1}( [0,\infty); H^1(\mathbf{T}_x^3) ).
\end{align}
\begin{lemma} \label{StCon:Pu}
Suppressing subsequences, $\mathbb{P}(u^\varepsilon) \to u$ strongly in $C( [0,\infty); L^2(\mathbf{T}_x^3) )$, i.e., it holds that
\begin{align*}
\| \mathbb{P}(u^\varepsilon) - u \|_{L^\infty( [0,\infty); L^2(\mathbf{T}_x^3) )} \to 0 \quad \text{as} \quad \varepsilon \to 0.
\end{align*}
\end{lemma}
\begin{proof}
We apply the Helmholtz decomposition $\mathbb{P}$ to the second equation of system (\ref{ruipAE}) to obtain that
\begin{align} \label{StCon:t:u}
\partial_t \mathbb{P}(u^\varepsilon) + \frac{c_1^\varepsilon}{\varepsilon} \mathbb{P}\Big( \nabla_x \cdot \big\langle A_\varepsilon \mathcal{L}^\varepsilon(f_\varepsilon) \big\rangle \Big) = - \frac{\kappa^2}{\nu_\ast} \mathbb{P}\big( \langle \mathrm{e}_1^\varepsilon f_\varepsilon^3 \rangle \big).
\end{align}
Integrating equation (\ref{StCon:t:u}) from $t_1$ to $t_2$ for some time interval $[t_1, t_2] \subset [0,\infty)$ and then take its inner product with $\mathbb{P}\big( u^\varepsilon (t_2) \big) - \mathbb{P}\big( u^\varepsilon (t_1) \big)$ in the sense of $L^2(\mathbf{T}_x^3)$ , we have that
\begin{equation} \label{StCon:eq}
\begin{split}
&\big\| \mathbb{P}\big( u^\varepsilon (t_2) \big) - \mathbb{P}\big( u^\varepsilon (t_1) \big) \big\|_{L^2(\mathbf{T}_x^3)}^2 \\
&\ \ = - \int_{t_1}^{t_2} \int_{\mathbf{T}^3} \frac{c_1^\varepsilon}{\varepsilon} \mathbb{P}\Big( \nabla_x \cdot \big\langle A_\varepsilon \mathcal{L}^\varepsilon(f_\varepsilon) \big\rangle \Big) \Big( \mathbb{P}\big( u^\varepsilon (t_2) \big) - \mathbb{P}\big( u^\varepsilon (t_1) \big) \Big) \, dx \, dt \\
&\ \ \ \ - \int_{t_1}^{t_2} \int_{\mathbf{T}^3} \frac{\kappa^2}{\nu_\ast} \mathbb{P}\big( \big\langle \mathrm{e}_1^\varepsilon f_\varepsilon^3 \big\rangle \big) \Big( \mathbb{P}\big( u^\varepsilon (t_2) \big) - \mathbb{P}\big( u^\varepsilon (t_1) \big) \Big) \, dx \, dt = (\mathrm{I}) + (\mathrm{II}).
\end{split}
\end{equation}
There exists $\delta_\ast >0$ such that for any $\varepsilon \in (0, \delta_\ast)$, it holds simultaneously that
\begin{equation} \label{sml:ep:1} 
\begin{split}
| \big\langle \mathrm{e}_1^\varepsilon \mathcal{P}^\varepsilon(f_\varepsilon)^3 \big\rangle_i | &\leq 2 |u_i^\varepsilon|^3 + 4 |u_i^\varepsilon| |u^\varepsilon|^2 + 4 (\rho^\varepsilon)^2 |u_i^\varepsilon| + 11 (\theta^\varepsilon)^2 |u_i^\varepsilon| + 6 |\rho^\varepsilon \theta^\varepsilon u_i^\varepsilon|, \quad \forall \; 1 \leq i \leq 3, \\
\| \mathrm{e}_i^\varepsilon \|_{L^2(\Omega_v)} &\leq \| \mathrm{e}_i \|_{L^2(\Omega_v)} + 1, \quad \forall \; 0 \leq i \leq 2, \\
\| A_\varepsilon \|_{L^2(\Omega_v)} &\leq \| A \|_{L^2(\Omega_v)} + 1. \\
\end{split}
\end{equation}
Using the fact that $\mathbb{P}$ is self-adjoint with respect to inner product $\langle \cdot, \cdot \rangle_{L^2(\mathbf{T}_x^3)}$, we deduce that
\begin{align*}
&\int_{\mathbf{T}^3} \big| \big\langle A_\varepsilon \cdot \nabla_x \mathcal{L}^\varepsilon(f_\varepsilon) \big\rangle \big| \cdot \big| \mathbb{P}\big( u^\varepsilon (t_2) \big) - \mathbb{P}\big( u^\varepsilon (t_1) \big) \big| \, dx \\
&\ \ \lesssim \| A_\varepsilon \|_{L^2(\Omega_v)} \| \mathcal{L}^\varepsilon(f_\varepsilon) \|_{H^1( \mathbf{T}_x^3; L^2(\Omega_v) )} \| u^\varepsilon \|_{L^\infty( [0,\infty); L^2(\mathbf{T}_x^3) )}.
\end{align*}
Hence, for $\varepsilon < \delta_\ast$, the first integral on the right hand side of equation (\ref{StCon:eq}) follows the estimate
\begin{align} \label{Es:AdGLf}
|(\mathrm{I})| \lesssim \frac{1}{\varepsilon} \left( \int_{t_1}^{t_2} \mathcal{D}_\varepsilon\big( f_\varepsilon \big)^2 (t) \, dt \right)^{\frac{1}{2}} \sqrt{t_2 - t_1} \| u^\varepsilon \|_{L^\infty( [0,\infty); L^2(\mathbf{T}_x^3) )}.
\end{align}
To estimate $(\mathrm{II})$, we decompose $\big\langle \mathrm{e}_1^\varepsilon f_\varepsilon^3 \big\rangle$ into the sum of $\big\langle \mathrm{e}_1^\varepsilon \mathcal{P}^\varepsilon(f_\varepsilon)^3 \big\rangle$ and $\big\langle \mathrm{e}_i^\varepsilon \mathcal{R}_{\mathcal{L}^\varepsilon, 3} \big\rangle$.
By Lemma \ref{PuCov:L3}, we have that
\begin{align} \label{Es:eiRL3}
\left| \int_{t_1}^{t_2} \int_{\mathbf{T}^3} \big\langle \mathrm{e}_i^\varepsilon \mathcal{R}_{\mathcal{L}^\varepsilon, 3} \big\rangle \Big( \mathbb{P}\big( u^\varepsilon (t_2) \big) - \mathbb{P}\big( u^\varepsilon (t_1) \big) \Big) \, dx \, dt \right| \lesssim \sqrt{\nu_\ast (t_2 - t_1)} \varepsilon^{1 - 18 \gamma} \| f_0 \|_X^4.
\end{align}
If $\varepsilon < \delta_\ast$, the first inequality of (\ref{sml:ep:1}) and Remark \ref{L2Bd:RUT} imply that
\begin{align} \label{Es:eiPf3}
\left| \int_{t_1}^{t_2} \int_{\mathbf{T}^3} \big\langle \mathrm{e}_1^\varepsilon \mathcal{P}^\varepsilon(f_\varepsilon)^3 \big\rangle \Big( \mathbb{P}\big( u^\varepsilon (t_2) \big) - \mathbb{P}\big( u^\varepsilon (t_1) \big) \Big) \, dx \, dt \right| \lesssim (t_2 - t_1) \| f_0 \|_X^4.
\end{align}
Therefore, this shows that $\{ \mathbb{P}(u^\varepsilon) \}_\varepsilon \subset C( [0,\infty); L^2(\mathbf{T}_x^3) )$ and $\big\{ \| \mathbb{P}\big( u^\varepsilon \big) (t) \|_{L^2(\mathbf{T}_x^3)} \big\}_\varepsilon$ is equi-continuous in time $t$. By the Arzel$\grave{\text{a}}$-Ascoli theorem, we obtain Lemma \ref{StCon:Pu}.
\end{proof}

\subsection{Convergence to the incompressible Navier-Stokes-Fourier limit}
\label{sub:cvNSF}

We further define projection $\mathbb{Q} := \mathbb{I} - \mathbb{P}$ where $\mathbb{I}$ denotes the identity projection and expand
\[
\mathbb{P}\big( \nabla_x \cdot (u^\varepsilon \otimes u^\varepsilon) \big) = \mathbb{P}\Big( \nabla_x \cdot \big( \mathbb{P}(u^\varepsilon) \otimes \mathbb{P}(u^\varepsilon) \big) \Big) + R_\mathbb{P}(u^\varepsilon)
\]
with
\begin{align*}
R_\mathbb{P}(u^\varepsilon) &= \mathbb{P} \Big( \nabla_x \cdot \big( \mathbb{P}(u^\varepsilon) \otimes \mathbb{Q}(u^\varepsilon) \big) \Big) + \mathbb{P} \Big( \nabla_x \cdot \big( \mathbb{Q}(u^\varepsilon) \otimes \mathbb{P}(u^\varepsilon) \big) \Big) + \mathbb{P} \Big( \nabla_x \cdot \big( \mathbb{Q}(u^\varepsilon) \otimes \mathbb{Q}(u^\varepsilon) \big) \Big).
\end{align*}
Let 
\[
C_\sigma^\infty(\mathbf{T}^3) := \{ f \in C^\infty(\mathbf{T}^3)^3 \bigm| \operatorname{div} f = 0 \; \; \text{in} \; \; \mathbf{T}^3 \}.
\]

\begin{lemma} \label{RemHel:0}
For any $\Psi \in C^\infty( [0,T]; C_\sigma^\infty(\mathbf{T}_x^3) )$ with $0<T<\infty$, it holds that
\begin{align*}
\left| \int_0^T \int_{\mathbf{T}^3} R_\mathbb{P}(u^\varepsilon) \cdot \Psi \, dx \, dt \right| \to 0 \quad \text{as} \quad \varepsilon \to 0.
\end{align*}
\end{lemma}
\begin{proof}
Let $\zeta^\varepsilon := \mu_2^\varepsilon \rho^\varepsilon + \mu_1^\varepsilon \theta^\varepsilon$.
We can deduce from system (\ref{ruipAE}) that
\begin{equation} \label{WeCov:QQ}
\left\{
 \begin{aligned}
 \varepsilon \partial_t \zeta^\varepsilon + \mu_4^\varepsilon \operatorname{div} \mathbb{Q}(u^\varepsilon) &=& &- \frac{\varepsilon \mu_2^\varepsilon \kappa^2}{\nu_\ast} \langle f_\varepsilon^3 \rangle - c_2^\varepsilon \mu_1^\varepsilon \nabla_x \cdot \big\langle B_\varepsilon \mathcal{L}^\varepsilon(f_\varepsilon) \big\rangle - \frac{\varepsilon \mu_1^\varepsilon \kappa^2}{\nu_\ast} \langle \mathrm{e}_2^\varepsilon f_\varepsilon^3 \rangle,& \\
 \varepsilon \partial_t \mathbb{Q}(u^\varepsilon) + \nabla_x \zeta^\varepsilon &=& &- c_1^\varepsilon \mathbb{Q}\Big( \nabla_x \cdot \big\langle A_\varepsilon \mathcal{L}^\varepsilon(f_\varepsilon) \big\rangle \Big) - \frac{\varepsilon \kappa^2}{\nu_\ast} \mathbb{Q}\big( \langle \mathrm{e}_1^\varepsilon f_\varepsilon^3 \rangle \big).&
 \end{aligned}
\right.
\end{equation}
By Lemma \ref{Re:BerL}, for any $i \in \{0, 1, 2\}$, we have that
\begin{equation} \label{L2E:eif3}
\begin{split}
\| \langle \mathrm{e}_i^\varepsilon f_\varepsilon^3 \rangle \|_{L^2(\mathbf{T}_x^3)} &\leq \| \mathrm{e}_i^\varepsilon \|_{L^2(\Omega_v)} \big\| \| f_\varepsilon^3 \|_{L^2(\Omega_v)} \big\|_{L^2(\mathbf{T}_x^3)} \\
&\lesssim \frac{1}{\varepsilon^{9 \gamma}} \| \mathrm{e}_i^\varepsilon \|_{L^2(\Omega_v)} \big\| \| f_\varepsilon \|_{L^1(\Omega_v)}^3 \big\|_{L^2(\mathbf{T}_x^3)} \leq \frac{1}{\varepsilon^{9 \gamma}} \| \mathrm{e}_i^\varepsilon \|_{L^2(\Omega_v)} \| f_0 \|_X^3.
\end{split}
\end{equation}
In addition, for $C \in L^2(\Omega_v)$ be either a matrix or a vector and $0<T<\infty$, it holds that
\begin{equation} \label{L2E:GCLf}
\begin{split}
\int_0^T \big\| \big\langle C^\mathrm{T} \cdot \nabla_x \mathcal{L}^\varepsilon(f_\varepsilon) \big\rangle \big\|_{L^2(\mathbf{T}_x^3)} \, dt &\leq T^{\frac{1}{2}} \| C \|_{L^2(\Omega_v)} \left( \int_0^T \| \mathcal{L}^\varepsilon(f_\varepsilon) \|_X^2 \, dt \right)^{\frac{1}{2}} \\
&\leq \varepsilon \sqrt{T \nu_\ast} \| C \|_{L^2(\Omega_v)} \| f_0 \|_X.
\end{split}
\end{equation}
If $\gamma < \frac{1}{9}$, by combining (\ref{L2E:eif3}) and (\ref{L2E:GCLf}), we can conclude that the right hand side of both equations of (\ref{WeCov:QQ}) converge to $0$ in $L_{loc}^1( [0,\infty); L^2(\mathbf{T}_x^3) )$ as $\varepsilon \to 0$.
Since both $\{ \zeta^\varepsilon \}_\varepsilon$ and $\{ \mathbb{Q}(u^\varepsilon) \}_\varepsilon$ are uniformly bounded in $L^\infty( [0,\infty); L^2(\mathbf{T}_x^3) )$, by a compensated compactness result due to Lions and Masmoudi \cite{LiMaI} (see also \cite[Theorem A.2]{GSR09}), it holds that
\[
\mathbb{P}\Big( \nabla_x \cdot \big( \mathbb{Q}(u^\varepsilon) \otimes \mathbb{Q}(u^\varepsilon) \big) \Big) \to 0
\]
in the sense of distributions as $\varepsilon \to 0$, i.e., for any $\Psi \in C^\infty( [0,T]; C_\sigma^\infty(\mathbf{T}_x^3) )$ with $0<T<\infty$,
\begin{align*}
\int_0^T \int_{\mathbf{T}^3} \mathbb{Q}(u^\varepsilon) \otimes \mathbb{Q}(u^\varepsilon) : \nabla_x \Psi \, dx \, dt \to 0 \quad \text{as} \quad \varepsilon \to 0.
\end{align*}
Combining with the facts that $\mathbb{P}(u^\varepsilon) \to u$ strongly in $L^\infty([0, \infty); L^2(\mathbf{T}_x^3) )$, $\{ \mathbb{Q}(u^\varepsilon) \}_\varepsilon$ is uniformly bounded in $L^\infty([0, \infty); L^2(\mathbf{T}_x^3) )$ and (\ref{Qu:0}), we obtain Lemma \ref{RemHel:0}.
\end{proof}

Combining Lemma \ref{StCon:Pu} with Lemma \ref{RemHel:0}, it is not hard to show that for $\Psi \in C^\infty( [0,T]; C_\sigma^\infty(\mathbf{T}_x^3) )$ with $0<T<\infty$,
\begin{align*}
&\int_0^T \int_{\mathbf{T}^3} \bigg\{ \partial_t \mathbb{P} (u^\varepsilon) - \frac{\nu_\ast}{12} \Delta_x \mathbb{P} (u^\varepsilon) + \frac{\sqrt{3} \kappa}{3} \mathbb{P} \big( \nabla_x \cdot (u^\varepsilon \otimes u^\varepsilon) \big) \bigg\} \cdot \Psi \, dx \, dt \\
&\ \ \to \int_{\mathbf{T}^3} u_0 \cdot \Psi(x,0) \, dx - \int_0^T \int_{\mathbf{T}^3} u \cdot \partial_t \Psi + \frac{\sqrt{3} \kappa}{3} (u \otimes u) : \nabla_x \Psi - \frac{\nu_\ast}{12} u \cdot \Delta_x \Psi \, dx \, dt
\end{align*}
as $\varepsilon \to 0$.
By Remark \ref{L2Bd:RUT}, we observe that $\{ F^\varepsilon \}_\varepsilon$ is uniformly bounded in $L^\infty( [0,\infty); L^2(\mathbf{T}_x^3) )$. Hence, by suppressing subsequences again, there exists $\mathcal{F}_0 \in L^\infty( [0,\infty); L^2(\mathbf{T}_x^3) )$ such that
\begin{align*}
- F^\varepsilon \overset{\ast}{\rightharpoonup} \mathcal{F}_0 \quad \text{in} \quad L^\infty( [0,\infty); L^2(\mathbf{T}_x^3) )
\end{align*}
as $\varepsilon \to 0$.
Furthermore, we note by H$\ddot{\text{o}}$lder's inequality that for any $1 \leq i \leq 3$,
\begin{align} \label{L32E:upu}
\| u_i^\varepsilon \partial_{x_i} u_i^\varepsilon \|_{L^{\frac{3}{2}}(\mathbf{T}_x^3)} \leq \| u^\varepsilon \|_{L^6(\mathbf{T}_x^3)} \| \nabla_x u^\varepsilon \|_{L^2(\mathbf{T}_x^3)} \leq \| f_0 \|_X^2,
\end{align}
i.e., $\{ H^\varepsilon \}_\varepsilon$ is uniformly bounded in $L^\infty( [0,\infty); L^{\frac{3}{2}}(\mathbf{T}_x^3) )$.
Thus analogously, by suppressing subsequences once more, there exists $\mathcal{H} \in L^\infty( [0,\infty); L^{\frac{3}{2}}(\mathbf{T}_x^3) )$ such that 
\begin{align*}
H^\varepsilon \overset{\ast}{\rightharpoonup} \mathcal{H} \quad \text{in} \quad L^\infty( [0,\infty); L^{\frac{3}{2}}(\mathbf{T}_x^3) )
\end{align*}
as $\varepsilon \to 0$.

\begin{remark} \label{Est:FH}
By Remark \ref{L2Bd:RUT} and estimate (\ref{L32E:upu}), the weak-$\ast$ convergence of $\{ F^\varepsilon \}_\varepsilon$ and $\{ H^\varepsilon \}_\varepsilon$ imply that
\begin{equation*}
 \begin{aligned}
 &\| \mathcal{F} \|_{L^\infty( [0,\infty); L^2(\mathbf{T}_x^3) )}& &\leq& &\underset{\varepsilon \to 0}{\operatorname{liminf}} \, \| F^\varepsilon \|_{L^\infty( [0,\infty); L^2(\mathbf{T}_x^3) )}& &\leq& \| f_0 \|_X^3, \\
 &\| \mathcal{H} \|_{L^\infty( [0,\infty); L^{\frac{3}{2}}(\mathbf{T}_x^3) )}& &\leq& &\underset{\varepsilon \to 0}{\operatorname{liminf}} \, \| H^\varepsilon \|_{L^\infty( [0,\infty); L^{\frac{3}{2}}(\mathbf{T}_x^3) )}& &\leq& \| f_0 \|_X^2.
 \end{aligned}
\end{equation*}
\end{remark}

Since the Helmholtz projection $\mathbb{P}$ is self-adjoint, for any $\Psi \in C^\infty( [0,T]; C_\sigma^\infty(\mathbf{T}_x^3) )$ with $0<T<\infty$, we have that
\begin{align*}
&\int_0^T \int_{\mathbf{T}^3} \left\{ - \frac{\kappa^2}{\nu_\ast} \mathbb{P} (F^\varepsilon) + \mathbb{P} (H^\varepsilon) + \mathbb{P}(\nu_\ast J^\varepsilon) + \mathbb{P} (\mathcal{O}_1^\varepsilon) \right\} \cdot \Psi \, dx \, dt \\
&\ \ \to \int_0^T \int_{\mathbf{T}^3} \left( \frac{\kappa^2}{\nu_\ast} \mathcal{F}_0 + \mathcal{H} + \nu_\ast \mathcal{J}_0 \right) \cdot \Psi \, dx \, dt
\end{align*}
as $\varepsilon \to 0$ where $\mathcal{J}_0 := -\frac{1}{10} (\partial_{x_1}^2 u_1^\varepsilon, \partial_{x_2}^2 u_2^\varepsilon, \partial_{x_3}^2 u_3^\varepsilon)$.
Apart from $\mathbb{P}(u^\varepsilon)$, we have another strong convergence.

\begin{lemma} \label{StCon:therh}
Suppressing subsequences, 
\[
\theta^\varepsilon - \mu_5^\varepsilon \rho^\varepsilon \to \theta - \frac{2 \sqrt{5}}{5} \rho \quad \text{strongly in} \quad C( [0,\infty]; L^2(\mathbf{T}_x^3) ),
\]
i.e., it holds that
\begin{align*}
\left\| (\theta^\varepsilon - \mu_5^\varepsilon \rho^\varepsilon) - \left( \theta - \frac{2 \sqrt{5}}{5} \rho \right) \right\|_{L^\infty( [0,\infty); L^2(\mathbf{T}_x^3) )} \to 0 \quad \text{as} \quad \varepsilon \to 0.
\end{align*}
\end{lemma}
\begin{proof}
The proof is basically the same as the proof of Lemma \ref{StCon:Pu}. 
Subtract $\mu_5^\varepsilon$ times the first equation from the third equation in system (\ref{ruipAE}), we obtain that
\begin{align} \label{StCon:eqdif}
\partial_t \big( \theta^\varepsilon - \mu_5^\varepsilon \rho^\varepsilon \big) + \frac{c_2^\varepsilon}{\varepsilon} \nabla_x \cdot \big\langle B_\varepsilon \mathcal{L}^\varepsilon(f_\varepsilon) \big\rangle = - \frac{\kappa^2}{\nu_\ast} \big\langle \mathrm{e}_2^\varepsilon f_\varepsilon^3 \big\rangle + \mu_5^\varepsilon \frac{\kappa^2}{\nu_\ast} \big\langle f_\varepsilon^3 \big\rangle.
\end{align}
Integrate equation (\ref{StCon:eqdif}) from $t_1$ to $t_2$ for some time interval $[t_1,t_2] \subset [0,\infty)$ and then take its inner product with 
\[
J_{t_2, t_1} := \big( \theta^\varepsilon(t_2) - \mu_5^\varepsilon \rho^\varepsilon(t_2) \big) - \big( \theta^\varepsilon(t_1) - \mu_5^\varepsilon \rho^\varepsilon(t_1) \big)
\]
in the sense of $L^2(\mathbf{T}_x^3)$, we have that
\begin{equation} \label{SCon:enerJ21}
\begin{split}
\| J_{t_2, t_1} \|_{L^2(\mathbf{T}_x^3)}^2 &= - \int_{t_1}^{t_2} \int_{\mathbf{T}^3} \frac{c_2^\varepsilon}{\varepsilon} \nabla_x \cdot \big\langle B_\varepsilon \mathcal{L}^\varepsilon(f_\varepsilon) \big\rangle J_{t_2, t_1} \, dx \, dt - \frac{\kappa^2}{\nu_\ast} \int_{t_1}^{t_2} \int_{\mathbf{T}^3} \big\langle \mathrm{e}_2^\varepsilon f_\varepsilon^3 \big\rangle J_{t_2, t_1} \, dx \, dt \\
&\ \ + \mu_5^\varepsilon \frac{\kappa^2}{\nu_\ast} \int_{t_1}^{t_2} \int_{\mathbf{T}^3} \big\langle f_\varepsilon^3 \big\rangle J_{t_2, t_1} \, dx \, dt = (\mathrm{I}) + (\mathrm{II}) + (\mathrm{III}).
\end{split}
\end{equation}
There exists $\zeta_\ast>0$ such that for any $\varepsilon \in (0, \zeta_\ast)$, it holds simultaneously that
\begin{equation*}
 \begin{aligned}
 | \big\langle \mathrm{e}_2^\varepsilon \mathcal{P}^\varepsilon(f_\varepsilon)^3 \big\rangle | &\leq& &25 |\theta^\varepsilon|^3 + 3 |\rho^\varepsilon| |u^\varepsilon|^2 + 3 |\theta^\varepsilon| |u^\varepsilon|^2 + 10 |\theta^\varepsilon| |\rho^\varepsilon|^2 + 6 |\rho^\varepsilon| |\theta^\varepsilon|^2,& \\
 | \big\langle \mathcal{P}^\varepsilon(f_\varepsilon)^3 \big\rangle | &\leq& &|\rho^\varepsilon|^3 + 2 |\theta^\varepsilon|^3 + 4 |\rho^\varepsilon| |u^\varepsilon|^2 + 3 |\theta^\varepsilon| |u^\varepsilon|^2 + 10 |\rho^\varepsilon| |\theta^\varepsilon|^2.&
 \end{aligned}
\end{equation*}
Then, by estimating the right hand side of equation (\ref{SCon:enerJ21}) analogously as in the proof of Lemma \ref{StCon:Pu}, we can show that if we restrict $\varepsilon < \mathrm{min} \, \{ \delta_\ast, \zeta_\ast \}$, then $\{ \theta^\varepsilon - \mu_5^\varepsilon \rho^\varepsilon \}_\varepsilon \subset C( [0,\infty); L^2(\mathbf{T}_x^3) )$ and $\{ \theta^\varepsilon - \mu_5^\varepsilon \rho^\varepsilon \}_\varepsilon$ is equi-continuous in time $t$. 
Indeed, $|(\mathrm{I})|$ follows estimate (\ref{Es:AdGLf}) and by Lemma \ref{PuCov:L3} and Remark \ref{L2Bd:RUT}, $|(\mathrm{II})|$ and $|(\mathrm{III})|$ can be estimated by (\ref{Es:eiRL3}) and (\ref{Es:eiPf3}).
By the Arzel$\grave{\text{a}}$-Ascoli theorem, we obtain Lemma \ref{StCon:therh}.
\end{proof}

Subtracting $\mu_5^\varepsilon$ times the first equation from the third equation in system (\ref{Re:Conser}), we obtain that
\begin{align} \label{thetaeq:ori}
\partial_t (\theta^\varepsilon - \mu_5^\varepsilon \rho^\varepsilon) - \frac{97 \nu_\ast}{420} \Delta_x \theta^\varepsilon + \frac{97 \sqrt{3} \kappa}{105} \operatorname{div} (u^\varepsilon \theta^\varepsilon) = - \frac{\kappa^2}{\nu_\ast} G_\varepsilon + \frac{\mu_5^\varepsilon \kappa^2}{\nu_\ast} E_\varepsilon + \mathcal{O}_{2,\varepsilon} - \mu_5^\varepsilon \mathcal{O}_{0,\varepsilon}.
\end{align}
Let $\widetilde{\theta^\varepsilon} := \theta^\varepsilon - \mu_5^\varepsilon \rho^\varepsilon$. We have that
\begin{align} \label{div:Qurho}
\operatorname{div} \big( \mathbb{Q}(u^\varepsilon) \rho^\varepsilon) = \frac{1}{\mu_2^\varepsilon + \mu_1^\varepsilon \mu_5^\varepsilon} \operatorname{div} \big( \mathbb{Q}(u^\varepsilon) \zeta^\varepsilon \big) - \frac{\mu_1^\varepsilon}{\mu_2^\varepsilon + \mu_1^\varepsilon \mu_5^\varepsilon} \operatorname{div} \big( \mathbb{Q}(u^\varepsilon) \widetilde{\theta^\varepsilon} \big)
\end{align}
and
\begin{align} \label{the:til}
\Big( 1 + \frac{2 \sqrt{5} \mu_5^\varepsilon}{15} \Big) \nabla_x \theta^\varepsilon = \nabla_x \widetilde{\theta^\varepsilon} + \frac{\sqrt{5} \mu_5^\varepsilon}{15} \mathcal{M}(\rho^\varepsilon, u^\varepsilon, \theta^\varepsilon)
\end{align}
where $\mathcal{M}(\rho^\varepsilon, u^\varepsilon, \theta^\varepsilon)$ is the right hand side of (\ref{Boussi:ep}).
Using (\ref{div:Qurho}) and (\ref{the:til}), equation (\ref{thetaeq:ori}) can be rewritten as
\begin{equation} \label{theta:eq}
\begin{split}
&\partial_t \widetilde{\theta^\varepsilon} - \frac{97 \nu_\ast \mu_6^\varepsilon}{420} \Delta_x \widetilde{\theta^\varepsilon} + \frac{97 \sqrt{3} \kappa}{105} \operatorname{div} (u^\varepsilon \widetilde{\theta^\varepsilon}) \\
&\ \ = \frac{\sqrt{5} \mu_5^\varepsilon \mu_6^\varepsilon}{15} \mathcal{M}(\rho^\varepsilon, u^\varepsilon, \theta^\varepsilon) - \frac{97 \sqrt{3} \kappa \mu_5^\varepsilon}{105 \mu_7^\varepsilon} \operatorname{div} \big( \mathbb{Q}(u^\varepsilon) \zeta^\varepsilon \big) + \frac{97 \sqrt{3} \kappa \mu_1^\varepsilon \mu_5^\varepsilon}{105 \mu_7^\varepsilon} \operatorname{div} \big( \mathbb{Q}(u^\varepsilon) \widetilde{\theta^\varepsilon} \big) \\
&\ \ \ \ - \frac{97 \sqrt{3} \kappa \mu_5^\varepsilon}{105} \operatorname{div} \big( \mathbb{P}(u^\varepsilon) \rho^\varepsilon \big) - \frac{\kappa^2}{\nu_\ast} G_\varepsilon + \frac{\mu_5^\varepsilon \kappa^2}{\nu_\ast} E_\varepsilon + \mathcal{O}_{2,\varepsilon} - \mu_5^\varepsilon \mathcal{O}_{0,\varepsilon}
\end{split}
\end{equation}
where $\mu_6^\varepsilon := \big( 1 + \frac{2 \sqrt{5} \mu_5^\varepsilon}{15} \big)^{-1}$ and $\mu_7^\varepsilon := \mu_2^\varepsilon + \mu_1^\varepsilon \mu_5^\varepsilon$.
Since $u^\varepsilon \overset{\ast}{\rightharpoonup} u$ in $L^\infty( [0,\infty); H^1(\mathbf{T}_x^3) )$ and $\widetilde{\theta^\varepsilon} \to \widetilde{\theta}$ in $C( [0,\infty); L^2(\mathbf{T}_x^3) )$ where $\widetilde{\theta} := \theta - \frac{2 \sqrt{5}}{5} \rho$, for any $\Phi \in C^\infty( [0,T] \times \mathbf{T}^3 )$ with $0<T<\infty$, it is easy to deduce that
\begin{align*}
&\int_0^T \int_{\mathbf{T}^3} \left\{ \partial_t \widetilde{\theta^\varepsilon} - \frac{97 \nu_\ast \mu_6^\varepsilon}{420} \Delta_x \widetilde{\theta^\varepsilon} + \frac{97 \sqrt{3} \kappa}{105} \operatorname{div} (u^\varepsilon \widetilde{\theta^\varepsilon}) \right\} \Phi \, dx \, dt \\
&\ \ \to \int_{\mathbf{T}^3} \left( \theta_0 - \frac{2 \sqrt{5}}{5} \rho_0 \right) \Phi(0,x) \, dx - \int_0^T \int_{\mathbf{T}^3} \widetilde{\theta} \partial_t \Phi \, dx \, dt \\
&\ \ \ \ - \frac{97 \nu_\ast}{532} \int_0^T \int_{\mathbf{T}^3} \widetilde{\theta} \Delta_x \Phi \, dx \, dt - \frac{97 \sqrt{3} \kappa}{105} \int_0^T \int_{\mathbf{T}^3} (\widetilde{\theta} u) \cdot \nabla_x \Phi \, dx \, dt
\end{align*}
as $\varepsilon \to 0$.

By considering the compensated compactness result by Lions and Masmoudi \cite{LiMaI} (see also \cite[Theorem A.2]{GSR09}) for system (\ref{WeCov:QQ}), we also know that $\operatorname{div} \big( \mathbb{Q}(u^\varepsilon) \zeta^\varepsilon   \big)$ converges to $0$ in the sense of distributions as $\varepsilon \to 0$.
In addition, since $\widetilde{\theta^\varepsilon} \to \widetilde{\theta}$ strongly in $C( [0,\infty); L^2(\mathbf{T}_x^3) )$ and $\mathbb{Q}(u^\varepsilon) \overset{\ast}{\rightharpoonup} 0$ in $L^\infty( [0,\infty); H^1(\mathbf{T}_x^3) )$, we can deduce that $\operatorname{div} \big( \mathbb{Q}(u^\varepsilon) \widetilde{\theta^\varepsilon} \big)$ also converges to $0$ in the sense of distributions as $\varepsilon \to 0$.
By Remark \ref{L2Bd:RUT} and further suppressing subsequences, there exist $\mathcal{G}_0, \mathcal{E}_0 \in L^\infty( [0,\infty): L^2(\mathbf{T}_x^3) )$ such that 
\begin{equation*}
- G_\varepsilon \overset{\ast}{\rightharpoonup} \mathcal{G}_0 \quad \text{and} \quad E_\varepsilon \overset{\ast}{\rightharpoonup} \mathcal{E}_0 \quad \text{in} \quad L^\infty( [0,\infty); L^2(\mathbf{T}_x^3) )
\end{equation*}
as $\varepsilon \to 0$.
Since $\operatorname{div} \big( \mathbb{P}(u^\varepsilon) \rho^\varepsilon \big)$ can be estimated in exactly the same way as (\ref{L32E:upu}), there exists $\mathcal{K}_0 \in L^\infty( [0,\infty); L^{\frac{3}{2}}(\mathbf{T}_x^3) )$ such that 
\begin{align*}
- \operatorname{div} \big( \mathbb{P}(u^\varepsilon) \rho^\varepsilon \big) \overset{\ast}{\rightharpoonup} \mathcal{K}_0 \quad \text{in} \quad L^\infty( [0,\infty); L^{\frac{3}{2}}(\mathbf{T}_x^3) )
\end{align*}
as $\varepsilon \to 0$.
Therefore, for any $\Phi \in C^\infty( [0,T] \times \mathbf{T}^3 )$ with $0<T<\infty$, it holds that
\begin{align*}
\int_0^T \int_{\mathbf{T}^3} \Pi(\rho^\varepsilon, u^\varepsilon, \theta^\varepsilon) \Phi \, dx \,dt \to \int_0^T \int_{\mathbf{T}^3} \left( \frac{194 \kappa \sqrt{15}}{525} \mathcal{K}_0 + \frac{\kappa^2}{\nu_\ast} \mathcal{G}_0 + \frac{2 \sqrt{5} \kappa^2}{5 \nu_\ast} \mathcal{E}_0 \right) \Phi \, dx \, dt
\end{align*}
as $\varepsilon \to 0$ where $\Pi(\rho^\varepsilon, u^\varepsilon, \theta^\varepsilon)$ is defined to be the right hand side of (\ref{theta:eq}).

Finally, we set $\nu := \frac{\nu_\ast}{12}$ and $\kappa = \sqrt{3}$. 
Summarizing all convergence results that we have derived in this paper, we show that $(u, \widetilde{\theta})$ satisfy the Navier-Stokes-Fourier type system
\begin{equation} \label{LNSF:ori}
\left\{
 \begin{aligned}
 \partial_t u - \nu \Delta_x u + u \cdot \nabla_x u + \nabla_x p &=& &\frac{1}{4 \nu} \mathcal{F}_0 + \mathcal{H} + 12 \nu \mathcal{J}_0,& \\ 
 \nabla_x \cdot u &=& &0,& \\
 \partial_t \widetilde{\theta} - \frac{291}{133} \nu \Delta_x \widetilde{\theta} + \frac{97}{35} u \cdot \nabla_x \widetilde{\theta} &=& &\frac{194 \sqrt{5}}{175} \mathcal{K}_0 + \frac{1}{4 \nu} \mathcal{G}_0 + \frac{\sqrt{5}}{10 \nu} \mathcal{E}_0& \\
 \end{aligned}
\right.
\end{equation} 
weakly globally.
As for the cutoff constant $\gamma$ in the definition of $\Lambda_\varepsilon$, it is sufficient to pick $\gamma < \frac{1}{18}$ for the whole theory to work.

\section*{Acknowledgement}
We would like to thank Professor Yoshikazu Giga and Professor Shota Sakamoto for many helpful advices regarding this paper.
The research of Tsuyoshi Yoneda was partly supported by the JSPS Grants-in-Aid for Scientific Research 20H01819.


\end{document}